\documentclass[10pt]{article}

\usepackage[toc,page]{appendix}
\usepackage[english]{babel}
\usepackage{amsmath}
\usepackage{amssymb,amsthm,amscd}
 \usepackage[pdftex]{color}
\usepackage[latin1]{inputenc}
\usepackage[textwidth=18cm,textheight=24cm]{geometry}
\usepackage{mathrsfs}
\usepackage[thinlines]{easybmat}
\usepackage[pdftex]{hyperref}

\newcommand{\diag}{\operatorname{diag}}

\newcommand{\C}{\mathbb{C}}

\renewcommand{\H}{\mathcal{H}}

\newcommand{\R}{\mathbb{R}}

\newcommand{\Z}{\mathbb{Z}}
\newcommand{\N}{\mathbb{N}}

\newcommand{\g}{\mathfrak{g}}

\newcommand{\A}{\mathscr A}

\newcommand{\I}{\mathbb{I}}
\renewcommand{\d}{\mathrm{d}}

\newcommand{\Exp}[1]{\operatorname{e}^{#1}}

\newcommand{\rI}{\operatorname{I}}
\newcommand{\rII}{\operatorname{II}}

\newcommand{\Cs}{\mathscr C}
\newcommand{\Ds}{\mathscr D}
\newcommand{\Es}{\mathscr E}

\newtheorem{pro}{Proposition}
\newtheorem{lemma}{Lemma}
\newtheorem{definition}{Definition}
\newtheorem{theorem}{Theorem}
\newtheorem{cor}{Corollary}

\begin{document}

\title{Multiple orthogonal polynomials of mixed type:\\Gauss-Borel factorization and the \\multi-component 2D Toda hierarchy }
\author{Carlos \'{A}lvarez-Fern\'{a}ndez$^{1,3,\dag}$, Ulises Fidalgo Prieto$^{2,\ddag}$ and Manuel Ma\~{n}as$^{1,\maltese}$\\ \\
$^1$Departamento de F\'{\i}sica Te\'{o}rica II, Universidad Complutense, \\
28040-Madrid, Spain\\
$^2$Departamento de Matem\'{a}tica Aplicada, Universidad Carlos III\\ 28911-Madrid, Spain \\
$^3$Departamento de M\'{e}todos Cuantitativos, Universidad Pontificia Comillas, \\ 28015-Madrid, Spain }
\date{$^\dag$calvarez@cee.upcomillas.es,  $^\ddag$ufidalgo@math.uc3m.es, $^{\maltese}$manuel.manas@fis.ucm.es}
\maketitle
\begin{abstract}
Multiple orthogonality is considered in the realm of a Gauss--Borel factorization problem for a semi-infinite moment matrix. Perfect combinations of weights and a finite Borel measure are constructed in terms of M-Nikishin systems. These perfect combinations ensure that the problem of mixed multiple orthogonality has a unique solution, that can be obtained from the solution of a Gauss--Borel factorization problem for a semi-infinite matrix, which plays the role of a moment matrix. This leads to sequences of multiple orthogonal polynomials, their duals and second kind functions. It also gives the corresponding linear forms that are bi-orthogonal to the dual linear forms. Expressions for these objects in terms of determinants from the moment matrix are given, recursion relations are found, which imply a multi-diagonal Jacobi type matrix with \emph{snake} shape,  and results like the ABC theorem or the Christoffel--Darboux formula are re-derived in this context (using the factorization problem and the generalized Hankel symmetry of the moment matrix). The connection between this description of multiple orthogonality and the multi-component 2D Toda hierarchy, which can be also understood and studied through a Gauss--Borel factorization problem, is discussed. Deformations of the weights, natural for M-Nikishin systems, are considered and the correspondence with solutions to the integrable hierarchy, represented as a collection of Lax equations, is explored. Corresponding Lax and Zakharov--Shabat matrices as well as wave functions and their adjoints are determined. The construction of discrete flows is discussed in terms of Miwa transformations which involve Darboux transformations for the multiple orthogonality conditions. The bilinear equations are derived and the $\tau$-function representation of the multiple orthogonality is given.
\end{abstract}

\newpage

\tableofcontents
\section{Introduction}

The topic of multiple orthogonality of polynomials  is very close  to that of  simultaneous  rational approximation (simultaneous Pad\'{e} aproximants) of systems of Cauchy transforms of measures. The history of simultaneous rational approximation starts in 1873 with the well known article \cite{He} in which Ch. Hermite proved the transcendence of the Euler number e.  Later, around the years 1934-35, K. Mahler delivered  at the University of Groningen several lectures \cite{Ma} where he settled down the  foundations of this theory. Meanwhile, two of Malher's students, J. Coates and H. Jager, made important contributions in this respect (see \cite{Co} and \cite{Ja}). In the case of Cauchy transforms, the simultaneous rational approximation definition may be written  in terms of multiple orthogonality of polynomials as follows. Given an interval $\Delta \subset \R$ of the real line, let ${\mathcal{M}}(\Delta)$ denote all the finite Borel measures which have support, $\mbox{supp}(\cdot)$ with infinitely many points in $\Delta,$ where they do not change sign. Fix $\mu \in {\mathcal{M}}(\Delta)$, and let us consider  a system of weights $\vec w=(w_1,\ldots,w_p)$ on $\Delta$, with $p \in {\mathbb{N}}$. (In this paper a``weight" on an interval $\Delta$ is meant to be  a real integrable function defined on $\Delta$ which does not change its sign on $\Delta$.) Fix a multi-index $\vec \nu=(\nu_1,\ldots, \nu_p) \in {\mathbb{Z}}_+^p,$ ${\mathbb{Z}}_+=\{0,1,2,\ldots\}$, and denote $|\vec \nu|=\nu_1+\cdots+\nu_p$. There exist polynomials, $A_1,\ldots, A_p$,  not all identically equal to zero which satisfy the following orthogonality relations
\begin{align}\label{tipoI}
\int_{\Delta}  x^{j} \sum_{a=1}^{p} A_a(x)w_{a} (x)\d\mu (x)&=0,  & \deg
 A_{a}&\leq\nu_{a}-1,&   j&=0,\ldots, |\vec \nu|-2.
\end{align}
Analogously, there exists a polynomial $B$ not identically equal to zero, such that
\begin{align}\label{tipoII}
\int_{\Delta} x^{j} B (x) w_{b} (x) \d\mu (x)&=0, & \deg B &\leq|\vec \nu|,&  j&=0,\ldots, \nu_b-1, \quad b=1,\ldots,p.
\end{align}

The resulting polynomials are said to be of type I and type II, respectively, with respect to the combination $(\mu, \vec w, \vec \nu)$ of the measure $\mu$, the systems of weights $\vec w$ and the multi-index $\vec \nu.$ When $p=1$ both definitions coincide with that of the standard orthogonal polynomials on the real line. The existence of a system of polynomials  $(A_1,\ldots, A_p)$ and a polynomial $B$ defined from (\ref{tipoI}) and (\ref{tipoII}) respectively, are ensured because in both cases finding the coefficients of the polynomials is  equivalent to solving a system of $|\vec \nu|$ linear homogeneous equations with $|\vec \nu|+1$ unknown coefficients. From the  theory of orthogonal polynomials we know that when $p=1$ each polynomial $A_1 \equiv  B$ has exactly degree $|\vec \nu|=\nu_1$; unfortunately if $p>1$ that is not true in general. For instance, let us take a system of weights $\vec w=(w_1,w_1,\dots,w_1)$, in this case the solution vector space has dimension bigger than one, and we can find two solutions which are linearly independent. Hence, there is at least an $a \in \{1,\ldots,p\}$ such that $\deg A_a < \nu_a-1$  and $\deg B< |\vec \nu|$. Given a measure $\mu \in {\mathcal{M}}(\Delta)$ and a system of weights $\vec w$ on $\Delta$ a multi-index $\vec \nu$ is called type I or type II normal if $\deg A_a$ must equal to $\nu_a-1,$ $a=1,\ldots,p$, or $\deg B$ must equal to $|\vec \nu|-1$, respectively. When for a pair $(\mu, \vec w)$ all the multi-indices are type I or type II normal, then the pair is called type I perfect or type II perfect, respectively. The concepts of normality and perfectness were introduced by  Malher (see Malher's, Coates' and Jager's articles cited above).

Multiple orthogonal of polynomials have been employed in several proofs of irrationality of numbers. For example, in \cite{Be}, F. Beukers shows that Apery's proof (see \cite{Ap}) of the irrationality of $\zeta(3)$ can be placed in the context of a combination of type I and type II multiple orthogonality, which is called mixed type multiple orthogonality of polynomials. More recently, mixed type approximation has appeared in random matrix and non-intersecting Brownian motion theories (see, for example, \cite{BleKui}, \cite{daems-kuijlaars} and \cite{arno}). A formalization of this kind of orthogonality was initiated by V. N. Sorokin \cite{So}. He studied a simultaneous rational approximation construction which is closely connected with multiple orthogonal  polynomials of mixed type. Surprisingly, in  \cite{fokas} a Riemann--Hilbert problem was found for the theory of orthogonal polynomials, and later  \cite{assche}  this result was largely extended to type I and II multiple orthogonality. In \cite{daems-kuijlaars} mixed type multiple orthogonality was analyzed from this perspective.

In order to introduce  multiple orthogonal polynomials of mixed type we need two systems of weights  $\vec w_1=(w_{1,1},\dots,w_{1,p_1})$  and $\vec w_2=(w_{2,1},\dots,w_{2,p_2})$ where  $p_1,p_2\in\N$,  (as we said a set of functions which do not change their sign in $\Delta$),  and two multi-indices $\vec \nu_1=(\nu_{1,1},\dots,\nu_{1,p_1})\in{\mathbb{Z}}_+^{p_1}$ and $\vec \nu_2=(\nu_{2,1},\dots,\nu_{2,p_2})\in {\mathbb{Z}}_+^{p_2}$ with $|\vec \nu_1|=|\vec \nu_2|+1$. There exist polynomials $A_1,\ldots, A_{p_1},$ not all identically zero, such that $\deg A_s < \nu_{1,s}$ which satisfy the following relations
\begin{equation}\label{orth}
\int_{\Delta} \sum_{a=1}^{p_1} A_a(x)w_{1,a} (x)w_{2,b}(x)x^{j} d\mu (x)=0, \quad  j=0,\ldots, \nu_{2,b}-1,\quad b=1,\ldots,p_2.
\end{equation}
They are called mixed multiple-orthogonal polynomials with respect to the combination $(\mu,\vec w_1,\vec w_2,\vec \nu_1,\vec\nu_2)$ of the measure $\mu,$ the systems of weights $\vec w_{1}$ and $\vec w_2$ and the multi-indices $\vec \nu_1$ and $\vec \nu_2.$  It is easy to show that finding the polynomials $A_1,\ldots, A_{p_1}$ is equivalent to solving  a system of $|\vec \nu_2|$ homogeneous linear equations  for the $|\vec \nu_1|$ unknown coefficients of the polynomials. Since $|\vec \nu_1|=|\vec \nu_2|+1$  the system always has a nontrivial solution.   The matrix of this system of equations is the so called moment matrix, and the study of its Gauss--Borel factorization will be the cornerstone of this paper. Observe that when $p_1=1$ we are in the type II case and if $p_2=1$ in  type I case. Hence in general we can find a solution of  \eqref{orth} where there is an $a \in \{1,\ldots,p_1\}$ such that $\deg A_a< \nu_{1,a}-1.$ When given a combination $(\mu,\vec w_1,\vec w_2)$ of a measure $\mu \in {\mathcal{M}}(\Delta)$ and systems of weights $\vec w_1$ and $\vec w_2$ on $\Delta$ if for each pair of multi-indices $(\vec \nu_1,\vec \nu_2)$ the conditions \eqref{orth} determine that $\deg A_a=\nu_{1,a}-1,$ $a=1, \ldots, p_1$, then we say that the combination $(\mu,\vec w_1,\vec w_2)$ is perfect. The concept of \emph{perfectness} will be rigorously introduced in Definition \ref{pag8}.

The seminal paper of M. Sato \cite{sato}, and further developments performed  by the Kyoto school through the use of the bilinear equation and the $\tau$-function formalism \cite{date1}-\cite{date3}, settled the basis for the Lie group theoretical description of integrable hierarchies, in this direction we have the relevant contribution by M. Mulase \cite{mulase} in which the factorization problems, dressing procedure, and linear systems were the key for
integrability. In this dressing setting the multicomponent integrable hierarchies of Toda type were analyzed in depth by K. Ueno and T. Takasaki \cite{ueno-takasaki}. See also the papers \cite{bergvelt} and \cite{kac} on the multi-component KP hierarchy and \cite{manas-martinez-alvarez} on the multi-component Toda lattice hierarchy. In a series of papers M. Adler and P. van Moerbeke showed how the Gauss--Borel factorization problem appears in the theory of the 2D Toda hierarchy and what they called the discrete KP hierarchy \cite{adler}-\cite{adler-van moerbeke 2}. In these papers it becomes clear --from a group-theoretical setup-- why standard orthogonality of polynomials and integrability of nonlinear equations of Toda type where so close. In fact, the Gauss--Borel factorization of the moment matrix may be understood as the Gauss--Borel factorization of the initial condition for the integrable hierarchy. To see the connection between the work of Mulase and that of Adler and van Moerbeke see \cite{felipe}. Later on, in the recent paper \cite{adler-vanmoerbeke 5}, it is shown that the multiple orthogonal construction described in previous paragraphs was linked with the multi-component KP hierarchy. In fact, for a given set of weights $(\vec w_1,\vec w_2)$ and degrees $(\vec \nu_1,\vec \nu_2)$ the authors constructed a finite matrix that plays the role of the moment matrix and, using the Riemann-Hilbert problem of \cite{daems-kuijlaars}, where able to show that determinants constructed from the moment matrix were $\tau$-functions solving the bilinear equation for the multi-component KP hierarchy. However, there is no mention in that paper to any Gauss--Borel factorization in spite of being the multicomponent integrable hierarchies connected with different factorization problems of these type. For further developments on the Gauss--Borel factorization and multi-component 2D Toda hierarchy see \cite{cum} and \cite{manas-martinez}.

This motivated our initial research in relation with this paper; i.e., the construction of an appropriate Gauss--Borel factorization in the group of semi-infinite matrices leading to multiple orthogonality and integrability in a simultaneous manner.  The main advantage of this approach lies in the application of different techniques based on the factorization problem used frequently in the theory of integrable systems. The key finding of this paper is, therefore, the characterization of a semi-infinite moment matrix whose Gauss--Borel factorization leads directly to multiple orthogonality. This makes sense when factorization can be performed, which is the case for perfect combinations $(\mu,\vec w_1,\vec w_2)$, which allows us to consider some sets of multiple orthogonal polynomials (called ladders) very much in the same manner as in the (non multiple) orthogonal polynomial setting. The Gauss--Borel factorization of this moment matrix leads, when one takes into account the Hankel type symmetry of the moment matrix, to results like: 1. Recursion relations, 2. ABC theorems and  3. Christoffel--Darboux formulas.  The first two are new results while the third is not new, as it was derived from the Riemann--Hilbert problem in \cite{daems-kuijlaars}. However, our derivation of the Christoffel--Darboux formula is based exclusively on the Gauss--Borel factorization, and its uniqueness and  existence for the multiple orthogonality problem  are the only requirements. Thus, it is sufficient to have a perfect combination $(\mu,\vec w_1,\vec w_2)$, and there are examples of this type which do not have a well defined Riemann--Hilbert problem in the spirit of \cite{daems-kuijlaars}.

When we seek for the appropriate  integrable hierarchy linked with multiple orthogonality we are lead to the multicomponent 2D Toda lattice hierarchy which extends the construction of the multicomponent KP hierarchy considered by M. J. Bergvelt and A. P. E. ten Kroode in \cite{bergvelt}; not to the multicomponent 2D Toda lattice hierarchy as described in \cite{ueno-takasaki} or \cite{manas-martinez-alvarez}.  In the spirit of this last mentioned articles, and complementing the continuous flows of the integrable hierarchy, we also introduce discrete flows, that could be viewed as Darboux transformations, and which correspond to Miwa transformations implying the addition of a zero/pole to the set of weights. Moreover, the Hankel type symmetry is related to an invariance under a number of flows, and to string equations. Bilinear equations can be derived from the Gauss--Borel factorization problem and moreover the $\tau$-function representation is available leading to a bridge to the results of \cite{adler-vanmoerbeke 5} in which no semi-infinite matrix or Gauss--Borel factorization was used.

This paper is divided into three sections,  \S 1 is this introduction which contains \S \ref{standard} in where we review the application of the $LU$ factorization of the moment matrix to the theory of orthogonal polynomials in the real line. Next, \S 2 is devoted to the presentation of the moment matrix and the discussion of the Gauss--Borel factorization. In this form we obtain perfect systems in terms of Nikishin systems, determinantal  expressions for the multiple orthogonal polynomials, their duals and second type functions, bi-orthogonality for the associated linear forms, recursion relations, ABC type theorems and the Christoffel--Darboux formula. Flows and the integrable hierarchy are studied in \S 3 in which an integrable hierarchy  \emph{a la} Bergvelt-ten Kroode is linked with the multiple orthogonality problem. We not only derive from the Gauss--Borel factorization the Lax and Zakarov--Shabat equations, but also we introduce discrete integrable flows, described by Miwa shifts, or Darboux transformations, and also construct an appropriate bilinear equation. Finally, we find the $\tau$ functions corresponding to the multiple orthogonality and link them to those of \cite{adler-vanmoerbeke 5}. At the end of the paper, we have added  two appendices: the first one collects the more technical proofs of the results  in this paper. In Appendix \ref{III} we consider discrete flows for the case of a measure $\mu$ with unbounded support $\operatorname{supp}\mu$.

\subsection{The Gauss--Borel factorization of the moment matrix and orthogonal polynomials}\label{standard}

Here we discuss how  the $LU$ factorization of the standard moment matrix $g=(\int x^{i+j}\d\mu)$ of a constant sign  finite Borel measure $\mu$ leads to traditional results
in the theory of orthogonal polynomials, namely recursion relation and Christoffel--Darboux formula. In spite that these results are well established we repeat them here because in their derivation is encoded the set of arguments we will use in  the multiple orthogonality setting. In the forthcoming exposition it will become clear the $LU$ factorization approach is just
 a compact way of using the orthogonality relations.

 The moment matrix can be written as the following  Grammian matrix
\begin{align*}
g=\int \chi(x)\chi(x)^\top\d\mu(x)
\end{align*}
  in terms of the monomial string $\chi(x):=(1,x,x^2,\dots)^\top$.

The Borel--Gauss factorization of $g$ is
\begin{align*}
  g&=S^{-1}\bar S, & S&=\left(\begin{smallmatrix}
    1&0&0&\cdots \\
    S_{1,0}&1&0&\cdots\\
    S_{2,0}&S_{2,1}&1&\cdots\\
    \vdots&\vdots&\vdots&\ddots
    \end{smallmatrix}\right), &
   \bar S^{-1}&=\left(\begin{smallmatrix}
     \bar S_{0,0}'&\bar S_{0,1}'&\bar S_{0,2}'&\cdots\\
     0&\bar S_{1,1}'&\bar S_{1,2}'&\cdots\\
     0&0&\bar S_{2,2}'&\cdots&\\
     \vdots&\vdots&\vdots&\ddots
   \end{smallmatrix}\right).
\end{align*}
The reader should notice that
\begin{itemize}
\item It makes sense whenever the truncated moment matrix $g^{[l]}=(g_{i,j})_{0\leq i,j<l }$ is an invertible matrix for any $l=1,2,\dots.$ If the factorization exists it is unique.
\item Although the truncated matrices $g^{[l]}$ are invertible it can be shown that $g$ itself is not invertible.
\item The matrix product of $S^{-1}$ with $\bar S$ involves only finite sums, but if we reverse the order of the factors we get series (with an infinite number of summands).
\end{itemize}
Given the factors $S$ and $\bar S$ we consider the following  polynomial strings, the semi-infinite vectors,
\begin{align*}
  P&:=S\chi=(P_0,P_1,\dots)^\top,&\bar P&:=(\bar S^{-1})^\top\chi=(\bar P_0,\bar P_1,\dots)^\top.
\end{align*}
The families of polynomials $\{P_l\}_{l=0}^\infty$ and $\{\bar P_k\}_{k=0}^\infty$ are biorthogonal:
\begin{align*}
  \int P(x) \bar P(x)^\top\d\mu(x)&=\int S\chi(x)\chi(x)^\top \bar S^{-1}\d\mu(x)=S\int\chi(x)\chi(x)^\top\d\mu(x)\bar S^{-1}\\&=\I
  \Rightarrow \int P_l(x)\bar P_k(x)\d\mu(x)=\delta_{l,k}.
\end{align*}
In this simple proof relies the basic connection between orthogonality and the $LU$ factorization, which we consider as the very same thing dressed in different manners.
 From the above orthogonality we conclude that
 \begin{align}\label{ortho-standard}
 \begin{aligned}
   \int P_l(x) x^j\d \mu(x)&=0, &j&=0,\dots,l-1,\\
      \int \bar P_l(x) x^j\d \mu(x)&=0,& j&=0,\dots,l-1,
 \end{aligned}
 \end{align}
 and we also have that  $P_l(x)$ and $\bar P_l$ are $l$-th degree polynomials where $P_l$  is monic and $\bar P_l$ satisfies $\int x^l \bar P_l(x)\d\mu(x)=1$, i.e. we have type II and type I normalizations.
 Given that the moment matrix is symmetric, $g=g^\top$ and the uniqueness of the $LU$ factorization we deduce that $\bar S=H(S^{-1})^\top$, with $H=\text{diag}(h_0,h_1,\dots)$; i.e., $g=S^{-1}H (S^{-1})^\top$ and the factorization is a Cholesky factorization (but this does not extend to the multiple orthogonal case). Therefore $\bar P_l=h_l^{-1} P_l$ so that
$\int P_l(x) P_k(x)\d\mu(x)=h_l \delta_{l,k}$, and $\{P_l\}_{l=0}^\infty$ is a family of monic orthogonal polynomials with respect to the measure $\mu$.

Considering the orthogonality relations as a linear system for the coefficients of the polynomials one concludes that polynomials and their duals can be expressed as
    \begin{align*}
  P_l&=\chi^{(l)}-\begin{pmatrix}
    g_{l,0}&g_{l,1}&\cdots&g_{l,l-1}
  \end{pmatrix}(g^{[l]})^{-1}\chi^{[l]}\\
  &=\bar S_{l,l}\begin{pmatrix}
    0 &0 &\cdots &0& 1
  \end{pmatrix}
  (g^{[l+1]})^{-1}\chi^{[l+1]}
  \\&=\frac{1}{\det g^{[l]}}\det
  \left(\begin{array}{cccc|c}
    g_{0,0}&g_{0,1}&\cdots&g_{0,l-1}&1\\
     g_{1,0}&g_{1,1}&\cdots&g_{1,l-1}&x\\
     \vdots &\vdots&            &\vdots&\vdots\\
        g_{l-1,0}&g_{l-1,1}&\cdots&g_{l-1,l-1}&x^{l-1}\\\hline
          g_{l,0}&g_{l,1}&\cdots&g_{l,l-1}&x^l
\end{array}\right),& l&\geq 1,
\end{align*}
and similar expressions for the dual polynomials. We are now ready to get the recursion relations for orthogonal polynomials:
   \begin{itemize}
 \item  First, we notice that the moment matrix $g$ is a Hankel matrix, $g_{i+1,j}=g_{i,j+1}$, which in terms of the shift matrix $\Lambda:=\left(\begin{smallmatrix}
    0&1&0&0&\dots\\
    0&0&1&0&\dots\\[-6pt]
    0&0&0&1&{\tiny\ddots}\\
    \vdots&\vdots&\vdots&{\tiny\ddots}&{\tiny\ddots}
  \end{smallmatrix}\right)$ can be written as $\Lambda g=g\Lambda^\top$.\footnote{From this symmetry property it follows, by contradiction, that the moment matrix is not invertible; i.e. the assumption of the existence of $g^{-1}$ leads to $g^{-1}\Lambda =\Lambda^\top g^{-1}$, and therefore the first row and column of $g^{-1}$ are identically zero, so that $g^{-1}$ is not invertible.}
 \item Second, we observe  the eigen-value property $\Lambda\chi(x)=x\chi(x)$.
  \item Third, we introduce the $LU$ factorization to get $\Lambda S^{-1}\bar S=S^{-1} \bar S\Lambda^\top$ $\Rightarrow S\Lambda  S^{-1}=\bar S\Lambda^\top\bar S^{-1}=:J$. From this last relation we deduce that the matrix $J=\left(\begin{smallmatrix}
    a_0&1&0&0&\dots\\
    b_1&a_1&1&0&\dots\\[-6pt]
    0&b_2&a_2&1&{\tiny\ddots}\\
    \vdots&\vdots&{\tiny\ddots}&{\tiny\ddots}&{\tiny\ddots}
  \end{smallmatrix}\right)$ is a tridiagonal matrix, i.e. a Jacobi matrix.

  \item Finally, we notice that the polynomial strings are eigenvectors of the Jacobi matrix: $J P(x)=S\Lambda S^{-1}S\chi(x)=S\Lambda \chi(x)=Sx\chi(x)=xP(x)$; i.e., the recursion relations $xP_k(x)=P_{k+1}(x)+a_kP_k(x)+b_kP_{k-1}(x)$, $k> 0$, hold.
\end{itemize}

We now consider the Aitken--Berg--Collar (ABC)  theorem (here we follow the nomenclature used \cite{simon}) for orthogonal polynomials. First we introduce
the Christoffel--Darboux kernel   and therefore we consider
  \begin{align*}
\H^{[l]}&=\R\{0\dots,x^{l-1}\},&\H&=\Big\{\sum_{0\leq l \ll \infty} c_l x^{l}, c_l\in\R\Big\},&
( \H^{[l]})^\bot&=\Big\{\sum_{l\leq k \ll \infty} c_lP^l(x), c_l\in\R\Big\}
\end{align*}
and the resolution of the identity $ \H=\H^{[l]}\oplus (\H^{[ l]})^\bot$, with the corresponding orthogonal projector $\pi^{(l)}$ such that  $\text{ker } \pi^{(l)}=(\H^{[l]})^\bot$ and $\text{Ran }\pi^{(l)}=\H^{[l]}$. Then, the  Christoffel--Darboux is defined as
\begin{align*}
K^{[l]}(x,y):=\sum_{k=0}^{l-1}P_{k}(y)\bar P_{k}(x)=\sum_{k=0}^{l-1}h_k^{-1}P_{k}(y)P_{k}(x),
\end{align*}
which, according to the bi-orthogonality property, gives the following integral representation of the projection operator
  \begin{align*}
 ( \pi^{(l)}f)(y)&=\int K^{[l]}(x,y)f(x)\d \mu(x), & \forall f\in\H,
\end{align*}

Any semi-infinite vector  $v$ can be written in block form as follows
\begin{align*}v&= \left(\begin{array}{c}
v^{[l]}\\\hline\\[-10pt]
 v^{[\geq l]}
\end{array}\right)
\end{align*}
 $v^{[l]}$ is the finite vector formed with the first $l$ coefficients of $v$ and $v^{[\geq l]}$ the
semi-infinite vector formed with the remaining coefficients.
This decomposition induces the following block structure for any semi-infinite matrix.
\begin{align*}
  g= \left(\begin{array}{c|c}
   g^{[l]}&g^{[l,\geq l]}\\\hline\\[-10pt]
     g^{[\geq l,l]}& g^{[\geq l]}
\end{array}\right).
\end{align*}
 Given a factorizable  moment matrix $g$ we have
\begin{align*}
  g^{[l]}&=(S^{[l]})^{-1}\bar S^{[l]},&(S^{-1})^{[l]}&=(S^{[l]})^{-1},&(\bar S^{-1})^{[\geq l]}&=(\bar S^{[\geq l]})^{-1}.
\end{align*}

  The   Christoffel--Darboux kernel is related to the moment matrix in the
  following way
\begin{align*}
     K^{[l]}(x,y)&=(\chi^{[l]}(x))^\top (g^{[l]})^{-1}\chi^{[l]}(y)
\end{align*}
which is a consequence of the following identities
\begin{align*}
  K^{[l]}(x,y)&=(\Pi^{[l]}\bar P(x))^\top(\Pi^{[l]} P(y))\\
  &=\chi^\top(x) \bar S^{-1}\Pi^{[l]} S\chi(y)\\
  &=\chi^\top(x) (\Pi^{[l]}\bar S^{-1}\Pi^{[l]})(\Pi^{[l]} S\Pi^{[l]})\chi(y) \\
      &=(\chi^{[l]}(x))^\top (g^{[l]})^{-1}\chi^{[l]}(y).
\end{align*}

The relations
\begin{align*}
(g^{[l]})^{-1}  \Lambda^{[l]}-(\Lambda^{[l]})^\top(g^{[l]})^{-1}&=(g^{[l]})^{-1}
\Big(g^{[l,\geq l]}(\Lambda^{[l,\geq l]})^\top-\Lambda^{[l,\geq l]}g^{[\geq l,l]}\Big)(g^{[l]})^{-1}
\end{align*}
follow from the block equation
\begin{align*}
  \Lambda^{[l]}g^{[l]}+\Lambda^{[l,\geq l]}g^{[\geq l,l]}&= g^{[l]}(\Lambda^{[l]})^\top+g^{[l,\geq
  l]}(\Lambda^{[l,\geq l]})^\top.
\end{align*}
We also have
 \begin{align*}
  \Lambda^{[l]}\chi^{[l]}(x)&=x\chi^{[l]}(x)-\Lambda^{[l,\geq l]}\chi^{[\geq l]}(x),&
  \Lambda^{[l,\geq l]}&= e_{l-1}e_{0}^\top,
\end{align*}
where $\{e_i\}_{i=0}^\infty$ is the canonical linear basis of $\H$.
With all these at hand we deduce
\begin{align*}
(\chi^{[l]}(x))^\top \big((g^{[l]})^{-1}
\Lambda^{[l]}-(\Lambda^{[l]})^\top(g^{[l]})^{-1}\big)\chi^{[l]}(y)&=(\chi^{[l]}(x))^\top
(g^{[l]})^{-1}
\big(g^{[l,\geq l]}(\Lambda^{[l,\geq l]})^\top-\Lambda^{[l,\geq l]}g^{[\geq
l,l]}\big)(g^{[l]})^{-1}\chi^{[l]}(y)
\end{align*}
so that,
\begin{align*}
  \hspace*{-5pt} (x-y)K^{[l]}(x,y)=&\Big((\chi^{[\geq l]}(x))^\top- (\chi^{[l]}(x))^\top
   (g^{[l]})^{-1}  g^{[l,\geq l]}\Big) e_{0}e_{l-1}^\top (g^{[l]})^{-1}
   \chi^{[l]}(y)\\&
   -(\chi^{[l]}(x))^\top (g^{[l]})^{-1} e_{l-1}e_{0}^\top\Big(\chi^{[\geq l]}(y)-g^{[\geq
   l,l]}(g^{[l]})^{-1}\chi^{[l]}(y)\Big).
\end{align*}
That using the determinantal expressions for the polynomials presented before leads to the Christoffel--Darboux formula
\begin{align*}
(x-y)K^{[l]}(x,y)=h_{l-1}^{-1}(P_l(x)P_{l-1}(y)-P_{l-1}(x)P_l(y)).
\end{align*}

\section{Multiple orthogonal polynomials and Gauss--Borel factorization}

\subsection{The moment matrix}

In this section we define the moment matrix in terms of  a measure $\mu \in {\mathcal{M}}(\Delta)$ and two systems of weights $\vec w_1$ and $\vec w_2$ on
$\Delta \subset {\mathbb{R}}$, as well as corresponding compositions (the order matters) $\vec
n_1=(n_{1,1},\dots,n_{1,p_1})\in\N^{p_1}$ and
 $\vec n_2=(n_{2,1},\dots,n_{2,p_2})\in\N^{p_2}$ \cite{stanley}. We will consider multi-indices of positive integers $\vec n=(n_1,\dots,
n_p)$, where $p\in \N$ and $n_a\in{\mathbb{Z}}_+$, $a=1,\dots,p$ and define $|\vec n|:=n_1+\dots+n_p$.
 Following \cite{bergvelt,stanley} we observe that  any  $i\in\Z_+:=\{0,1,2,\dots\}$ determines  unique non-negative
integers $q(i),a(i),r(i)$, such that the composition
\begin{align}\label{i}
  i&=q(i)|\vec n|+n_{1}+\dots+n_{a(i)-1}+r(i), & 0&\leq r(i)<n_{a(i)},
\end{align}
holds. 
Hence, given $i$ there is a unique $k(i)$ with
\begin{align}\label{k}
  k(i)&=q(i) n_{a(i)}+r(i),& 0&\leq r(i)<n_{a(i)}.
\end{align}
Let us introduce the function integer part function $[\cdot]:{\mathbb{R}}_+ \rightarrow {\mathbb{Z}}_+$,
$[x]=\max \{ y \in \Z_+, y \leq x\}$.
Combining  \eqref{i} and \eqref{k} we can obtain a formula which expresses explicitly the dependence between the quantities $i$, $k$ and $a$
\begin{align}\label{iak}
i=\left[\frac{k}{n_a}\right] \left(|\vec n|-n_a\right)+n_1+\cdots +n_{a-1}+k .
\end{align}
Let ${\mathbb{R}}^{\infty}$ denote the vector space of all sequences with elements in ${\mathbb{R}}$. An element $\lambda \in {\mathbb{R}}^{\infty}$ may be interpreted as a column semi-infinity vector as follows
\begin{align*}
\lambda&= (\lambda^{(0)},\lambda^{(1)},\dots)^\top,&    \lambda^{(j)} &\in \R,&  j&=0,1,\dots.
\end{align*}
We consider  the set $\{e_j\}_{j\geq0} \subset {\mathbb{R}}^{\infty}$ with
\begin{align*}
e_j=(\overset{j}{\overbrace{0,0,\dots,0}},1,0,0,\dots)^{\top}.
\end{align*}
Here $(\cdot)^\top$ denotes the transposition function on vectors and matrices. Analogously, we denote by $(\R^p)^\infty$ the set of all sequences of vectors with $p$ components and observe that each sequence which belongs to $(\R^p)^\infty$ can also be understood as semi-infinity column vector: given the vector sequence $(\vec v_0,\vec v_1,\dots)$ with $\vec v_j=(v_{j,1},\dots,v_{j,p})^\top$ we have the corresponding sequence in $\R^\infty$ given by $(v_{0,1}\dots,v_{0,p},v_{1,1},\dots,v_{1,p},\dots)$; i.e., ${\mathbb{R}}^{\infty} \cong({\mathbb{R}}^p)^{\infty}$. Therefore, we consider also the set    $\{e_a(k)\}_{\substack{a=1,\dots,p\\ k=0,1,\dots}}\subset ({\mathbb{R}}^p)^{\infty}$ where for each pair $(a,k)\in \{1,\ldots,p\}\times {\mathbb{Z}}_+$ $e_{a}(k)=e_{i(k,a)}$ and the function $i(a,k) \in {\mathbb{Z}}_+$ satisfies the equality \eqref{iak}.

Now, we are ready to introduce the monomial strings
 \begin{align}\label{chion}
  \chi_a&:=\sum_{k=0}^\infty e_a(k) z^k,&
  \chi_a^{(l)}&=\begin{cases}
    z^{k(l)},& a=a(l),\\
    0, &a\neq a(l),
  \end{cases}&
  \chi^*_a:&=z^{-1}\chi_a(z^{-1}).
\end{align}
These vectors may be understood as sequences of monomials according to the
composition $\vec n$ introduced previously.
We also define the following weighted monomial string
\begin{align}
\xi&:=\sum_{a=1}^p\chi_aw_a, &
  \xi^{(l)}&=    w_{a(l)}z^{k(l)},
\end{align}
which is a sequence of weighted monomials  for each given composition $\vec n$. Sometimes, when we what to stress the
dependence in the composition we write $\chi_{\vec n,a},\chi_{\vec n}$ and $\xi_{\vec n}$.  Given the weighted monomials $\xi_{\vec n_1}$ and $\xi_{\vec n_2}$, associated to the compositions $\vec n_1$ and $\vec n_2$, we introduce the moment matrix in the following manner
\begin{definition}
  The moment matrix is given by
\begin{align}
  \label{compact.g}
  g_{\vec n_1,\vec n_2}:=\int \xi_{\vec n_1}(x)\xi_{\vec n_2}(x)^\top\d \mu(x).
\end{align}
\end{definition}
 In terms of the canonical basis $\{E_{i,j}\}$ of the linear space of semi-infinite matrices and for each pair $(i,j)\in
 {\mathbb{Z}}_+^{2}$ we consider the binary permutations or transpositions $ \pi_{i,j}=E_{i,j}+E_{j,i}$. Observe that
 $\pi_{i,j}^2=\I$ and therefore $\pi_{i,j}^{-1}=\pi_{i,j}$. Given two transpositions $\pi_{i,j}$ and $\pi_{k,l}$ the
 permutation endomorphism corresponding to its product is well defined $\pi_{i,j}\pi_{k,l}=\pi_{k,l}\pi_{i,j}$.
 Taking a sequence of pairs $I=\{(i_s,j_s)\}_{s \in {\mathbb{Z}}_+}$, $i_s,j_s\in{\mathbb{Z}}_+$, we introduce the permutation endomorphism $\pi_I$ as the
 infinite product  $\pi_I=\prod_{s \in {\mathbb{Z}}_+} \pi_{i_s,j_s}$, with $\pi_I\pi_I^\top=\pi_I^\top\pi_I=\I$.
 Given two compositions, $\vec n',\vec n$, there exists  a permutation $\pi_{\vec n',\vec
 n}$ such that $  \xi_{\vec n'}=\pi_{\vec n',\vec
 n}\xi_{\vec n}$ through a permutation semi-infinite matrix as just described.

The change in the compositions is modeled as follows
\begin{pro}
Given two set of weights $\vec w_{\ell}=(w_{\ell,1},\dots,w_{\ell,p_\ell})$
and  compositions $\vec  n_\ell$ and $\vec n'_\ell$, $\ell=1,2$, there exist permutation matrices
$\pi_{\vec n_\ell',\vec n_\ell}$ such that
\begin{align}
   \label{different moments}
   g_{\vec n_1',\vec n_2'}=\pi_{\vec n_1',\vec n_1}g_{\vec n_1,\vec n_2}\pi_{\vec n_2',\vec n_2}^\top.
 \end{align}
\end{pro}
\begin{proof}
  For any set of weights $\vec w={w_1,\dots,w_p}$ and two compositions $\vec  n$ and $\vec n'$ we have that the
  corresponding vectors of weighted monomials are connected,
\begin{align*}
  \xi_{\vec n'}=\pi_{\vec n',\vec n}\xi_{\vec n}
\end{align*}
trough a permutation semi-infinite matrix; i.e,  $\pi_{\vec n',\vec n}^\top=\pi_{\vec n',\vec n}^{-1}$.
 Therefore, the announced result follows.
\end{proof}

 For the sake of notation simplicity and when the context is clear enough we will drop the subindex indicating the two
 compositions and just write $g$ for the moment matrix. Let us discuss in more detail the block Hankel structure of the moment matrix. 
For each pair $(i,j) \in {\mathbb{Z}}_+^2$ there exists a unique combination of three
others pairs  $(q_1,q_2) \in {\mathbb{Z}}_+^2,$ $(a_1,a_2)\in\{1,\ldots,p_1\}\times\{1,\ldots,p_2\} $ and $
(r_1,r_2) \in \{0,\ldots, n_{1,a_1}-1\} \times \{0,\ldots,n_{2,a_2}-1\},$
such that
\[
i=q_1|\vec n_1|+n_{1,1}+\ldots+n_{1,a_1-1}+r_1 \quad \mbox{and} \quad j=q_2|\vec n_2|+n_{2,1}+\ldots+n_{2,a_2-1}+r_2.
\]
Hence taking $k_{\ell}= q_{\ell} n_{\ell,a_{\ell}}+r_{\ell},$ $\ell=1,2,$ the coefficients $g_{i,j}\in {\mathbb{R}}$
of the moment matrix $g=(g_{i,j})$ have the following explicit form
\begin{align}\label{explicit g2}
\begin{aligned}
  g_{i,j}&=\int  x^{k_1+k_2}w_{1,a_1}(x)w_{2,a_2}(x)\d \mu(x) .
\end{aligned}
\end{align}
Observe that pairs $(k_1,a_1)$ and $(k_2,a_2)$ are univocally determined by $i$ and $j$ respectively.

Before we continue with the study of this moment matrix it is necessary to introduce some auxiliary objects associated with the vector space $\R^\infty$. First,  we have the unity matrix $\I=\sum_{k=0}^\infty e_k e_k^\top$
and the shift matrix $\Lambda:=\sum_{k=0}^\infty e_ke_{k+1}^\top$. We also define the projections $  \Pi^{[l]}:=\sum_{k=0}^{l-1} e_k e_k^\top$,
and with the help of the set  $\{e_a(k)\}_{\substack{a=1,\dots,p\\ k=0,1,\dots}}$ we construct the projections
$ \Pi_a:=\sum_{k=0}^{\infty}e_a(k) e_a(k)^\top$ with $\sum_{a=1}^p\Pi_a=\I$, and
\begin{align}\label{projectors}
  P_{1}&:=\diag(\I_{n_1}, 0_{n_2},\dots,0_{n_p}),&
  P_{2}&:=\diag( 0_{n_1},\I_{n_2},\dots,0_{n_p}),&&\dots&
   P_{p}&:=\diag( 0_{n_1},0_{n_2},\dots,\I_{n_p}),
\end{align}
where $\I_{n_s}$ is the $n_s\times n_s$ identity matrix. Finally we introduce the notation
\begin{align}
  x^{\vec n}:=x^{n_1}P_{1}+\dots+x^{n_p}P_{p}=\diag(x^{n_1}\I_{n_1},\dots,x^{n_p}\I_{n_p}):\R\to\R^{|\vec
  n|\times|\vec n|}.
\end{align}

For a better insight of the moment matrix let us introduce the following $n_{1,a}\times n_{2,b}$   matrices
 \begin{align}\label{Wab}
  m_{a,b}(x)=w_{1,a}(x)w_{2,b}(x)\begin{pmatrix}
   1& x&\cdots& x^{n_{2,b}-1}\\
   x&  x^2&\cdots&  x^{n_{2,b}}\\
     \vdots&\vdots&&\vdots\\
       x^{n_{1,a}-1}&x^{n_{1,a}}&\cdots&  x^{n_{1,a}+n_{2,b}-2}
  \end{pmatrix},&&\begin{aligned}
   a&=1,\dots,p_1,\\ b&=1,\dots,p_2,
  \end{aligned}
\end{align}
in terms of which we build up the following $|\vec n_1|\times|\vec n_2|$-matrix
\begin{align}\label{W}
  m:=\begin{pmatrix}
    m_{1,1}&m_{1,2}&\cdots&m_{1,p_2}\\
       m_{2,1}&m_{2,2}&\cdots&m_{2,p_2}\\
\vdots&\vdots&&\vdots\\
            m_{p_1,1}&m_{p_1,2}&\cdots&m_{p_1,p_2}
  \end{pmatrix}:\R\to\R^{|\vec n_1|\times|\vec n_2|}.
\end{align}
Then, the moment matrix $g$ has the following block structure
\begin{align}\label{block G}
  g&:=(G_{i,j})_{i,j\geq 0}\in\R^{\infty\times\infty},& 
  G_{i,j}:=\int x^{i \vec n_1}m(x)x^{j\vec n_2}\d \mu(x)\in \R^{|\vec n_1|\times|\vec n_2|}.
\end{align}

Fix now a number $l \in {\mathbb{N}}$ and consider the pair $(l,l+1)$.  There exists a
unique combination of pairs $(q_1,q_2) \in {\mathbb{Z}}_+^2,$ $(a_1,a_2)\in\{1,\ldots,p_1\}\times\{1,\ldots,p_2\} $
and $
(r_1,r_2) \in \{0,\ldots, n_{1,a_1}-1\} \times \{0,\ldots,n_{2,a_2}-1\},$
such that
\[
l=q_1|\vec n_1|+n_{1,1}+\dots+n_{1,a_1-1}+r_1 \quad \mbox{and} \quad l+1=q_2|\vec
n_2|+n_{2,1}+\dots+n_{2,a_2-1}+r_2.
\]


Given the compositions $\vec n_1$ and $\vec n_2$ we introduce the  degree multi-indices $\vec \nu_1 \in {\mathbb{Z}}_+^{p_1}$ and $\vec  \nu_2 \in{\mathbb{Z}}_+^{p_2}$  \cite{bergvelt} where
for each $\ell=1,2,$ we have
\begin{align}\label{mult:u}
\begin{aligned}\vec \nu_\ell&=(\nu_{\ell,1},\dots,\nu_{\ell,a_\ell-1},\nu_{\ell,a_\ell},\nu_{\ell,a_\ell+1},\dots, \nu_{\ell,p_\ell})\\
&=((q_\ell+1)n_{\ell,1},\dots,(q_\ell+1)n_{\ell,a_\ell-1},q_\ell n_{\ell,a_\ell}+r_\ell,q_\ell
n_{\ell,a_\ell+1},\dots,q_\ell n_{\ell,p_\ell}),
\end{aligned}
\end{align}
which satisfy
\begin{align}\label{formulitas}
  k_\ell(i+1)&=\nu_{\ell,a_\ell(i)}(i),&|\vec\nu_\ell(i)|&=i+1,&
  \vec \nu_\ell(i+l|\vec n_\ell|)&=\vec\nu_\ell(i)+l\vec n_\ell,
\end{align}
and consider the $l \times (l+1)$ block matrix $\Gamma_{l}$ from $g$
\begin{align}\label{Gamma}
\Gamma_{l}=\begin{pmatrix}
g_{0,0} & g_{0,1} & \cdots & g_{0,l} \\
g_{1,0} & g_{1,1} & \cdots & g_{1,l} \\
\vdots   & \vdots   &  & \vdots  \\
g_{l-1,0} & g_{l-1,1} & \cdots & g_{l-1,l} \\
\end{pmatrix}.
\end{align}
Let us study the homogeneous system $\Gamma_{l}{\boldsymbol x}_{l+1}={\boldsymbol 0}_{l},$ where $\boldsymbol
x_{l+1}\in\R^{l+1}$ and ${\bf 0}_{l+1}$ is the null vector in ${\mathbb{R}}^{l+1}$. Taking into account $\Gamma_{l}$'s
structure \eqref{explicit g2}, we see that such equation is exactly the expression of the orthogonality relations
\eqref{orth}. We can see now that for each $l \in {\mathbb{N}}$ the existence of a system of mixed multiple orthogonal
polynomials $\big(A^{(l)}_1, \ldots, A^{(l)}_{p_1}\big)$ is ensured; that is because $\Gamma_{l}$ in \eqref{Gamma} is
a $l \times (l+1)$ matrix, and the homogeneous matrix equation $\Gamma_{l} {\boldsymbol x}_{l+1}=\boldsymbol{0}_{l},$
which is satisfied by the coefficients of the polynomials corresponds to a system of $l$ homogeneous linear equations
for $l+1$ unknown coefficients. Thus, the system always has a non-trivial solution.   Obviously, $\big(A^{(l)}_1,
\ldots, A^{(l)}_{p_1}\big)$ is not univocally determined by the matrix equation $\Gamma_l{\boldsymbol x}_{l+1}={\bf
0}_{l}$ or equivalently by the orthogonality relations \eqref{orth}, because its solution space has at least dimension
1.    Hence, the appropriate  question to consider  is the uniqueness question without counting constant factors, or
equivalently if the solution space has exactly dimension 1. In terms of $\Gamma_{l}$ the question becomes: Does
$\Gamma_{l}$ have rank $l$? In order to have a positive answer it is sufficient to ensure that the $l \times l$ square
matrix
 \begin{align}\label{gl}
  g^{[l]}&:=\begin{pmatrix}
    g_{0,0}&g_{0,1}&\cdots&g_{0,l-1}\\
    g_{1,0}&g_{1,1}&\cdots&g_{1,l-1}\\
    \vdots&\vdots&&\vdots\\
     g_{l-1,0}&g_{l-1,1}&\cdots&g_{l-1,l-1}
  \end{pmatrix},& l\geq 1,
  \end{align}
 is invertible, where $g^{[l]}$ results from $\Gamma_{l}$ after removing  its last column. It is easy to prove that
 such condition is equivalent to require that  all possible  solutions of \eqref{orth} satisfy $\deg A_{p_1}
 =\nu_{1,p_1}-1$. Obviously  this requirement is ensured when the polynomials $\big(A^{(l)}_1, \ldots,
 A^{(l)}_{p_1}\big)$ fulfill  $\deg A_{j} = \nu_{1,j} -1,$ $j=1,\ldots, p_1$.

 \subsection{Perfect combinations and Nikishin systems}\label{perfectness}

We introduce the concept of a perfect combination.

\begin{definition}\label{pag8}  A combination $(\mu, \vec w_1, \vec w_2)$ of a measure $\mu \in {\mathcal{M}}(\Delta)$ and two
systems of weights $\vec w_1$ and $\vec w_2$ on $\Delta \subset {\mathbb{R}}$ is said to be perfect if for each pair
of multi-indices $(\vec \nu_1,\vec \nu_2)$, with $|\vec \nu_1|=|\vec \nu_2|+1$  the orthogonality relations
\eqref{orth} imply that $\deg A_a=\nu_{1,a}-1,$ $a=1, \ldots, p_1$.
\end{definition}

For  a perfect combination $(\mu,\vec w_1, \vec, w_2)$ and any  given $l \in {\mathbb{Z}}_+$ the solution space of the  equation $\Gamma_{l} {\boldsymbol x}_{l+1}=\boldsymbol 0_{l}$ is one-dimensional.  Then, we can determine a unique system of mixed type orthogonal polynomials $\big(A_1, \ldots, A_{p_1}\big)$ satisfying \eqref{orth} requiring for  $a_1 \in \{1, \ldots p_1\}$ that $A_{a_1}$ monic. Following \cite{daems-kuijlaars} we say that we have a type II normalization and denote the corresponding system of polynomials by   $A_a^{(\rII,a_1)},$ $j=1, \ldots, p_1$. Alternatively, we can proceed as follows, since the system of weights is perfect from \eqref{orth} we deduce that
\begin{align*}
\int x^{\nu_{2,b_2}} \sum_{a=1}^{p_1} A_a(x)w_{1,a} (x)w_{2,b_2}(x) \d\mu (x)\neq 0.
\end{align*}
Then, we can determine a unique system of mixed type of multi-orthogonal polynomials
$(A_1^{(\rI,b_2)},\ldots,A_{p_2}^{(\rI,b_2)})$ imposing that
\begin{align*}
\int x^{\nu_{2,b_2}} \sum_{a=1}^{p_1} A_a^{(\rI,b_2)}(x)w_{1,a} (x)w_{2,b}(x) \d\mu (x)=1,
\end{align*}
which is a type I normalization. We will use the notation $A_{[\vec\nu_1;\vec\nu_2],a}^{(\rII,a_1)}$ and $A_{[\vec\nu_1;\vec\nu_2],a}^{(\rI,b_2)}$ to denote these multiple orthogonal polynomials with type II and I normalizations, respectively.

A known illustration of perfect combinations $(\mu, \vec w_1, \vec w_2)$ can be constructed with an arbitrary positive finite Borel measure $\mu$ and  systems of weights formed with exponentials:
\begin{align}\label{exponentials}
&  (\Exp{\gamma_1x},\ldots,\Exp{\gamma_px}),& \gamma_i &\neq \gamma_j, &i &\neq j,& i,j& = 1,\ldots,p,
\end{align}
 or by binomial functions
\begin{align}\label{binomials}
  &((1-z)^{\alpha_1},\ldots,(1-z)^{\alpha_p}),& \alpha_i -\alpha_j &\not\in {\mathbb{Z}}, &i &\neq j,& i,j& =
  1,\ldots,p.
\end{align}
or combining  both classes, see \cite{Nik}. Recently a wide class of systems of weights where proven to be perfect \cite{LF2}; these systems of functions, now called Nikishin systems, were introduced by E.M. Nikishin  \cite{Nik} and initially named MT-systems (after Markov and Tchebycheff).

Given a closed interval $\Delta$ let $\overset{\circ}{\Delta}$ be the interior set of $\Delta.$ Let us take two intervals $\Delta_{\alpha}$ and $\Delta_{\beta}$ whose interior sets are disjoint, i.e. $\overset{\circ}{\Delta}_{\alpha} \cap \overset{\circ}{\Delta}_{\beta}=\emptyset.$ Set two measures $\mu_{\alpha} \in {\mathcal{M}}(\Delta_{\alpha})$ and $\mu_{\beta}\in {\mathcal{M}}(\Delta_{\beta})$ such that the measure $\langle \mu_{\alpha},\mu_{\beta}\rangle$ with the following differential form
\[
\d \langle \mu_{\alpha},\mu_{\beta}\rangle(x)= \int \frac{ \d \mu_{\beta}(t)}{x-t} \d \mu_{\alpha}(x)=\widehat{\mu}_{\beta}(x) \d \mu_{\alpha}(x),
\]
is a finite measure, that implies that $\langle \mu_{\alpha},\mu_{\beta}\rangle \in {\mathcal{M}}(\Delta_{\alpha}).$  The function $\widehat{\mu}_{\beta}$ denotes the Cauchy transform corresponding to $\mu_{\beta}.$  Let us consider then a system of $p$ intervals $\Delta_1,\ldots, \Delta_p$ such that $\overset{\circ}{\Delta}_j \cap \overset{\circ}{\Delta}_{j+1} =\emptyset$,  $j \in \{1, \ldots,p-1\}$. Take $p$ measures $\mu_j \in {\mathcal{M}}(\Delta_j),$ which  for each $j=1, \ldots,p-1,$ the measure $\langle \mu_j,\sigma_{j+1}\rangle$ belongs to ${\mathcal{M}}(\Delta_j)$. So the system of measures $(\xi_0,\ldots,\xi_p)$ where
\[
\zeta_1 = \mu_1, \quad \zeta_2= \langle \mu_1, \mu_2\rangle, \quad \zeta_3=\langle \mu_1,\langle \mu_2,\mu_3\rangle \rangle=\langle \mu_1,\mu_2,\mu_3\rangle, \quad  \ldots, \quad \zeta_p=\langle \mu_1, \ldots, \mu_p\rangle,
\]
is the Nikishin system of measures generated  by the system $(\sigma_1,\ldots,\sigma_m)$. So we denote $(\zeta_1,\ldots,\zeta_p)={\mathcal{N}}\left(\sigma_1,\ldots,\sigma_p\right).$

Actually, in \cite{LF2} the authors shown perfectness for combinations of Nikishin systems where intervals $\Delta_1,\ldots, \Delta_p$ are bounded and for  each $j \in \{1,\ldots,p-1\}$ the intervals $\Delta_j$ and $\Delta_{j+1}$ are disjoint. The same authors have communicated to us that they were able to prove a generalization of this result to unbounded intervals such that $\Delta_j \cap \Delta_{j+1} \not = \emptyset$. Consequently, in what follows  we assume  such generalization.

As we have seen, general Nikishin systems  have an intricate structure; therefore, in order to make easy the reader we focus on a ``simple" class of Nikishin systems which we call M-Nikishin systems. Set the interval $\Delta_1=[0,1]$ and let ${\mathcal{M}}_0(\Delta_1) \subset {\mathcal{M}}(\Delta_1)$ denote the set of measures in ${\mathcal{M}}(\Delta_1)$ such that if $\sigma \in {\mathcal{M}}_0(\Delta_1)$ then, the function
\begin{align}\label{constrain}
\widetilde{\sigma}(x):= \int_{\Delta_1} \frac{\d \sigma(t)}{1-tx} \quad \mbox{satisfies } \quad \lim_{\underset{x \in\overset{\circ}{\Delta}_1}{x \to 1}} \left|\int_{\Delta_1} \frac{\d \sigma(t)}{1-t x}\right| =
\lim_{\underset{x \in \overset{\circ}{\Delta}_1}{x \to 1}} \int_{\Delta_1} \frac{\left|\d \sigma(t)\right|}{1-t x}< +\infty, \quad \mbox{where} \quad \overset{\circ}{\Delta}_1=(0,1).
\end{align}
The constraint \eqref{constrain} guarantees that the function $\widetilde{\sigma}$ is a weight in compact intervals in $(-\infty,1]$. As $(1-tx)$ does not vanish for $(t,x) \in \Delta_1\times( \C\setminus [1,+\infty))$ we deduce that  $1/(1-tx)$ is a continuous function in $x$ for $t \in \Delta_1$. Therefore, we conclude that $\widetilde{\sigma}$ is a holomorphic function on $\C \setminus (1,+\infty)$, having a continuation as continuous function in $1.$ Taking into account that $\widetilde{\sigma}$ does not vanish in ${\mathbb{C}} \setminus (1,+\infty)$ and that it takes real values on ${\mathbb{R}} \setminus (1,+\infty)=(-\infty,1]$, we deduce that it is a continuous weight on $(-\infty,1]$. Observe that
\begin{equation}\label{transform}
\widetilde{\sigma} (x)=\int_{\Delta_1} \frac{d \sigma (t)}{1-tx}=\int_{[1,+\infty)} \frac{\zeta d \sigma(1/\zeta)}{x-\zeta}=\int_{[1,+\infty)} \frac{\d \mu (\zeta)}{x-\zeta},
\end{equation}
is the Cauchy transform of another measure $\mu \in {\mathcal{M}}([1,+\infty)),$ such that $\left|\widehat{\mu} (1)\right|=\left|\widetilde{\sigma}(1)\right| < +\infty.$

Given two measures  $\sigma_{\alpha} \in {\mathcal{M}}_0(\Delta_1), \sigma_{\beta} \in {\mathcal{M}}_0(\Delta_1)$ we define a third one as follows (using the differential notation)
\begin{align*}
\d [ \sigma_{\alpha},\sigma_{\beta}] (x) &= \widetilde{\sigma}_{\beta}(x) \d\sigma_{\alpha}(x),&
\widetilde{\sigma}_\beta(x)&=\int_{\Delta_1} \frac{\d \sigma_{\beta} (\zeta)}{1-x\zeta}.
\end{align*}
As $\widetilde{\sigma}_{\beta}$ is a continuous weight on $\Delta_1$ we conclude that $[ \sigma_{\alpha},\sigma_{\beta} ] \in\mathcal M_0(\Delta_1)$. If we take a system of measures $(\sigma_1,\ldots,\sigma_p)$ such that $\sigma_j \in {\mathcal{M}}_0(\Delta_1), j=1,\ldots,p$,  we say that $(s_1,\ldots,s_p) = {\mathcal{MN}}(\sigma_1,\ldots,\sigma_p)$, where
\begin{align}\label{M-Nikishin}
  s_1& = \sigma_1,&s_2&=[\sigma_1,\sigma_2],&s_3&= [ \sigma_1,[\sigma_2,\sigma_3]]=[ \sigma_1,\sigma_2,\sigma_3],&& \dots & s_p &= [ \sigma_1,
  \sigma_2,\ldots,\sigma_p ]
\end{align}
is the M-Nikishin system of measures generated by $(\sigma_1,\ldots,\sigma_p)$, with corresponding  M-Nikishin system of functions given by $\vec w=(w_1,\ldots,w_p)=(\widetilde{s}_1,\ldots,\widetilde{s}_p)={\mathcal{M\widehat{N}}}(\sigma_1,\ldots,\sigma_p)$.

Notice that $s_i \in {\mathcal{M}}_0(\Delta_1)$ which implies that for each arbitrary compact subinterval of $(-\infty,1]$ the system of functions $\vec w$ conforms a system of continuous  weights. M-Nikihsin systems are included in the class of Nikishin systems. Taking into account the identity (\ref{transform}) we see that the M-Nikishin system defined in (\ref{M-Nikishin}) can be written as a classical Nikishin system. Let us take a system $(\mu_1, \ldots,\mu_p)$ where
\[
\mu_1=\sigma_1, \quad \d \mu_2(x)= x \d \sigma_2 (1/x), \quad \mu_3=\sigma_3, \quad \ldots, \quad \mu_{2[p/2]-1}=\sigma_{2[p/2]-1},\quad \d \mu_{2[p/2]}(x)=x \sigma_{2[p/2]}(1/x),
\]
and if $p$ is odd $\mu_p= \sigma_p.$ Notice then
 \begin{align*}
s_1& = \zeta_1= \sigma_1,& s_2&=\zeta_2= \langle \mu_1,\mu_2 \rangle,& \dots & &s_p&=\zeta_p= \langle \mu_1, \mu_2,\ldots,\mu_p \rangle.
\end{align*}
Hence $(s_1,\ldots,s_p)={\mathcal{MN}}(\sigma_1,\ldots,\sigma_p)={\mathcal{N}}(\mu_1,\ldots,\mu_p)=(\zeta_1,\ldots,\zeta_p).$ Fixing two M-Nikishin systems of functions $\vec w_\ell(x)=(\widetilde{s}_{\ell,1}(x),\ldots\widetilde{s}_{\ell,p}(x))$  whose
elements are weights on $\Delta_0=[-1,1]$,   and a  measure $\mu\in {\mathcal{M}}(\Delta_0)$ we have at our disposal the perfect combination
$(\mu,\vec w_1,\vec w_2)$. We can also obtain a perfect combination $(\mu,\vec w_1,\vec w_2)$ choosing $\vec w_1$ and $\vec w_2$  between two different of the classes mentioned in (\ref{exponentials}) and (\ref{binomials}) (not necessarily  the same).

\begin{pro}\label{ProTaylor} The Taylor series at $\zeta=0$ corresponding to the  functions $\widetilde{s}_j(\zeta)$ and  $f_j(\zeta):=\log \widetilde{s}_j(\zeta)$ converge uniformly to $\widetilde{s}_j$ and $f_j$ respectively on $\Delta_1$, i.e.
\begin{align}\label{taylor}
\widetilde{s}_j(x)= \sum_{i=0}^{\infty} \lambda_{i, j} x^{i}=e^{\sum_{i=0}^{\infty} t_{i, j} x^{i}}, \qquad x \in \Delta_0, \qquad j=1, \ldots, p.
\end{align}
 where $\lambda_{i,j}$ and $t_{i,j}$ are constants.
\end{pro}

\begin{proof} For each $j\in \{1,\ldots,p\},$ $\widetilde{s}_j$ is a holomorphic function on the open unitary disc centered on the origin. That implies that
\[
\widetilde{s}_j(x)= \sum_{i=0}^{\infty} \lambda_{i, j} x^{i}=e^{\sum_{i=0}^{\infty} t_{i, j} x^{i}}, \qquad x \in  \left\{\left|\zeta\right|<1\right\}, \qquad j=1, \ldots, p.
\]
Notice that
\[
\left| \sum_{i=0}^{\infty} \lambda_{i,j} \right|=\lim_{\overset{x \to 1}{x \in [0,1)}} \left|\sum_{i=0}^{\infty} \lambda_{i,j} x^i\right|=\lim_{\overset{x \to 1}{x \in \overset{\circ}{\Delta}_1}} \left|\int \frac{\d s_{j}(t)}{1-xt}\right|< +\infty,
\]
So the first equality in (\ref{taylor}) is proved. The second one comes immediately from the fact that the functions $\widetilde{s}_j$ do not vanish on $\Delta_0.$ That implies that $\sum_{i=0}^{\infty} t_{i,j} x^{i}$ are also bounded and therefore continuous. Hence we can proceed analogously as in the first equality.
\end{proof}

\subsubsection{The inverse problem}
Given the series
\begin{align}\label{theform}
w_j(x)=\sum_{i=0}^{\infty} \lambda_{i, j} x^{i}=e^{\sum_{i=0}^{\infty} t_{i, j} x^{i}}, \qquad x \in  \Delta_0, \qquad
j=1, \ldots, p,
\end{align}
we consider the problem of finding  conditions over $\{\lambda_{i,j}\}$ such that the set of series $\{w_j\}_{j=1}^p$ form a M-Nikishin system of functions. The reader should notice that $\lambda_{i,j}=S_i(t_{i,0},t_{i,1},\dots,t_{i,j})$
where $S_i$ is the $i$-th elementary Schur polynomial. Elementary Schur polynomials $S_j(t_1,\dots,t_j)$ are defined
by the following generating relation $\exp({\sum_{j=1}^{\infty}t_{j}z^j})=\sum_{j=0}^{\infty}S_j(t_1,t_2,\dots,t_j)
z^j$, and therefore
$S_j=\sum_{p=1}^j\sum_{j_1+\cdots+j_p=j} t_{j_1}\cdots t_{j_p }$. Given a partition $\vec n=(n_1,\dots,n_r)\in\Z_+^r$
we have the Schur function $s_{\vec n}(t) = \det(S_{n_i-i+j}(t))_{1\leq i,j\leq r}$. For more on the relation of these Schur functions and those in \cite{macdonald}, see \cite{orlov}.

In order to state sufficient conditions in this direction we need some preliminary definitions and results.

\begin{definition}\label{def:mp} Given a sequence  $C=\{c_i\}_{i =0}^\infty\subset \R$ its Hausdorff moment problem consists in finding a measure $\sigma \in {\mathcal{M}}(\Delta)$ such that 
\begin{align*}
c_i=\int \zeta^i \d \sigma (\zeta),\qquad i\in{\mathbb{Z}}_+.
\end{align*}
Moreover, if we further impose the  constraint  $\sigma \in {\mathcal{M}}_0(\Delta)$ we say that we have a restricted Hausdorff moment problem.
\end{definition}

Here we have made a variation in the classical definition of a Hausdorff problem, where the solutions are positive measures. In our Hausdorff problem we look for measures in a wider class where they do not change their sign. Obviously, since ${\mathcal{M}}_0(\Delta) \subset {\mathcal{M}}(\Delta)$ each solution of a restricted Hausdorff problem is also a solution of a Hausdorff problem. In the pages 8 and 9 in  \cite{ShTm}  J. A. Shohat and  J. D. Tamarkin study Hausdorff problems and give a sufficient and necessary condition over the sequences to have solution. Using this result we deduce the following Lemma.

\begin{lemma}\label{hpc} The Hausdorff moment problem for a sequence $C=\{c_i\}_{i =0}^\infty\subset \R$ has a solution if and only  if
\begin{align}\label{hpc1}
\sum_{i=0}^n\binom{n}{i} (-1)^ic_{i+k} &\geq 0 &\forall (n,k)&\in\mathbb{Z}^2_+&&\text{or } &\sum_{i=0}^n\binom{n}{i} (-1)^ic_{i+k} &\leq 0 &\forall (n,k)&\in\mathbb{Z}^2_+.
\end{align}
When  \eqref{hpc1} holds a necessary and sufficient condition that ensures solution for the restricted Hausdorff moment problem of $C$ is
\begin{equation}\label{hpc2}
\Big|\sum_{i=0}^{\infty} c_{i}\Big| < +\infty.
\end{equation}
\end{lemma}

\begin{proof} Theorem 1.5 in \cite{ShTm} states that the first set of inequalities in \eqref{hpc1} is a necessary and sufficient condition to have a positive measure $\sigma$ solving the classical Hausdorff problem. Following their proof it is not hard to conclude that adding the second set of inequalities leads to a solution in ${\mathcal{M}}(\Delta)$.  Let us take a measure $\sigma \in {\mathcal{M}}(\Delta)$ and observe that $\int \frac{\d \sigma(t)}{1-xt}$ is a holomorphic function on $\bar {\mathbb{C}} \setminus [1,+\infty)$, then if
 $C$ is its moment sequence we deduce
\begin{align*}
\int_{\Delta} \frac{\d \sigma(t)}{1-xt}=\sum_{i=0}^{\infty} c_i x^i, \qquad  x \in \left\{ \left| \zeta \right| < 1\right\}.
\end{align*}
Thus, since all the $c_i$'s have the same sign, by Lebesgue's dominated convergence Theorem we have
\begin{align*}
\lim_{\underset{x \in [0,1)}{x \to 1}}\Big|\int_{\Delta} \frac{\d \sigma(t)}{1-xt}\Big|= \lim_{\underset{x \in[0,1)}{x \to 1}} \Big|\sum_{i=0}^{\infty}x^ic_i \Big|=\Big|\sum_{i=0}^{\infty} c_i\Big|.
\end{align*}
Thus $\sigma \in {\mathcal{M}}_0(\Delta)$ if and only if  \eqref{hpc2} takes place.
\end{proof}

Given the series
\begin{align}\label{w1j}
w_j(x)=\sum_{i=0}^{\infty} \lambda_{i, j,1} x^{i}, \qquad x \in  \Delta_1, \qquad j=1, \ldots, p,
\end{align}
we introduce  a set of semi-infinite matrices $\Theta_k$ and semi-infinite vectors $\theta_{j,k}$, $j=k,\ldots,p,$
$k=1,\ldots,p$ in  the following recursive way.
 First, we define
\begin{align*}
 \Theta_1&:=
\begin{pmatrix}
\lambda_{0,1,1} &\lambda_{1,1,1}&\lambda_{2,1,1}&\cdots\\
\lambda_{1,1,1}&\lambda_{2,1,1}&\lambda_{3,1,1}&\cdots\\
\lambda_{2,1,1}&\lambda_{3,1,1}&\lambda_{4,1,1}&\cdots\\
\vdots        &  \vdots        &\vdots       &\ddots
\end{pmatrix},& \theta_{j,1}&:=\begin{pmatrix}
\lambda_{0,j,1}\\
\lambda_{1,j,1}\\
\lambda_{2,j,1}\\
\vdots
\end{pmatrix}, & j&=1,\dots,p.
\end{align*}
Then, we seek  solutions $\theta_{j,2}:=\left(\begin{smallmatrix}
\lambda_{0,j,2}\\
\lambda_{1,j,2}\\
\lambda_{2,j,2}\\
\vdots
\end{smallmatrix}\right)$ of $\Theta_1\theta_{j,2}=\theta_{j,1}$, for $j=2,\dots,p$, and if these solutions exist we define
$
\Theta_2:=\left(
\begin{smallmatrix}
 \lambda_{0,2,2} &\lambda_{1,2,2}&\lambda_{2,2,2}&\cdots\\
\lambda_{1,2,2}&\lambda_{2,2,2}&\lambda_{3,2,2}&\cdots\\
\lambda_{2,2,2}&\lambda_{3,2,2}&\lambda_{4,2,2}&\cdots\\
\vdots        &  \vdots        &\vdots       & \ddots
\end{smallmatrix}\right)$. Then, we look for  $\theta_{j,3}=\left(\begin{smallmatrix}
\lambda_{0,j,3}\\
\lambda_{1,j,3}\\
\lambda_{2,j,3}\\
\vdots
\end{smallmatrix}\right)$ which solves $\Theta_2\theta_{j,3}=\theta_{j,2}$, for $j=3,\dots,p$, and when such solutions exist we
introduce
$\Theta_3=
\left(\begin{smallmatrix}
\lambda_{0,3,3} &\lambda_{1,3,3}&\lambda_{2,3,3}&\cdots\\
\lambda_{1,3,3}&\lambda_{2,3,3}&\lambda_{3,3,3}&\cdots\\
\lambda_{2,3,3}&\lambda_{3,3,3}&\lambda_{4,3,3}&\cdots\\
\vdots        &  \vdots        &\vdots       &
\end{smallmatrix}\right)$. In this way we get for $k \in \{1,\ldots,p\}$ the matrices $\Theta_k$ and  vectors $\theta_{j,k},$
$j=k, \dots,p$, linked by  $\Theta_k \theta_{j,k+1}=\theta_{j,k}$ with expressions
\begin{align*}
 \Theta_{k+1}&=
\begin{pmatrix}
\lambda_{0,k+1,k+1} &\lambda_{1,k+1,k+1}&\lambda_{2,k+1,k+1}&\cdots\\
\lambda_{1,k+1,k+1}&\lambda_{2,k+1,k+1}&\lambda_{3,k+1,k+1}&\cdots\\
\lambda_{2,k+1,k+1}&\lambda_{3,k+1,k+1}&\lambda_{4,k+1,k+1}&\cdots\\
\vdots        &  \vdots        &\vdots       & \ddots
\end{pmatrix},&
\theta_{j,k+1}&=\begin{pmatrix}
 \lambda_{0,j,k+1}\\
\lambda_{1,j,k+1}\\
\lambda_{2,j,k+1}\\
\vdots
\end{pmatrix}.
\end{align*}
Here we understand $\Theta_k \theta_{j,k+1}=\theta_{j,k}$ as
\begin{align*}
  \sum_{i=0}^{\infty}\lambda_{l+i,k,k} \lambda_{i,j,k+1}=\lambda_{l,j,k}, \quad l\in {\mathbb{Z}}_+.
\end{align*}
We now consider the sequences
 \begin{align}\label{ckj}
  C_{k,k}&:=\{\lambda_{i,k,k}\}_{i=0}^\infty, &C_{j,k}&:=\{\lambda_{i,j,k}\}_{i=0}^\infty,  &j&=k,\ldots,p,&k&=1, \ldots,p.
 \end{align}
Later, we will prove that none of the semi-infinite Hankel matrices $\Theta_k$, $k=1,\dots,p$, are invertible. Hence such infinite linear systems are either undetermined or incompatible. In this last  case we say that the systems of
 sequences $(C_{k,k},\dots, C_{p,k})$, $k=1, \dots,p$, do not exist.

 First we need the following  preliminary
\begin{lemma}\label{lem:1} The series
\[
w(x)=\sum_{i=0}^{\infty} \lambda_{i} x^{i}, \qquad x \in  \Delta_0,
\]
converges uniformly on $\Delta_0$ to a function $\widetilde{\sigma}(x)=\int \d \sigma(t)/(1-tx)$ corresponding to a measure $\sigma \in {\mathcal{M}}_0(\Delta_1)$ on $\Delta_0$ if and only if the restricted Hausdorff moment problem corresponding to the sequence $\{\lambda_i: i\in {\mathbb{Z}}_+\}$ has a solution.
\end{lemma}
\begin{proof}
Let us assume that the restricted Hausdorff moment problem of a sequence $\{\lambda_i: i\in {\mathbb{Z}}_+\}$ has a solution. That means that there exists a measure $\sigma \in {\mathcal{M}}_0(\Delta)$ such that
\begin{align*}
 \lambda_i&=\int_{\Delta_1} t^i \d \sigma(t), & i&\in {\mathbb{Z}}_+,&\lim_{\underset{x \in [0,1)}{x \to 1}}\left| \int_{\Delta_1} \frac{\d \sigma (t)}{1-tx}\right|&=\left|\sum_{i=0}^{\infty} \lambda_i
\right| < +\infty.
\end{align*}
Since $|\lambda_i x^i| \leq |\lambda_i|$, $|x|\leq 1$ and $\sum_{i=0}^{\infty} |\lambda_i| < + \infty,$ by Weirestrass' Theorem $\sum_{i=0}^{\infty} \lambda_ix^i$ converges uniformly on $\Delta_0.$
 This proves the \emph{if} implication in the  Lemma \ref{lem:1}. On the other hand
 \[
 \lim_{\overset{x \to 1}{x \in [0,1)}} \left|\int_{\Delta} \frac{\d \sigma(t)}{1-tx}\right|= \lim_{\overset{x \to 1}{x \in [0,1)}}\left|\sum_{i=0}^{\infty} \lambda_i x^i\right|=\left|\sum_{i=0}^{\infty} \lambda_i \right|
 \]
 because $\left|\sum_{i=0}^{\infty} \lambda_i x^i\right|$ must be continuous on $\Delta_0$. $\lambda_{i}$ coincides  with the $i$-th moment corresponding to the measure $\sigma$ which completes the proof.
\end{proof}

\begin{theorem}\label{suffcondnik}
The system of weights  $\{w_{1,j}\}_{j=1}^p$, as in \eqref{w1j}, converges uniformly in $\Delta_0$ to  an M-Nikishin system of functions $\{\widehat{s}_j\}_{j=1}^p$ if and only if  for each $k=1,\dots, p-1$, there exists a system of sequences $(C_{k,k},\dots, C_{p,k})$  as in \eqref{ckj}, such that  their restricted Hausdorff moment problems have  solutions.
\end{theorem}

\begin{proof}
The proof of this Theorem goes as follows. From Lemma \ref{lem:1} we have that for each $j=1,\ldots,p,$
\[
w_{j,1}(x)=\sum_{i=0}^{\infty} \lambda_{i, j,1} x^{i}, \qquad x \in  \Delta_0,
\]
converges in $\Delta_0$ to a function $\widetilde{s}_{j}(x)=\int \d s_{j}(t)/(1-tx)$ if and only if the restricted Hausdorff moment problem corresponding to  $\{\lambda_{i,j,1}: i\in
{\mathbb{Z}}_+\}$ has a solution. We assume that $w_{j,1}$ converges uniformly on $\Delta_0$ to the function $\widetilde{s}_j$ corresponding to the $s_j \in {\mathcal{M}}(\Delta_1).$ In order to prove the necessity in Theorem's statement we suppose that $(s_1, \ldots, s_p)={\mathcal{MN}}(\sigma_1, \ldots, \sigma_p)$ is an M-Nikishin system of measures as it was defined in the \S \ref{perfectness}. Fixed $k \in \{1, \ldots,p\}$ we define another M-Nikishin system $(s_{k,k},\ldots,s_{k,p})={\mathcal{MN}}(\sigma_k,\ldots,\sigma_p).$ Let us observe that $(s_{1,1},\ldots, s_{1,p})=(s_{1},\ldots, s_{p}).$

By construction for each $k\in \{1,\ldots,p\},$ we have that $\d s_{k,j}= \widetilde{s}_{k+1,j} \d s_{k,k},$ $j=k,\ldots,p.$ When $j=k$ we understand $\widetilde{s}_{k+1,k}\equiv 1.$ Fixed $j\in \{k+1, \ldots,p\},$ $\widetilde{s}_{k+1,j}$ is  a holomorphic function on $\bar{\mathbb{C}} \setminus (-\infty,1)$;  hence, its Taylor's series
\[
w_{j,k+1}(t)=\sum_{i=0}^{\infty} \lambda_{i, j,k+1} t^{i}, \qquad t \in  \Delta_1 \subset \Delta_0
\]
converges uniformly to $\widetilde{s}_{k+1,j}$ on $\Delta_1$. Then, for each $x \in \Delta_1$
\[
\widetilde{s}_{k,j}(x)=\sum_{l=0}^{\infty} \lambda_{l, j,k} x^{l}=\int_\R  \widetilde{s}_{k+1,j}(t) \frac{\d
s_{k,k}(t)}{1-tx}=\int_\R \sum_{i=0}^{\infty} \lambda_{i, j,k+1} t^{i} \frac{\d s_{k,k}(\zeta)}{1-tx}=
\]
\[
=\sum_{l=0}^{\infty} \sum_{i=0}^{\infty} \lambda_{i, j,k+1}  \int_\R  t^{i+l} \d s_{k,k}(t)=\sum_{l=0}^{\infty}
\sum_{i=0}^{\infty} \lambda_{i, j,k+1}  \lambda_{l+i,k,k} x^{l},
\]
which proves one implication of the equivalence. The other implication comes immediately from Lemma \ref{lem:1}.

\end{proof}

We remark from the statements of Lemma \ref{hpc} that the conditions in Theorem \ref{suffcondnik} are equivalent to the inequalities in (\ref{hpc1}). Hence, by continuity criteria, such conditions are stable under perturbations of the coefficients $\lambda_{i,1,1},$ $i \in {\mathbb{Z}}_+.$ We will come to this later in \S 3, when we consider deformations of the weights leading to the multicomponent 2D Toda flows in the precise form discussed in this section.

\subsection{The Gauss--Borel factorization and multiple orthogonal polynomials}
Given a perfect combination  $(\mu, \vec w_1,\vec w_2)$ we consider
\cite{adler-van moerbeke}
\begin{definition}
  The  Gauss--Borel factorization  (also known as   $LU$  factorization) of a semi-infinite moment matrix $g$, determined by $(\mu, \vec w_1,\vec w_2)$,  is the
  problem of finding the solution of
\begin{align}\label{facto}
  g&=S^{-1}\bar S, & S&=\begin{pmatrix}
    1&0&0&\cdots \\
    S_{1,0}&1&0&\cdots\\
    S_{2,0}&S_{2,1}&1&\cdots\\
    \vdots&\vdots&\vdots&\ddots
    \end{pmatrix}, &
   \bar S^{-1}&=\begin{pmatrix}
     \bar S_{0,0}'&\bar S_{0,1}'&\bar S_{0,2}'&\cdots\\
     0&\bar S_{1,1}'&\bar S_{1,2}'&\cdots\\
     0&0&\bar S_{2,2}'&\cdots&\\
     \vdots&\vdots&\vdots&\ddots
   \end{pmatrix}, & S_{i,j},\bar S'_{i,j}&\in\R.
\end{align}
In terms of these matrices we construct the polynomials
    \begin{align} \label{defmops}
A^{(l)}_a:={\sum}'_iS_{l,i}x^{k_1(i)},
\end{align}
where the sum $\sum'$ is taken for a fixed $a=1,\dots,p_1$ over those $i$
such that $a=a_1(i)$ and $i\leq l$. We also construct the dual polynomials
\begin{align} \label{defdualmops}
\bar A^{(l)}_b:={\sum}'_jx^{k_2(j)}\bar S_{j,l}',
\end{align}
where the sum $\sum'$ is taken for a given $b$ over those $j$ such that $b=a_2(j)$ and $j\leq l$.
\end{definition}

This factorization makes sense whenever all the principal minors of $g$ do not vanish, i.e., if $\det g^{[l]}\neq 0$ $l=1,2,\dots,$ and in our case it is true because $(\mu, \vec w_1,\vec w_2)$  is a perfect
 combination. It can be shown that the following sets
 \begin{align*}
   G_-&:=\Big\{S=\left(\begin{smallmatrix}
    1&0&0&\cdots \\
    S_{1,0}&1&0&\cdots\\
    S_{2,0}&S_{2,1}&1&\\
    \vdots&\vdots&\vdots&\;\ddots
    \end{smallmatrix}\right), S_{i,j}\in\R  \Big\}, & G_+:=\Big\{ \bar S=\left(\begin{smallmatrix}
     \bar S_{0,0}&\bar S_{0,1}&\bar S_{0,2}&\cdots\\
     0&\bar S_{1,1}&\bar S_{1,2}&\cdots\\
     0&0&\bar S_{2,2}&\cdots&\\
     \vdots&\vdots&\vdots&\;\ddots
   \end{smallmatrix}\right), \bar S_{i,j}\in\R,\bar S_{i,i}\neq 0 \Big\}
 \end{align*}
 are groups. Indeed, the multiplication  of two arbitrary semi-infinite matrices is, in general, not well defined as it involves, for each coefficient of the product,  a series; however if the two matrices lie on $G_-$, the mentioned series collapses into a finite sum, and the same holds for $G_+$. Moreover, the inverse of a matrix in $S\in G_-$ can be found to be in $G_-$ in a recursive way: first we express $S=\I+\sum_{i>0} S_i( \Lambda^\top)^i$ with $S_i=\diag (S_i(0),S_i(1),\dots)$ a diagonal matrix, then we assume $S^{-1}=\I+\sum_{i>0} \tilde S_i( \Lambda^\top)^i$ to have the same form, and finally  we find that the diagonal matrix unknown coefficients $\tilde S_i$ are expressed in terms of $S_0,\dots, S_i$ in a unique way; the same holds in $G_+$. Given, two elements $S\in G_-$ and $\bar S\in G_+$ the coefficients of the product $S\bar S$ are  finite sums. However,  this is not the case for $\bar S S$,  where the coefficients are  series. Therefore, given an  $LU$ factorizable element $g=S^{-1}\bar S$ we can not ensure that $g$ has an inverse, observe  that in spite of the existence of $S$ and $\bar S^{-1}$, the existence of $\bar S^{-1} S=g^{-1}$ is not ensured as this product involves the evaluation of series instead of finite sums.

 With the use of the coefficients of the matrices $S$ and $\bar S$ we construct multiple orthogonal polynomials of mixed type with  normalizations of type I and II
  \begin{pro}
    We have the following identifications
    \begin{align*}
A^{(l)}_a&=A_{[\vec\nu_{1}(l);\vec\nu_{2}(l-1)],a}^{(\rII,a_1(l))}, &
\bar A^{(l)}_b&= A^{(\rI,a_1(l))}_{[\vec\nu_{2}(l);\vec\nu_{1}(l-1)],b},
\end{align*}
in terms of multiple orthogonal polynomials of mixed type with two  normalizations $\rI$ and $\rII$, respectively.
  \end{pro}
 \begin{proof}
  From the $LU$ factorization we deduce
\begin{align}\label{orth-1}
  \sum_{i=0}^{l}S_{l,i}g_{i,j}&=0,& j&=0,1,\dots,l-1,& S_{ii}&:=1.
\end{align}
With the aid of \eqref{mult:u}  and \eqref{defmops} 
 we express \eqref{orth-1} as follows
\begin{align}\label{orthogonality}
 \int \Big(\sum_{a=1}^{p_1}A^{(l)}_{a}(x)w_{1,a}(x)\Big)w_{2,b}(x)x^{k}\d \mu(x)&=0,& \deg
 A^{(l)}_{a}&\leq\nu_{1,a}(l)-1,\\ & &0\leq k&\leq
 \nu_{2,b}(l-1)-1.\notag
\end{align}
    We recognize these equations as those defining a set of multiple orthogonal polynomials of mixed type as
  discussed in \cite{daems-kuijlaars}. This fact leads to
   $A_{[\vec\nu_{1};\vec\nu_{2}],a} := A^{(l)}_{a}$ where $\vec\nu_1=\vec\nu_1(l)$ and $\vec\nu_2=\vec\nu_2(l-1)$.
   Observe that for a given $l$ each polynomial $A_{[\vec\nu_{1};\vec\nu_{2}],a}$ has at much  $\nu_{1,a}(l)$
   coefficients, and therefore we have
   $|\vec\nu_{1}(l)|=l+1$ unknowns, while we have $|\vec \nu_{2}(l-1)|=l$ equations.
 Moreover, from the normalization condition $S_{ii}=1$ we get that the
    polynomial $A_{[\vec\nu_{1};\vec\nu_{2}],a_1(l)}$ is monic with $\deg
    A_{[\vec\nu_{1};\vec\nu_{2}],a_1(l)}=\nu_{1,a_1(l)}(l)-1=k_1(l+1)-1$, so that we are
    dealing with a type II normalization and therefore we can write
    $A^{(l)}_{a}=A_{[\vec\nu_{1};\vec\nu_{2}],a}^{(\rII,a_1(l))}$.

Dual equations to \eqref{orth-1} are
\begin{align}\label{orth-2-1}
  \sum_{j=0}^{l}g_{i,j}\bar S_{j,l}'&=0,& i&=0,1,\dots,l-1,\\\label{orth-2-2}   \sum_{j=0}^{l}g_{l,j}\bar S_{j,l}'&=1.
\end{align}
Now, using again \eqref{mult:u} and \eqref{defdualmops}, \eqref{orth-2-1} becomes
\begin{align}\label{orthogonality dual}
\int \Big(\sum_{b=1}^{p_2}\bar  A^{(l)}_{b}
(x)w_{2,b}(x)\Big)w_{1,a}(x)x^{k}\d \mu(x)&=0,& \deg \bar A^{(l)}_{b}&\leq\nu_{2,b}(l)-1,\\ & &0\leq k&\leq
\nu_{1,a}(l-1)-1,\notag
\end{align}
while \eqref{orth-2-2} reads
\begin{align}
\int_{\R} \Big(\sum_{b=1}^{p_2}\bar A^{(l)}_{b}
(x)w_{2,b}(x)\Big)w_{1,a_1(l)}(x)x^{k_1(l)}\d \mu(x)&=1,&
\end{align}
where using \eqref{formulitas} we obtain
\begin{align}
  k_1(l)=\nu_{1,a_1(l)}(l-1).
\end{align}
As above we are dealing with multiple orthogonal polynomials and therefore  $\bar
A^{(l)}_{b}=\bar A_{[\vec\nu_{2};\vec\nu_{1}],b}$, with
$\vec\nu_1=\vec\nu_1(l-1)$ and $\vec\nu_2=\vec\nu_2(l)$, which now happens to have a  normalization  of type I and
consequently we  write
$\bar A^{(l)}_{b}=\bar A^{(\rI,a_1(l))}_{[\vec\nu_{2};\vec\nu_{1}],b}$.
\end{proof}

Given a definite sign finite Borel measure the corresponding set of monic orthogonal polynomials $\{p_l\}_{l=0}^\infty$, $\deg p_l=l$, can be viewed as a ladder of polynomials, in which to get up to  a
given degree one needs to ascend $l$ steps in the ladder. For multiple orthogonality the situation is different  as we have, instead of a chain, a multi-dimensional lattice of
degrees.
Let us consider a perfect combination $(\mu,\vec w_1,\vec w_2)$ and the corresponding  set of multiple orthogonal
polynomials $\{A_{[\vec \nu_1;\vec \nu_2],a}\}_{a=1}^{p_1}$, with degree vectors such that $|\vec\nu_1|=|\vec
\nu_2|+1$.
  There always exists compositions $\vec n_1,\vec n_2$  and an integer $l$ with $|\vec \nu_1|=l+1$  and $|\vec \nu_2|=l$
  such that the polynomials  $\{A_a^{(l)}\}_{a=1}^{p_1}$ coincides with $\{A_{[\vec \nu_1;\vec
  \nu_2],a}\}_{a=1}^{p_1}$. Therefore, the set of sets of multiple orthogonal polynomials
  $\big\{\{A_a^{(k)}\}_{a=1}^{p_1}, k=0,\dots,l\big\}$, can be understood as a ladder leading to the desired set of
  multiple orthogonal polynomials $\{A_{[\vec \nu_1;\vec \nu_2],a}\}_{a=1}^{p_1}$ after ascending $l$ steps in the
  ladder, very much in same style as in standard orthogonality (non multiple) setting. The ladder can be identified with the
  compositions $(\vec n_1,\vec n_2)$.
     However, by no means there is always a unique ladder to achieve this, in general there are several compositions
     that do the job. A particular ladder, which we refer to as the simplest $[\vec \nu_1;\vec\nu_2]$  ladder, is
     given by the choice $ \vec n_1=\vec \nu_1$ and $\vec n_2=\vec \nu_2+\vec e_{2,p_2}$. Many of the  expressions that will be
     derived later on in this paper for multiple orthogonal polynomials and second kind functions only depend on the
     integers $(\vec \nu_1,\vec \nu_2)$ and not on the particular ladder chosen, and therefore compositions, one uses to reach to it.

\subsection{Linear forms and multiple bi-orthogonality}

We introduce  linear forms associated with multiple orthogonal polynomials as follows
\begin{definition}
  \label{def:linear forms}
  Strings of linear forms and dual linear forms associated with  multiple ortogonal polynomials and their duals  are
  defined by
  \begin{align}\label{linear forms S}
Q:=  \begin{pmatrix}
    Q^{(0)}\\
    Q^{(1)}\\
    \vdots
  \end{pmatrix}&=S\xi_{1},&
  \bar Q:=\begin{pmatrix}
    \bar Q^{(0)}\\
    \bar Q^{(1)}\\
    \vdots
  \end{pmatrix}&=(\bar S^{-1})^\top\xi_{2},
\end{align}
\end{definition}
It can be immediately checked that
\begin{pro}
  \label{proposition: linear forms and mops}
  The linear  forms and their duals, introduced in Definition \ref{def:linear forms}, are given by
\begin{align}\label{linear.forms}
 Q^{(l)}(x)&:= \sum_{a=1}^{p_1}A^{(l)}_{a}(x)w_{1,a}(x),&
\bar Q^{(l)}(x)&:= \sum_{b=1}^{p_2}\bar  A^{(l)}_{b}
(x)w_{2,b}(x).
\end{align}
\end{pro}
Sometimes we use the alternative   notation $Q^{(l)}=Q_{[\vec\nu_{1};\vec\nu_{2}]}$ and
 $\bar Q^{(l)}=\bar Q_{[\vec\nu_{2};\vec\nu_{1}]}$.  It is also trivial to check the following
 \begin{pro}\label{proposition: mo and linear forms}
The orthogonality relations
 \begin{align}\label{linear.form.orthogonality}
 \begin{aligned}
  \int  Q^{(l)}(x)w_{2,b}(x)x^k\d \mu(x)&=0,&0&\leq k\leq \nu_{2,b}(l-1)-1,&b&=1,\dots,p_2,\\
      \int  \bar Q^{(l)}(x)w_{1,a}(x)x^k\d \mu(x)&=0,&0&\leq k\leq \nu_{1,a}(l-1)-1,&a&=1,\dots,p_1,
 \end{aligned}
   \end{align}
   are fulfilled.
   \end{pro}
   Moreover, we have that these linear forms are bi-orthogonal
   \begin{pro}\label{proposition: biorthogonality}
     The following multiple bi-orthogonality relations among linear forms and their duals
\begin{align}\label{biotrhoganility}
  \int  Q^{(l)}(x)\bar Q^{(k)}(x)\d \mu(x)&=\delta_{l,k},& l,k\geq 0,
\end{align}
hold.
   \end{pro}
   \begin{proof}
Observe that
\begin{align*}
  \int_R Q(x)\bar Q(x)^\top \d \mu(x)&=\int  S\xi_1(x)\xi_2(x)^\top\bar S^{-1}\d \mu(x)&&\text{from \eqref{linear
  forms S}} \\
  &=S\Big(\int  \xi_1(x)\xi_2(x)^\top\d \mu(x)\Big)\bar S^{-1} \\
  &=Sg\bar S^{-1}&&\text{from \eqref{compact.g}} \\
  &=\I.&&\text{from \eqref{facto}}
  \end{align*}
   \end{proof}
   \begin{definition}
     \label{definiton: truncated}
   Denote by $\xi_i^{[l]}$, $i=1,2$  the truncated  vector formed with the first $l$ components of $\xi_i$.
   \end{definition}
We are ready to give different expressions for these linear forms and their duals
\begin{pro}
  \label{proposition: expressions for linear forms}
  The linear forms can be expressed  in terms of the moment matrix in the following  different ways
  \begin{align}
  Q^{(l)}&=\xi_{1}^{(l)}-\begin{pmatrix}
    g_{l,0}&g_{l,1}&\cdots&g_{l,l-1}
  \end{pmatrix}(g^{[l]})^{-1}\xi_1^{[l]}\notag\\
  &=\bar S_{l,l}\begin{pmatrix}
    0 &0 &\cdots &0& 1
  \end{pmatrix}
  (g^{[l+1]})^{-1}\xi^{[l+1]}_1 \notag \\&=\frac{1}{\det g^{[l]}}\det
  \left(\begin{BMAT}{cccc|c}{cccc|c}
    g_{0,0}&g_{0,1}&\cdots&g_{0,l-1}&\xi_1^{(0)}\\
     g_{1,0}&g_{1,1}&\cdots&g_{1,l-1}&\xi_1^{(1)}\\
     \vdots &\vdots&            &\vdots&\vdots\\
        g_{l-1,0}&g_{l-1,1}&\cdots&g_{l-1,l-1}&\xi_1^{(l-1)}\\
          g_{l,0}&g_{l,1}&\cdots&g_{l,l-1}&\xi_1^{(l)}
\end{BMAT}\right),& l&\geq 1,
\end{align}
and the dual linear forms as
\begin{align}
\bar  Q^{(l)}&=(\bar S_{l,l})^{-1}\Big(\xi_{2}^{(l)}-(\xi_2^{[l]})^\top(g^{[l]})^{-1}\begin{pmatrix}
    g_{0,l}\\g_{1,l}\\\vdots\\g_{l-1,l}
  \end{pmatrix}\Big)\notag\\
  &=(\xi^{[l+1]}_2)^\top
  (g^{[l+1]})^{-1}\begin{pmatrix}
    0 \\0\\\vdots \\0\\ 1
  \end{pmatrix} \notag\\&=\frac{1}{\det g^{[l+1]}}\det
  \left(\begin{BMAT}{cccc|c}{cccc|c}
    g_{0,0}&g_{0,1}&\cdots&g_{0,l-1}& g_{0,l}\\
     g_{1,0}&g_{1,1}&\cdots&g_{1,l-1}& g_{1,l} \\
     \vdots &\vdots&        &\vdots&\vdots\\
        g_{l-1,0}&g_{l-1,1}&\cdots&g_{l-1,l-1}&g _{l-1,l}\\
        \xi_2^{(0)}& \xi_2^{(1)} & \cdots & \xi_2^{(l-1)} & \xi_2^{(l)}
\end{BMAT}\right),& l\geq 0.
\end{align}
\end{pro}
\begin{proof}
See Appendix \ref{II}.
\end{proof}
As a consequence we get different expressions for the multiple orthogonal polynomials and their duals
\begin{cor}\label{cormops}
The multiple orthogonal polynomials and their duals have the following alternative expressions
  \begin{align}
  A^{(l)}_a&=\chi_{1,a}^{(l)}-\begin{pmatrix}
    g_{l,0}&g_{l,1}&\cdots &g_{l,l-1}
  \end{pmatrix}(g^{[l]})^{-1}\chi_{1,a}^{[l]}\notag\\&=\bar S_{l,l}\begin{pmatrix}
    0 &0 &\dots &0& 1
  \end{pmatrix}  (g^{[l+1]})^{-1}\chi_{1,a}^{[l+1]}\notag\\
    \label{detmops} &=\frac{1}{\det g^{[l]}}\det
  \left(\begin{BMAT}{cccc|c}{cccc|c}
    g_{0,0}&g_{0,1}&\cdots&g_{0,l-1}&\chi_{1,a}^{(0)}\\
     g_{1,0}&g_{1,1}&\cdots&g_{1,l-1}&\chi_{1,a}^{(1)}\\
     \vdots &\vdots&            &\vdots&\vdots\\
        g_{l-1,0}&g_{l-1,1}&\cdots&g_{l-1,l-1}&\chi_{1,a}^{(l-1)}\\
          g_{l,0}&g_{l,1}&\cdots&g_{l,l-1}&\chi_{1,a}^{(l)},
\end{BMAT}\right),& l&\geq 1.
\end{align}
and
\begin{align} \label{dualmops.2}
  \bar A^{(l)}_b&=
  (\bar S_{l,l})^{-1}\Big(\chi_{2,b}^{(l)}-(\chi_{2,b}^{[l]})^\top(g^{[l]})^{-1}\begin{pmatrix}
    g_{0,l}\\g_{1,l}\\\vdots\\g_{l-1,l}
  \end{pmatrix}\Big)  \\
  &=(\chi^{[l+1]}_{2,b})^\top
  (g^{[l+1]})^{-1}\begin{pmatrix}
    0 \\0\\\vdots \\0\\ 1
  \end{pmatrix}\notag\\
\label{detdualmops} &=\frac{1}{\det g^{[l+1]}}\det
  \left(\begin{BMAT}{cccc|c}{cccc|c}
    g_{0,0}&g_{0,1}&\cdots&g_{0,l-1}& g_{0,l}\\
     g_{1,0}&g_{1,1}&\cdots&g_{1,l-1}& g_{1,l} \\
     \vdots &\vdots&   &\vdots&\vdots\\
        g_{l-1,0}&g_{l-1,1}&\cdots&g_{l-1,l-1}&g _{l-1,l}\\
        \chi_{2,b}^{(0)}& \chi_{2,b}^{(1)} & \dots & \chi_{2,b}^{(l-1)} & \chi_{2,b}^{(l)}
\end{BMAT}\right), & l\geq 0.
\end{align}
\end{cor}

Observe that \eqref{s'}, Appendix \ref{II}, implies
\begin{align}
  \bar S_{l,l}=\frac{\det g^{[l+1]}}{\det g ^{[l]}}.
\end{align}

\subsection{Functions of the second kind}

The Cauchy transforms of the linear forms \eqref{linear.forms} play a crucial role in the
Riemann--Hilbert problem associated with the multiple orthogonal polynomials of mixed type
\cite{daems-kuijlaars}. Following  the approach of  Adler and van Moerbeke we will show that these Cauchy transforms are also related to the $LU$ factorization
considered in this paper.

Observe that the construction of multiple orthogonal polynomials  performed so far is synthesized in the following
strings of multiple orthogonal polynomials and their duals
\begin{align}\label{mop-s}
\begin{aligned}
\A_{a}&:=  \begin{pmatrix}
    A^{(0)}_{a}\\
    A^{(1)}_{a}\\
    \vdots
  \end{pmatrix}=S\chi_{1,a},&
  \bar\A_{b}&:= \begin{pmatrix}
 \bar A^{(0)}_{b}\\
    \bar A^{(1)}_{b}\\
    \vdots
  \end{pmatrix}=(\bar S^{-1})^\top\chi_{2,b},& a&=1,\dots,p_1,& b&=1,\dots,p_2.
\end{aligned}
\end{align}
In order to complete these formulae and in terms of $\chi^*$ as in \eqref{chion}  we consider
\begin{definition}
Let us introduce the  following formal semi-infinite vectors
\begin{align}\label{cauchy-S}
\begin{aligned}
  \Cs_b&=\begin{pmatrix}
    C_b^{(0)}\\C_b^{(1)}\\\vdots
  \end{pmatrix}=\bar S\chi_{2,b}^*(z),&
  \bar\Cs_a&=\begin{pmatrix}
    \bar C_a^{(0)}\\\bar C_a^{(1)}\\\vdots
  \end{pmatrix}=(S^{-1})^\top\chi_{1,a}^*(z),& b&=1,\dots,p_2,& a&=1,\dots,p_1,
\end{aligned}\end{align}
that we call strings of second kind functions.
\end{definition}
These objects are actually Cauchy transforms of the linear forms $Q^{(l)}$, $l \in {\mathbb{Z}}_+$, whenever the series converge and outside the support of the measures involved.
Notice that fixed $z \in {\mathbb{C}}$ the entries in each string $\bar\Cs_a$ and $\Cs_b$ are series not necessarily convergent. In the non-convergent case we obviously understand  the definition only formally. For each $l \in {\mathbb{Z}}_+$ we
denote by $\bar D_a^{(l)}$ and $D_b^{(l)}$ on $\mathbb{C}$ the domains  where the series $\bar C_a^{(l)}$ and $C_b^{(l)}$ are uniform convergent, respectively,  and we understand them as their corresponding limits. From properties of Taylor's series, we know that uniform convergence of these series hops only on  $\bar D_a^{(l)}$ and $D_b^{(l)}$ when they are the biggest open disks around $z=\infty$ which do not contain the respectively supports, $\operatorname{supp} (w_{2,a}\d\mu)$ and $\operatorname{supp} (w_{2,b}\d\mu)$. Outside the sets $\bar D_a^{(l)}$ and $D_b^{(l)}$ the series diverges at every point. Hence to have non-empty sets  in $\bar D_a^{(l)}$ and $D_b^{(l)}$ the corresponding supports $\operatorname{supp} (w_{2,a}\d\mu)$ and $\operatorname{supp} (w_{2,b}\d\mu)$ must be bounded.
\begin{pro}\label{pro:cauchy tr}
For each  $l\in {\mathbb{Z}}_+$ the second kind functions can be expressed as follows
   \begin{align}
   \begin{aligned}
C_b^{(l)}(z)&=\int \frac{Q^{(l)}(x)w_{2,b}(x)}{z-x}\d \mu(x),& z \in D_b^{(l)} \setminus \operatorname{supp}
(w_{1,b}\d \mu(x)),\\
\bar C_a^{(l)}(z)&=\int  \frac{\bar Q^{(l)}(x)w_{1,a}(x)}{z-x}\d \mu(x),&  z \in \bar D_a^{(l)} \setminus
\operatorname{supp} (w_{2,a}\d \mu(x)).
\end{aligned}
\end{align}
\end{pro}
\begin{proof}
The Gauss--Borel factorization leads to
\begin{align*}
\notag C_b^{(l)}(z)&=\sum_{n=0}^\infty\sum_{k=0}^lS_{lk}g_{kn}(\Pi_{2,b}\chi_2^*(z))_n\\
&=\sum_{n=0}^{\infty}\int  \sum_{k=0}^l S_{lk} x^{k_1(k)}w_{1,a_1(k)}(x)w_{2,b}(x)\frac{x^{n}}{z^{n+1}}\d
\mu(x)\notag&&\text{use \eqref{explicit g2}}\\
&=\sum_{n=0}^{\infty}\frac{1}{z^{n+1}}\int  x^{n} Q^{(l)}(x)w_{2,b}(x)\d \mu(x).\notag&&\text{use
\eqref{linear.forms}}
\end{align*}
When $D_b^{(l)} \setminus \operatorname{supp} (w_{2,b}\d\mu)= \emptyset$ the proof is trivial. Given a non empty compact set $\mathcal K\subset D_b^{(l)} \setminus \operatorname{supp} (w_{2,b}\d\mu)\not = \emptyset$ and recalling the closed character of $\operatorname{supp} (w_{2,b}\d\mu)$,  we have that the distance between them $d^{(l)}_b(K):=\text{distance}(
{\mathcal{K}}, \operatorname{supp} (w_{2,b}\d\mu)) > 0$ is positive and that $\sup \{|z|:z \in \mathcal K\}=:M_{\mathcal K} < +\infty$. Taking into account that the series
\begin{align*}
C_b^{(l)}(z)=\sum_{n=0}^{\infty}\frac{1}{z^{n+1}}\int  x^{n} Q^{(l)}(x)w_{2,b}(x)\d \mu(x)
\end{align*}
converges uniformly on ${\mathcal{K}}$ we can ensure
\begin{equation}\label{connec}
\lim_{n \to \infty} \sup_{|z| \in {\mathcal{K}}}\left\{\left|\frac{1}{z^{n+1}}\int  x^{n} Q^{(l)}(x)w_{2,b}(x)\d
\mu(x)\right|\right\} =0.
\end{equation}
Hence, we have the bound
\begin{align}\label{ct1}
\notag \left|\sum_{i=0}^{n}\frac{1}{z^{i+1}}\int  x^{i} Q^{(l)}(x)w_{2,b}(x)\d \mu(x)-\int  Q^{(l)}(x)w_{2,b}(x)\frac{1}{z-x}\d \mu(x)\right|&=\left|z \frac{1}{z^{n+1}}\int  x^n
Q^{(l)}(x)w_{2,b}(x)\frac{\d \mu(x) }{z-x}\right| \leq\\
\leq \frac{M_{\mathcal K}}{d^{(l)}_b({\mathcal{K}})}\sup_{|z| \in {\mathcal{K}}}\left\{\left|\frac{1}{z^{n+1}}\int  x^{n}
Q^{(l)}(x)w_{2,b}(x)\d \mu(x)\right|\right\},& \forall z \in {\mathcal{K}}.
\end{align}
Taking into account  \eqref{connec} we deduce from \eqref{ct1}  the first equality for any compact set $\mathcal K$. Therefore, we get the first claim of the Proposition; the second equality can be proved analogously.
\end{proof}

 Given $l\geq 1$ and  $a=1,\cdots, p$ the $+$ ($-$) associated integer is  the smallest (largest) integer
$ l_{+a}$ ($l_{-a}$) such that  $l_{+a} \geq l$ ($l_{-a} \leq l$ ) and $a( l_{+a})=a$ ($a( l_{-a})=a$).
It can be shown that
\begin{align}\label{ai}
\begin{aligned}
  l_{-a}&:=\begin{cases} q(l)|\vec n|+\sum_{i=1}^an_{i}-1, & a<a(l), \\
  l, &a=a(l),\\
  q(l)|\vec n|-\sum_{i=a+1}^{p}n_{i}-1, & a>a(l-1), \end{cases}\\
  l_{+a}&:=\begin{cases}
    (  q(l)+1)|\vec n|+\sum_{i=1}^{a-1}n_{i},& a<a(l),\\
  l,&a=a(l),\\
 ( q(l)+1)|\vec n|-\sum_{i=a}^{p}n_{i},& a>a(l).
 \end{cases}
 \end{aligned}
\end{align}
 To give a determinantal expression for these second kind formal series we need
 \begin{definition}
We introduce
\begin{align}
\begin{aligned}
  \Gamma^{(l)}_{k,a}&:=\sum_{k'=l_{+a}}^{\infty}g_{k',k
}z^{-k_1(k')-1}\delta_{a_1(k'),a},&
   \bar\Gamma^{(l)}_{k,b}&:=\sum_{k'=\bar
l_{+b}}^{\infty}g_{k,k'}z^{-k_2(k')-1}\delta_{a_2(k'),b}.
\end{aligned}
\end{align}
Here $l_{+a}$ is the $+$ associated integer within the $\vec n_1$ composition, while $\bar l_{+b}$ is the $+$ associated
integer for the $\vec n_2$ composition.
 \end{definition}
 With these definitions we can state
 \begin{pro}\label{proposition: expressions for ct}
   The following determinantal expressions for the functions of the second kind hold
  \begin{align}
  C_b^{(l)}&=\frac{1}{\det g^{[l]}}\det
  \left(\begin{BMAT}{cccc|c}{cccc|c}
    g_{0,0}&g_{0,1}&\cdots&g_{0,l-1}&\bar\Gamma^{(l)}_{0,b}\\
     g_{1,0}&g_{1,1}&\cdots&g_{1,l-1}&\bar\Gamma^{(l)}_{1,b}\\
     \vdots &\vdots&            &\vdots&\vdots\\
        g_{l-1,0}&g_{l-1,1}&\cdots&g_{l-1,l-1}&\bar\Gamma^{(l)}_{l-1,b}\\
          g_{l,0}&g_{l,1}&\cdots&g_{l,l-1}&\bar\Gamma^{(l)}_{l,b},
\end{BMAT}\right),& l&\geq 1,&&
\\
\bar C_a^{(l)}&=\frac{1}{\det g^{[l+1]}}\det
  \left(\begin{BMAT}{cccc|c}{cccc|c}
    g_{0,0}&g_{0,1}&\cdots&g_{0,l-1}&g_{0,l}\\
     g_{1,0}&g_{1,1}&\cdots&g_{1,l-1}&g_{1,l}\\
     \vdots &\vdots&            &\vdots&\vdots\\
        g_{l-1,0}&g_{l-1,1}&\cdots&g_{l-1,l-1}&g_{l-1,l}\\
          \Gamma^{(l)}_{0,a}& \Gamma^{(l)}_{1,a}&\cdots& \Gamma^{(l)}_{l-1,a}& \Gamma^{(l)}_{l,a}
\end{BMAT}\right),& l&\geq 1.&&
  \end{align}
 \end{pro}
\begin{proof}
 See Appendix \ref{II}.
\end{proof}

Following \cite{LF2} we consider the Markov--Stieltjes functions and polynomials of the second type.
\begin{definition}
  The Markov--Stieltjes functions are defined by
  \begin{align}
    \hat\mu_{a,b}(z):=\int
\frac{w_{1,a}(x)w_{2,b}(x)}{z-x}\d \mu(x),
  \end{align}
in terms of which we define
\begin{align}
\begin{aligned}
  H_b^{(l)}(z)&:=\sum_{a=1}^{p_1}A^{(l)}_a(z)\hat\mu_{a,b}(z)-C^{(l)}_b(z),\\
  \bar H_a^{(l)}(z)&:=\sum_{b=1}^{p_2}\hat\mu_{a,b}(z)\bar A^{(l)}_b(z)-  \bar C^{(l)}_a(z).
  \end{aligned}
\end{align}\end{definition}
\begin{pro}
  The functions  $H_b^{(l)}$ and $\bar H_a^{(l)}$ are polynomials in $z$.
\end{pro}
\begin{proof}
  The reader should notice that the functions $H^{(l)}_b$ and $\bar H^{(l)}_a$ are
\begin{align*}
\begin{aligned}
  H^{(l)}_b(z)&=\int\sum_{a=1}^{p_1}w_{1,a}(x)\frac{A^{(l)}_a(z)-A^{(l)}_a(x)}{z-x}w_{2,b}(x)\d\mu(x),&
  \bar H_a^{(l)}(z)&=\int \sum_{b=1}^{p_2}w_{1,a}(x)\frac{\bar A^{(l)}_b(z)-\bar
  A^{(l)}_b(x)}{z-x}w_{2,b}(x)\d\mu(x),
  \end{aligned}
\end{align*}
and as $z=x$ is a zero of the polynomials $A^{(l)}_a(z)-A^{(l)}_a(x)$ and $\bar A^{(l)}_b(z)-\bar A^{(l)}_b(x)$ from the above formulae we conclude that they are indeed polynomials in $z$.
\end{proof}

  \subsection{Recursion relations}\label{sec:recursion}

The moment matrix has a  Hankel type symmetry that implies the recursion relations and the Christoffel--Darboux
formula. We  consider the shift operators defined by
\begin{align}
\Lambda_a&:=\sum_{k=0}^\infty e_a(k)e_a(k+1)^\top.
\end{align}
notice that
\begin{itemize}
  \item  $\Lambda_a$ leaves invariant the subspaces $\Pi_{a'}\R^\infty$, for  $a'=1,\dots,p$, and
      $\Pi_{a'}\Lambda_a=\Lambda_a\Pi_{a'}$.
\item The  set of semi-infinite matrices $\{\Lambda_a^j\}_{\substack{a=1,\dots,p\\j=1,2,\dots}}$ is commutative.
\item We have the eigenvalue property
\begin{align}
  \Lambda_a\chi_{a'}=\delta_{a,a'}z\chi_{a}.
\end{align}
\end{itemize}

\begin{definition}\label{sym gen}
 We define the following multiple shift matrices
    \begin{align}
  \Upsilon_1&:=\sum_{a=1}^{p_1}\Lambda_{1,a},&
    \Upsilon_2&:=\sum_{b=1}^{p_2}\Lambda_{2,b},
  \end{align}
and we also introduce the integers
 \begin{align*}
   N_{1,a}&:=|\vec n_1|-n_{1,a}+1=\sum_{\substack{a'=1,\dots,p_1\\a'\neq a}}n_{1,a'}+1,& a&=1,\dots,p_1,& N_1&:=\max_{a=1,\dots,p_1}N_{1,a},\\
      N_{2,b}&:=|\vec n_2|-n_{2,b}+1=\sum_{\substack{b'=1,\dots,p_2\\b'\neq b}}n_{2,b'}+1, & b&=1,\dots,p_2,& N_2&:=\max_{b=1,\dots,p_2}N_{2,b}.
   \end{align*}
  \end{definition}
A careful but straightforward computation leads to
 \begin{pro}
 We have the following structure for $\Upsilon_1$ and $\Upsilon_2$
 \begin{align*}
   \Upsilon_1&=D_{1,0}\Lambda+D_{1,1}\Lambda^{N_{1,1}}+\dots+D_{1,p_1}\Lambda^{N_{1,p_1}},&
     \Upsilon_2&=D_{2,0}\Lambda+D_{2,1}\Lambda^{N_{2,1}}+\dots+D_{2,p_2}\Lambda^{N_{2,p_2}}.
 \end{align*}
 where $D_{1,a}$, $a=1,\dots,p_1$, and $D_{2,b}$, $b=1,\dots,p_2$, are the following semi-infinite diagonal matrices:
 \begin{align*}
    D_{1,a}&=\diag(D_{1,a}(0),D_{1,a}(1),\dots),&
   D_{1,a}(n)&:=\begin{cases}
     1,& n=k|\vec n_1|+\sum_{a'=1}^a n_{1,a'}-1,\quad k\in\Z_+,\\
       0,& n\neq k|\vec n_1|+\sum_{a'=1}^a n_{1,a'}-1,\quad k\in\Z_+,
   \end{cases}&
   D_{1,0}&=\I-\sum_{a=1}^{p_1}D_{1,a},\\
       D_{2,b}&=\diag(D_{2,b}(0),D_{2,b}(1),\dots),&
   D_{2,b}(n)&:=\begin{cases}
     1,& n=k|\vec n_2|+\sum_{b'=1}^b n_{2,b'}-1,\quad k\in\Z_+,\\
       0,& n\neq k|\vec n_2|+\sum_{b'=1}^b n_{2,b'}-1,\quad k\in\Z_+,
   \end{cases}&
   D_{2,0}&=\I-\sum_{b=1}^{p_1}D_{2,b}.
 \end{align*}
 \end{pro}

In terms of these shift matrices we can describe the particular Hankel symmetries for the moment matrix
\begin{pro}
  The moment matrix $g$ satisfies the Hankel type symmetry
      \begin{align}\label{sym2}
 \Upsilon_1g  =
 g\Upsilon_2^\top.
  \end{align}
\end{pro}
\begin{proof}
  With the use of \eqref{formulitas} and \eqref{explicit g2} we get
  \begin{align}\label{sym}
 \Lambda_{1,a}g  \Pi_{2,b}=
   \Pi_{1,a}g\Lambda_{2,b}^\top,
  \end{align}
  and summing up in $a=1,\dots,p_1$ and $b=1,\dots,p_2$ we get the desired result.
\end{proof}
Observe that from \eqref{sym2} we deduce that in spite of being all the truncated moment matrices $g^{[l]}$, $l=1,2,\dots$ invertible, the moment matrix $g=\lim_{l\to\infty}g^{[l]}$is not invertible. Suppose that the inverse $g^{-1}=(\tilde g_{i,j})_{1,j=0,1,\dots}$ of $g$ exists so that \eqref{sym2} implies $ g^{-1}\Upsilon_1  = \Upsilon_2^\top g^{-1}$, and therefore $\tilde g_{i,0}=\tilde g_{0,j}=0$ for all $i,j=0,1,\dots$, which is  contradictory with the invertibility of $g$.

\begin{pro}\label{hess}
 From the symmetry of the moment matrix one derives
  \begin{align}\label{bigraded}
  S\Upsilon_1 S^{-1}=\bar S \Upsilon_2^\top\bar S^{-1}.
 \end{align}
\end{pro}
\begin{proof}
  If we introduce \eqref{facto}
 into \eqref{sym2} we get
 \begin{align*}
   \Upsilon_1S^{-1}\bar S=S^{-1}\bar S\Upsilon_2^\top \Rightarrow  S\Upsilon_1S^{-1}=\bar S\Upsilon_2^\top\bar S^{-1}.
 \end{align*}
 \end{proof}
\begin{definition}
We define the matrices
\begin{align*}
 J&:=J_++J_-,& J_+&:=(  S\Upsilon_1 S^{-1})_+, & J_-&:=(\bar S \Upsilon_2^\top\bar S^{-1})_-,
\end{align*}
where the sub-indices + and $-$denotes the upper triangular and strictly lower triangular projections.
 \end{definition}
 Thus, $J_+$ is an upper triangular matrix and $J_-$ a strictly  lower triangular matrix. Moreover, from the string equation \eqref{bigraded} we have the alternative expressions
 \begin{align*}
   J= S\Upsilon_1 S^{-1}=\bar S \Upsilon_2^\top\bar S^{-1}.
 \end{align*}
We now analyze the structure of $J_+:= (SD_{1,0}\Lambda S^{-1})_++ (SD_{1,1}\Lambda^{N_{1,1}}S^{-1})_++
\dots+(SD_{1,p_1}\Lambda^{N_{1,p_1}}S^{-1})_+$. It is clear that we need to evaluate  expressions of the form $SE_{i,j}S^{-1}$ with $i= \kappa_1(k,a)-1$ and  $j=\kappa_1(k+1,a-1)$ being $\kappa_1(k,a):= k|\vec n_1|+\sum_{a'=1}^an_{1,a'}$. Given the form of $S$, see \eqref{facto}, we have
\begin{align*}
  (SE_{i,j}S^{-1})_+&=E_{i,j}+\sum_{l,l\in \mathbb L_{i,j}}s_{l,l'}E_{l,l'},&\mathbb L_{i,j}&:=\{(l,l')\in\Z_+^2|  l<i, l'<j,l'\geq l\},
\end{align*}
for some numbers $s_{l,l'}\in\R$ depending on the coefficients of $S$ and on $i,j$; this matrix has zeroes everywhere but on a region of it that can be represented as a right triangle with hypotenuse lying on the main diagonal, this hypotenuse has its opposite vertex precisely on the $(i,j)$ position.
Therefore
\begin{align*}
  J_+&= (SD_{1,0}\Lambda S^{-1})_++\sum_{a=1}^{p_1}\sum_{k=0}^\infty\Bigg(E_{\kappa_1(k,a)-1,\kappa_1(k+1,a-1)}+\sum_{l,l'\in\mathbb L_{1,k,a}}s_{l,l'}E_{l,l'}\Bigg),& \mathbb L_{1,k,a}&:=\mathbb L_{\kappa_1(k,a)-1,\kappa_1(k+1,a-1)}
\end{align*}
We see that $J_+$ can be schematically  represented  as a staircase, the $\vec n_1$-staircase, descending over the main diagonal with steps --which are built with right triangles with hypotenuse lying on the main diagonal  and opposite vertex (and therefore  corner of the step) located at the $(\kappa_1(k,a)-1,\kappa_1(k+1,a-1))$ position of the matrix--  having width and height given by the integers in the composition $\vec n_1$.
For example, the $j$-th step has width $n_{1,\frac{j}{p_1}-[\frac{j}{p_1}]}$ and height $n_{1,\frac{j+1}{p_1}-[\frac{j+1}{p_1}]}$. A similar description holds for $J_-^\top$ but replacing the composition $\vec n_1$ by $\vec n_2$. Therefore, the matrix $J$ is a generalized Jacobi matrix and, in contrast with the non multiple case, now is multi-diagonal (having in general more than three diagonals) and has a diagonal band of length $N_{1}+N_{2}+1$. Moreover, this band has  a number of zeroes on it, according to the $\vec n_1$-stair on the upper part and to the $\vec n_2$-stair in the lower part, we refer to this as a double $(\vec n_1,\vec n_2)$-staircase shape. To illustrate this \emph{snake} shape let us write for the case $\vec n_1=(4,3,2)$ and $\vec n_2=(3,2)$ the corresponding truncated, $l=27$,  Jacobi type matrix
\begin{align}\label{J}
\small
J^{[27]}=\left(
  \begin{array}{ccccccccccccccccccccccccccc}
   \textcolor{blue}{*} & \boldsymbol{\textcolor{blue}{1}} & 0 & 0 & 0 & 0 & 0 & 0 & 0 & 0 & 0 & 0 & 0 & 0 & 0 & 0 & 0 & 0 & 0 & 0 & 0 & 0 & 0 & 0 & 0 & 0 & 0 \\
 \boldsymbol{\textcolor{blue}{*}} & \boldsymbol{\textcolor{blue}{*}} &\boldsymbol{\textcolor{blue}{1}} & 0 & 0 & 0 & 0 & 0 & 0 & 0 & 0 & 0 & 0 & 0 & 0 & 0 & 0 & 0 & 0 & 0 & 0 & 0 & 0 & 0 & 0 & 0 & 0 \\
    0 & \boldsymbol{\textcolor{blue}{*}} & \boldsymbol{\textcolor{blue}{\textcolor{blue}{*}}} &\boldsymbol{\textcolor{blue}{1}} & 0 & 0 & 0 & 0 & 0 & 0 & 0 & 0 & 0 & 0 & 0 & 0 & 0 & 0 & 0 & 0 & 0 & 0 & 0 & 0 & 0 & 0 & 0 \\
    0 & 0 & \boldsymbol{\textcolor{blue}{*}} & \boldsymbol{\textcolor{blue}{*}} & \boldsymbol{\textcolor{blue}{*}} & \boldsymbol{\textcolor{blue}{*}} &\boldsymbol{ \textcolor{blue}{*}} & \boldsymbol{\textcolor{blue}{*}} & \boldsymbol{\textcolor{blue}{*}} &\boldsymbol{\textcolor{blue}{1}} & 0 & 0 & 0 & 0 & 0 & 0 & 0 & 0 & 0 & 0 & 0 & 0 & 0 & 0 & 0 & 0 & 0 \\
    0 & 0 & \boldsymbol{\textcolor{blue}{*}} & \textcolor{blue}{*} & \textcolor{blue}{*} & \textcolor{blue}{*} &\textcolor{blue}{*} & \textcolor{blue}{*} & \textcolor{blue}{*} & \boldsymbol{\textcolor{blue}{*}} & 0 & 0 & 0 & 0 & 0 & 0 & 0 & 0 & 0 & 0 & 0 & 0 & 0 & 0 & 0 & 0 & 0 \\
    0 & 0 & \boldsymbol{ \textcolor{blue}{*}} & \boldsymbol{\textcolor{blue}{*}} & \boldsymbol{\textcolor{blue}{*}} &\textcolor{blue}{*} & \textcolor{blue}{*} & \textcolor{blue}{*} &\textcolor{blue}{*} & \boldsymbol{\textcolor{blue}{*}} & 0 & 0 & 0 & 0 & 0 & 0 & 0 & 0 & 0 & 0 & 0 & 0 & 0 & 0 & 0 & 0 & 0 \\
    0 & 0 & 0 & 0 & \boldsymbol{\textcolor{blue}{*}} &\textcolor{blue}{*} & \textcolor{blue}{*} & \textcolor{blue}{*} &\textcolor{blue}{*} & \boldsymbol{\textcolor{blue}{*}} & \boldsymbol{\textcolor{blue}{*}} & \boldsymbol{\textcolor{blue}{*}} & \boldsymbol{\textcolor{blue}{*}} &\boldsymbol{ \textcolor{blue}{1}} & 0 & 0 & 0 & 0 & 0 & 0 & 0 & 0 & 0 & 0 & 0 & 0 & 0 \\
    0 & 0 & 0 & 0 & \boldsymbol{\textcolor{blue}{*}} & \textcolor{blue}{*} & \textcolor{blue}{*} & \textcolor{blue}{*} & \textcolor{blue}{*} &\textcolor{blue}{*} & \textcolor{blue}{*} & \textcolor{blue}{*} & \textcolor{blue}{*} & \boldsymbol{\textcolor{blue}{*}} & 0 & 0 & 0 & 0 & 0 & 0 & 0 & 0 & 0 & 0 & 0 & 0 & 0 \\
    0 & 0 & 0 & 0 & \boldsymbol{\textcolor{blue}{*}} & \boldsymbol{\textcolor{blue}{*}} & \boldsymbol{\textcolor{blue}{*}} & \boldsymbol{\textcolor{blue}{*}} & \textcolor{blue}{*} & \textcolor{blue}{*} &\textcolor{blue}{*} & \textcolor{blue}{*} & \textcolor{blue}{*} & \boldsymbol{\textcolor{blue}{*}} & \boldsymbol{\textcolor{blue}{*}} & \boldsymbol{\textcolor{blue}{*}} &\boldsymbol{\textcolor{blue}{1}} & 0 & 0 & 0 & 0 & 0 & 0 & 0 & 0 & 0 & 0 \\
    0 & 0 & 0 & 0 &   0 & 0 & 0 & \boldsymbol{\textcolor{blue}{*}} & \textcolor{blue}{*} & \textcolor{blue}{*} & \textcolor{blue}{*} &\textcolor{blue}{*}& \textcolor{blue}{*} & \textcolor{blue}{*} &\textcolor{blue}{*} & \textcolor{blue}{*} & \boldsymbol{\textcolor{blue}{*}} & 0 & 0 & 0 & 0 & 0 & 0 & 0 & 0 & 0 & 0 \\
    0 & 0 & 0 & 0 &   0 & 0 & 0& \boldsymbol{\textcolor{blue}{*}} & \boldsymbol{\textcolor{blue}{*}} & \boldsymbol{\textcolor{blue}{*}} &\textcolor{blue}{*}& \textcolor{blue}{*} & \textcolor{blue}{*} &\textcolor{blue}{*}& \textcolor{blue}{*} & \textcolor{blue}{*} & \boldsymbol{\textcolor{blue}{*}} & 0 & 0 & 0 & 0 & 0 & 0 & 0 & 0 & 0 & 0 \\
    0 & 0 & 0 & 0 & 0 & 0 & 0 & 0 & 0 & \boldsymbol{\textcolor{blue}{*}} & \textcolor{blue}{*} & \textcolor{blue}{*} & \textcolor{blue}{*} & \textcolor{blue}{*} &\textcolor{blue}{*}& \textcolor{blue}{*} & \boldsymbol{\textcolor{blue}{*}} & 0 & 0 & 0 & 0 & 0 & 0 & 0 & 0 & 0 & 0 \\
    0 & 0 & 0 & 0 & 0 & 0 & 0 & 0 & 0 & \boldsymbol{\textcolor{blue}{*}} &\textcolor{blue}{*} & \textcolor{blue}{*} &\textcolor{blue}{*}& \textcolor{blue}{*} & \textcolor{blue}{*} &\textcolor{blue}{*}& \boldsymbol{\textcolor{blue}{*}}  &\boldsymbol{ \textcolor{blue}{*}} &\boldsymbol{\textcolor{blue}{1}} & 0 & 0 & 0 & 0 & 0 & 0 & 0 & 0 \\
    0 & 0 & 0 & 0 & 0 & 0 & 0 & 0 & 0 & \boldsymbol{\textcolor{blue}{*}} &\boldsymbol{\textcolor{blue}{*}} & \boldsymbol{\textcolor{blue}{*}} & \boldsymbol{\textcolor{blue}{*}} &\textcolor{blue}{*}& \textcolor{blue}{*} & \textcolor{blue}{*} &\textcolor{blue}{*} & \textcolor{blue}{*} & \boldsymbol{\textcolor{blue}{*}} & 0 & 0 & 0 & 0 & 0 & 0 & 0 & 0 \\
    0 & 0 & 0 & 0 & 0 & 0 & 0 & 0 &  0 & 0 &0 & 0 & \boldsymbol{\textcolor{blue}{*}} &\textcolor{blue}{*} &\textcolor{blue}{*} & \textcolor{blue}{*} &\textcolor{blue}{*} & \textcolor{blue}{*} &\boldsymbol{ \textcolor{blue}{*}} & 0 & 0 & 0 & 0 & 0 & 0 & 0 & 0 \\
    0 & 0 & 0 & 0 & 0 & 0 & 0 & 0 & 0 & 0 & 0 & 0 &\boldsymbol{\textcolor{blue}{*}} & \boldsymbol{ \textcolor{blue}{*}} & \boldsymbol{\textcolor{blue}{*}} & \textcolor{blue}{*} & \textcolor{blue}{*}  & \textcolor{blue}{*} & \boldsymbol{\textcolor{blue}{*}}& \boldsymbol{\textcolor{blue}{*}} & \boldsymbol{\textcolor{blue}{*}} &\boldsymbol{\textcolor{blue}{*}} &\boldsymbol{\textcolor{blue}{1}} & 0 & 0 & 0 & 0 \\
    0 & 0 & 0 & 0 & 0 & 0 & 0 & 0 & 0 & 0 & 0 & 0 &0 & 0 & \boldsymbol{\textcolor{blue}{*}} & \textcolor{blue}{*} & \textcolor{blue}{*}  & \textcolor{blue}{*} & \textcolor{blue}{*}& \textcolor{blue}{*} & \textcolor{blue}{*} &\textcolor{blue}{*} & \boldsymbol{\textcolor{blue}{*}} & 0 & 0 & 0 & 0 \\
    0 & 0 & 0 & 0 & 0 & 0 & 0 & 0 & 0 & 0 & 0 & 0 &0 & 0 & \boldsymbol{\textcolor{blue}{*}} & \textcolor{blue}{*} & \textcolor{blue}{*}  & \textcolor{blue}{*} & \textcolor{blue}{*}& \textcolor{blue}{*} & \textcolor{blue}{*} &\textcolor{blue}{*} & \boldsymbol{\textcolor{blue}{*}} & \boldsymbol{\textcolor{blue}{*}} & \boldsymbol{\textcolor{blue}{*}} &\boldsymbol{\textcolor{blue}{1}} & 0 \\
    0 & 0 & 0 & 0 & 0 & 0 & 0 & 0 & 0 & 0 & 0 & 0 & 0 & 0 &\boldsymbol{\textcolor{blue}{*}} & \boldsymbol{\textcolor{blue}{*}} & \boldsymbol{\textcolor{blue}{*}}  & \boldsymbol{\textcolor{blue}{*}} & \textcolor{blue}{*}& \textcolor{blue}{*} & \textcolor{blue}{*} &\textcolor{blue}{*} & \textcolor{blue}{*} & \textcolor{blue}{*} & \textcolor{blue}{*} &\boldsymbol{ \textcolor{blue}{*}} & 0 \\
    0 & 0 & 0 & 0 & 0 & 0 & 0 & 0 & 0 & 0 & 0 & 0 & 0 & 0 & 0 & 0 & 0 & \boldsymbol{\textcolor{blue}{*}} & \textcolor{blue}{*}& \textcolor{blue}{*} & \textcolor{blue}{*} &\textcolor{blue}{*} & \textcolor{blue}{*} & \textcolor{blue}{*} & \textcolor{blue}{*} & \boldsymbol{\textcolor{blue}{*}} &  0 \\
    0 & 0 & 0 & 0 & 0 & 0 & 0 & 0 & 0 & 0 & 0 & 0 & 0 & 0 & 0 & 0 & 0 & \boldsymbol{\textcolor{blue}{*}} & \boldsymbol{\textcolor{blue}{*}}& \boldsymbol{\textcolor{blue}{*}} & \textcolor{blue}{*} &\textcolor{blue}{*} & \textcolor{blue}{*} & \textcolor{blue}{*} & \textcolor{blue}{*} & \boldsymbol{\textcolor{blue}{*}}  & 0 \\
    0 & 0 & 0 & 0 & 0 & 0 & 0 & 0 & 0 & 0 & 0 & 0 & 0 & 0 & 0 & 0 & 0 & 0 & 0 &  \boldsymbol{\textcolor{blue}{*}} & \textcolor{blue}{*} &\textcolor{blue}{*} & \textcolor{blue}{*} & \textcolor{blue}{*} & \textcolor{blue}{*} & \boldsymbol{\textcolor{blue}{*}}   & \boldsymbol{\textcolor{blue}{*}} \\
    0 & 0 & 0 & 0 & 0 & 0 & 0 & 0 & 0 & 0 & 0 & 0 & 0 & 0 & 0 & 0 & 0 & 0 & 0 & \boldsymbol{\textcolor{blue}{*}} & \textcolor{blue}{*} &\textcolor{blue}{*} & \textcolor{blue}{*} & \textcolor{blue}{*} & \textcolor{blue}{*} &\textcolor{blue}{*} &\textcolor{blue}{*}\\
0 & 0 & 0 & 0 & 0 & 0 & 0 & 0 & 0 & 0 & 0 & 0 & 0 & 0 & 0 & 0 & 0 & 0 & 0 & \boldsymbol{\textcolor{blue}{*}} & \boldsymbol{\textcolor{blue}{*}} &\boldsymbol{\textcolor{blue}{*}} & \boldsymbol{\textcolor{blue}{*}} & \textcolor{blue}{*} & \textcolor{blue}{*} & \textcolor{blue}{*}  & \textcolor{blue}{*} \\
    0 & 0 & 0 & 0 & 0 & 0 & 0 & 0 & 0 & 0 & 0 & 0 & 0 & 0 & 0 & 0 & 0 & 0 & 0 & 0 & 0 & 0 & \boldsymbol{\textcolor{blue}{*}} & \textcolor{blue}{*} & \textcolor{blue}{*} & \textcolor{blue}{*} & \textcolor{blue}{*} \\
  0 & 0 & 0 & 0 & 0 & 0 & 0 & 0 & 0 & 0 & 0 & 0 & 0 & 0 & 0 & 0 & 0 & 0 & 0 &0 &0 &0 & \boldsymbol{\textcolor{blue}{*}} & \boldsymbol{\textcolor{blue}{*}} & \boldsymbol{\textcolor{blue}{*}} & \textcolor{blue}{*}  & \textcolor{blue}{*} \\
    0 & 0 & 0 & 0 & 0 & 0 & 0 & 0 & 0 & 0 & 0 & 0 & 0 & 0 & 0 & 0 & 0 & 0 & 0 & 0 & 0 & 0 &0 & 0 & \boldsymbol{\textcolor{blue}{*}} & \textcolor{blue}{*} & \textcolor{blue}{*} \\
  \end{array}
\right),
\end{align}
where $*$ denotes a non-necessarily null real number.
We  can write
    \begin{align}
  J=J_{N_1}\Lambda^{N_1}+\dots+J_1\Lambda+J_{0}+ J_{-1}\Lambda^\top
  +\dots+ J_{-N_2}(\Lambda^\top)^{N_2},
\end{align}
where
 $J_i=\diag(J_i(0),J_i(1),\dots)$. For convenience we extend the notation with
$  J_r(s)=0$ whenever $r+s< 0$ or $s<0$.
We introduce
  \begin{definition}
  The semi-infinite vectors $c_b$ and $\bar c_a$ are given by
  \begin{align}
  \begin{aligned}
    c_b&:=\bar S e_{2,b}(0),& b&=1,\dots,p_2,\\
    \bar c_a&:=(S^{-1})^\top e_{1,a}(0),& a&=1,\dots,p_1.
    \end{aligned}
  \end{align}
  \end{definition}
  It is not difficult to show that
  \begin{align}
    c_b&=\sum_{l=0}^{n_{2,1}+\dots+n_{2,b}-1} \bar S_{l,n_{2,1}+\dots+n_{2,b}-1}e_{l},&
   \bar  c_a&=\sum_{l=0}^{n_{1,1}+\dots+n_{1,a}-1}  (S^{-1})_{n_{1,1}+\dots+n_{1,a}-1,l}e_{l}.
  \end{align}

The semi-infinite matrices  $J$ and $J^\top$ have the following important property
  \begin{pro}\label{eigen}
    The following equations are fulfilled
    \begin{align}
    \begin{aligned}
      J\A_a(z)&=z\A_a(z),&    J^\top\bar\A_{b}(z)&=  z\bar\A_{b}(z),\\
      J\Cs_b(z)&=z\Cs_b(z)-c_b, &   J^\top\bar\Cs_a(z)&=z\Cs_a(z)-\bar c_a.
      \end{aligned}
    \end{align}
  \end{pro}
  \begin{proof}
    From \eqref{mop-s} and \eqref{cauchy-S}
     \begin{align*}
 J\A_{a}(z)&=S\Upsilon_1 S^{-1} S\chi_{1,a}(z)=zS\chi_{1,a}=z\A_a(z),\\
 J\Cs_b(z)&=\bar S\Upsilon_2^\top\bar S^{-1}\bar S\chi_{2,b}^*(z)=\bar S(z\chi^*_{2,b}(z)-e_{b}(0))=z\Cs_b(z)-c_b,
     \end{align*}
   where we have taken into account that $\Upsilon_2^\top\chi_{2,b}^*(z)=z\chi^*_{2,b}(z)-e_b(0)$.
For $J^\top$ we proceed similarly:
       \begin{align*}
 J^\top\bar\A_{b}(z)&=(\bar S^{-1})^\top \Upsilon_2 \bar S^\top (\bar S^{-1})^\top\chi_{2,b}(z)=z(\bar
 S^{-1})^\top\chi_{2,a}=z\bar\A_b(z),\\
 J^\top\bar \Cs_a(z)&=( S^{-1})^\top \Upsilon_1^\top S^\top( S^{-1})^\top \chi_{1,a}^*(z)=( S^{-1})^\top
 (z\chi^*_{1,a}(z)-e_{a}(0))=z\bar\Cs_a(z)-\bar c_a.
     \end{align*}
      \end{proof}

  \begin{theorem}
  The multiple orthogonal polynomials and their associated second kind functions fulfill the following recursion relations
  \begin{align}
  \begin{aligned}
  zA^{(l)}_{a}(z)&=J_{-N_2}(l)A^{(l-N_2)}_{a}(z)+\cdots+J_{N_1}(l)
  A^{(l+N_1)}_{a}(z),\\
    zC^{(l)}_{b}(z)-c_b^{(l)}&=J_{-N_2}(l)C_b^{(l-N_2)}(z)+\cdots+J_{N_1}(l)
  C_b^{(l+N_1)}(z),
  \end{aligned}
\end{align}
while the dual relations are
\begin{align}
\begin{aligned}
  z\bar A^{(l)}_{b}(z)&=J_{-N_2}(l+N_2)\bar A^{(l+N_2)}_{b}(z)+\cdots+J_{N_1}(l-N_1)\bar A^{(l-N_1)}_{b}(z),\\
 z\bar C^{(l)}_{a}(z)-\bar c_a&=J_{-N_2}(l+N_2)\bar C^{(l+N_2)}_{a}(z)+\cdots+J_{N_1}(l-N_1)\bar C^{(l-N_1)}_{a}(z).
 \end{aligned}
\end{align}
  \end{theorem}

We see that given integers $(\vec \nu_1,\vec \nu_2)$ there are several recursion relations associated with $A_{[\vec
\nu_1;\vec \nu_2],a}$. In fact they are as many as different ladders exists leading to this set of degrees.  For the simplest ladder, i.e.   $\vec n_1=\vec \nu_1$ and $\vec n_2=\vec \nu_2+\vec e_{2,p_2}$, we get the longest recursion,
in the sense that we have more polynomials contributing in the recursion relation, as smaller are the integers in the compositions
shorter is the recursion. Observe also that the multiple orthogonal polynomials involved  in each case are different.

Attending to  \eqref{J} we get that  the recursion relations corresponding to $l=8$ an $l=14$ are of the form
\begin{align*}
 zA^{(8)}_a(z)&=*A_a^{(4)}(z)+\cdots+*A_a^{(15)}(z)+A_a^{(16)}(z),&a=1,2,3,\\
  zA^{(14)}_a(z)&=*A_a^{(12)}(z)+\cdots+*A_a^{(18)}(z),& a=1,2,3.
\end{align*}
We see that the first recursion has 13 terms while the second one only 7 terms.

In order to identify these polynomials with mops of the form $A_{[\vec \nu_1;\vec \nu_2,a]}$ we use following the table of degrees  for the compositions $\vec n_1=(4,3,2)$ and $\vec n_2=(3,2)$ is

\vspace*{5pt}
 \tiny
\hspace*{-10pt}\begin{tabular}{|c||c|c|c|c|c|c|c|c|c|c|c|c|c|c|c|}
  \hline
   l& 4 & 5 & 6 & 7 & 8 & 9 &10 & 11 & 12 & 13 & 14 & 15 & 16 &17 & 18
\\ \hline
  $\vec\nu_1(l)$ & (4,0,0) & (4,1,0) & (4,2,0) & (4,3,0) & (4,3,1) & (4,3,2) & (5,3,2) & (6,3,2) & (7,3,2) & (8,3,2) & (8,4,2) & (8,5,2) & (8,6,2)  & (8,6,3)& (8,6,4)\\
  $\vec\nu_2 (l-1)$& (2,0) & (3,0) & (3,1) & (3,2) & (4,2) & (5,2) & (6,2) & (6,3) & (6,4) & (7,4) &(8,4)  & (9,4) & (9,5) & (9,6)&(10,6)\\
  \hline
\end{tabular}
\vspace*{5pt}
\normalsize

\subsection{Christoffel--Darboux type formulae}
From \eqref{mop-s} and \eqref{cauchy-S} we can infer the value of the following series constructed in terms of
multiple orthogonal polynomials and  corresponding functions of the second kind.
\begin{pro}\label{proposition: CD type relations}
  The following relations hold
  \begin{align}
  \begin{aligned}
    \sum_{l=0}^\infty \bar C_a^{(l)}(z)A_{a'}^{(l)}(z')&=\frac{\delta_{a,a'}}{z-z'},& |z'|<|z|,\\
     \sum_{l=0}^\infty  C_b^{(l)}(z)\bar A_{b'}^{(l)}(z')&=\frac{\delta_{b,b'}}{z-z'},& |z'|<|z|,\\
      \sum_{l=0}^\infty \bar C_a^{(l)}(z)C_b^{(l)}(z')&=-\frac{\hat\mu_{a,b}(z)-\hat\mu_{a,b}(z')} {z-z'}, &
      |z|,|z'|>R_{a,b},
      \end{aligned}
  \end{align}
where $R_{a,b}$ is the radius of any origin centered disk containing $\operatorname{supp}(w_{1,a}w_{a,b}\d \mu)$.
\end{pro}
\begin{proof}
  See Appendix \ref{II}.
\end{proof}

\subsubsection{Projection operators and the Christoffel--Darboux kernel}
To introduce the Christoffel--Darboux kernel  we need
\begin{definition}\label{CD projection}
  We will use the  following spans
\begin{align}
\H^{[l]}_1&=\R\{\xi_1^{(0)},\dots,\xi_1^{(l-1)}\},&
\H^{[l]}_2&=\R\{\xi_2^{(0)},\dots,\xi_2^{(l-1)}\},
\end{align}
and their limits
\begin{align}
\H_1&=\Big\{\sum_{0\leq l \ll \infty} c_l \xi_1^{(l)}, c_l\in\R\Big\},&
\H_2&=\Big\{\sum_{0\leq l \ll  \infty} c_l \xi_2^{(l)},c_l\in\R\Big\}.
\end{align}
The corresponding splittings
\begin{align}
 \H_1&=\H^{[l]}_1\oplus (\H^{[l]}_{1})^{\perp},&  \H_2&=\H^{[l]}_2\oplus (\H^{[l]}_{2})^{\perp},
\end{align}
induce the associated orthogonal projections
\begin{align}
  \pi^{(l)}_1&:\H_1\to\H^{[l]}_1,&  \pi^{(l)}_2&:\H_2\to\H^{[l]}_2.
\end{align}
\end{definition}
In the previous definition $l\ll \infty$ means that in the series there are only a finite number
of nonzero contributions. It is easy to realize that
\begin{pro}
  We have the following characterization of the previous linear subspaces
  \begin{align}\begin{aligned}
\H^{[l]}_1&=\R\{Q^{(0)},\dots,Q^{(l-1)}\},&
\H^{[l]}_2&=\R\{\bar Q^{(0)},\dots,\bar Q^{(l-1)}\}, \\
(\H^{[l]}_1)^{\perp}&=\Big\{\sum_{l\leq j \ll \infty} c_j Q^{(j)}, c_j\in\R\Big\},&
(\H^{[l]}_2)^{\perp}&=\Big\{\sum_{l\leq j \ll  \infty} c_j\bar Q^{(j)},c_j\in\R\Big\},\end{aligned}
\end{align}
and
\begin{align}
\H_1&=\Big\{\sum_{0\leq l \ll \infty} c_l Q^{(l)}, c_l\in\R\Big\},&
\H_2&=\Big\{\sum_{0\leq l \ll  \infty} c_l\bar Q^{(l)},c_l\in\R\Big\}.
\end{align}
\end{pro}
\begin{definition}
The  Christoffel--Darboux kernel is
\begin{align}
  \label{def.CD}
  K^{[l]}(x,y)&:=\sum_{k=0}^{l-1}Q^{(k)}(y)\bar Q^{(k)}(x).
\end{align}
\end{definition}
This is the kernel of the integral representation of the projections introduced in Definition \ref{CD projection}.
\begin{pro}
The integral representation
  \begin{align}
  \begin{aligned}
 ( \pi^{(l)}_1f)(y)&=\int  K^{[l]}(x,y)f(x)\d \mu(x), & \forall f\in\H_1,\\
  ( \pi^{(l)}_2f)(y)&=\int  K^{[l]}(y,x)f(x)\d \mu(x), &\forall f\in\H_2,
  \end{aligned}
\end{align}
holds.
\end{pro}
\begin{proof}
It follows from the bi-orthogonality condition \eqref{biotrhoganility}.
\end{proof}
This Christoffel--Darboux kernel   has the reproducing property
\begin{pro}
  The kernel $K^{[l]}(x,y)$ fulfills
  \begin{align}
K^{[l]}(x,y)=\int  K^{[l]}(x,v)K^{[l]}(v,y)\d\mu( v).
\end{align}
\end{pro}
\begin{proof}
From
\begin{align*}
 f(y)&=\int  K^{[l]}(x,y)f(x)\d \mu(x), &\forall  f\in\H^{[l]}_1,\\
f(y)&=\int  K^{[l]}(y,x)f(x)\d \mu(x), & \forall f\in\H^{[l]}_2,
\end{align*}
 and  $K^{[l]}(x,y)\in \H^{[l]}_1$ as a function of $y$ and $K^{[l]}(x,y)\in \H^{[l]}_2$ as a function of $x$ we
 conclude the reproducing property.
\end{proof}

\subsubsection{The ABC type theorem}

We also have an ABC (Aitken--Berg--Collar) type theorem ---here we follow  \cite{simon}--- for the Christoffel--Darboux kernel
\begin{definition}
  The partial  Christoffel--Darboux kernels  are defined by
  \begin{align}
  \label{pdef.CD}
  K^{[l]}_{b,a}(x,y)&:=\sum_{k=0}^{l-1} \bar A_b^{(k)}(x)A_a^{(k)}(y).
\end{align}
\end{definition}
Observe that
\begin{align}
K^{[l]}(x,y)&=\sum_{\substack{a=1,\dots,p_1\\b=1,\dots,p_2}}K^{[l]}_{b,a}(x,y)w_{1,a}(y)w_{2,b}(x).
\end{align}

We introduce the notation
\begin{definition}\label{vec-block}
Any semi-infinite vector  $v$ can be written in block form as follows
\begin{align}v&= \left(\begin{BMAT}{c}{c|c}
v^{[l]}\\
 v^{[\geq l]}
\end{BMAT}\right),
\end{align}
where $v^{[l]}$ is the finite vector formed with the first $l$ coefficients of $v$ and $v^{[\geq l]}$ the
semi-infinite vector formed with the remaining coefficients.
This decomposition induces the following block structure for any semi-infinite matrix.
\begin{align}
  g= \left(\begin{BMAT}{c|c}{c|c}
   g^{[l]}&g^{[l,\geq l]}\\
     g^{[\geq l,l]}& g^{[\geq l]}
\end{BMAT}\right).
\end{align}
\end{definition}
From \eqref{facto} we get
\begin{pro}
 Given a moment matrix $g$ satisfying \eqref{facto} we have
\begin{align}
  g^{[l]}&=(S^{[l]})^{-1}\bar S^{[l]},
\end{align}
and $(S^{-1})^{[l]}=(S^{[l]})^{-1}$, $(\bar S^{-1})^{[\geq l]}=(\bar S^{[\geq l]})^{-1}$.
\end{pro}
\begin{proof}
  Use the block structure of $g,S$ and $\bar S$.
\end{proof}
Then, we are able to conclude the following result
\begin{theorem}\label{th:abc}
  The   Christoffel--Darboux kernel is related to the moment matrix in the
  following way
\begin{align}\label{abc-p}
     K^{[l]}_{b,a}(x,y)&=(\chi_{2,b}^{[l]}(x))^\top (g^{[l]})^{-1}\chi_{1,a}^{[l]}(y).
\end{align}
\end{theorem}
\begin{proof}
 The ABC theorem  is a consequence of the following chain of identities
\begin{align*}
  K^{[l]}_{b,a}(x,y)&=(\Pi^{[l]}\bar \A_b(x))^\top(\Pi^{[l]} \A_a(y))& &\text{the sum is over the first $l$
  components}\\
  &=\chi_{2,b}^\top(x) \bar S^{-1}\Pi^{[l]} S\chi_{1,a}(y)& &\text{see \eqref{linear forms S}} \\
  &=\chi_{2,b}^\top(x) (\Pi^{[l]}\bar S^{-1}\Pi^{[l]})(\Pi^{[l]} S\Pi^{[l]})\chi_{1,a}(y)& &\text{lower and upper form
  of $S$ and $\bar S$}\\&=(\chi_{2,b}^{[l]}(x))^\top (\bar
  S^{[l]})^{-1}S^{[l]}\chi_{1,a}^{[l]}(y) & \\
      &=(\chi^{[l]}_{2,b}(x))^\top (g^{[l]})^{-1}\chi^{[l]}_{1,a}(y)& &\text{$LU$ factorization \eqref{facto}.}
\end{align*}
\end{proof}
We immediately deduce the
\begin{cor}
   For the  Christoffel--Darboux kernel we have
\begin{align}\label{abc}
    K^{[l]}(x,y)&=(\xi^{[l]}_2(x))^\top (g^{[l]})^{-1}\xi^{[l]}_1(y).
\end{align}
\end{cor}

\subsubsection{Christoffel--Darboux formula}
 In this subsection we derive a Christoffel--Darboux type formula from the symmetry property   \eqref{sym2}  of
 the moment matrix $g$. We need some preliminary lemmas
\begin{lemma}\label{lemma1}
The relations
\begin{align}
(g^{[l]})^{-1}  \Upsilon_1^{[l]}-(\Upsilon_2^{[l]})^\top(g^{[l]})^{-1}=(g^{[l]})^{-1}
\Big(g^{[l,\geq l]}(\Upsilon_{2}^{[l,\geq l]})^\top-\Upsilon_{1}^{[l,\geq l]}g^{[\geq l,l]}\Big)(g^{[l]})^{-1},
\end{align}
hold true.
\end{lemma}
\begin{proof}
  The first block of \eqref{sym2} is
\begin{align*}
  \Upsilon_1^{[l]}g^{[l]}+\Upsilon_{1}^{[l,\geq l]}g^{[\geq l,l]}= g^{[l]}(\Upsilon_2^{[l]})^\top+g^{[l,\geq
  l]}(\Upsilon_{2}^{[l,\geq l]})^\top,
\end{align*}
from where the result follows immediately.
\end{proof}
\begin{lemma}
  \label{lemma2}
We have
 \begin{align}
  \Upsilon_\ell^{[l]}\xi_\ell^{[l]}(x)&=x\xi_\ell^{[l]}(x)-\Upsilon^{[l,\geq l]}_{\ell}\xi^{[\geq l]}_{\ell}(x),& \ell&=1,2.
\end{align}
\end{lemma}
\begin{proof}
  It follows from the block decomposition of Definitions \ref{vec-block} and the eigen-value
  property of $\Upsilon_\ell$.
\end{proof}

 After a careful computation from  Definition \ref{sym gen} we get
\begin{lemma}\label{lemma3}
  If we assume that $l\geq \max(|\vec n_1|,|\vec n_2|)$ 
we can write
\begin{align}
  \Upsilon_{1}^{[l,\geq l]}&=\sum_{a =1 }^{p_1} e_{(l-1)_{-a}}e_{l_{+a}-l}^\top,
  &
    \Upsilon_{2}^{[l,\geq l]}&=\sum_{b =1}^{p_2}e_{\overline {(l-1)}_{-b}}e_{\bar l_{+b}-l}^\top.
\end{align}
Here $l_{\pm a}$ is the $\pm$ associated integer within the $\vec n_1$ composition, while $\bar l_{\pm b}$ is the $\pm$
associated integer for the $\vec n_2$ composition.
\end{lemma}

Finally, to derive a Christoffel--Darboux type  formula we need the following objects
\begin{definition}\label{associated linear forms}
 Associated polynomials are given by
\begin{align}
\begin{aligned}
A_{+a,a'}^{(l)}(y)&:=\chi_{1,a'}^{(l_{+a})}(y)-
       \begin{pmatrix}
    g_{l_{+a},0}&g_{l_{+a},1}&\cdots& g_{l_{+a},l-1}
  \end{pmatrix}(g^{[l]})^{-1}\chi_{1,a'}^{[l]}(y),\\
  \bar A^{(l)}_{-a,b'}(x)&:=(\chi_{2,b'}^{[l+1]}(x))^\top(g^{[l+1]})^{-1}e_{l_{-a}},\\
  A_{-b,a'}^{(l)}(y)&:=e_{\bar l_{-b}}^\top(g^{[l+1]})^{-1}\chi_{1,a'}^{[l+1]}(y),\\
  \bar A_{+b,b'}^{(l)}(x)&:=\Big(\chi_{2,b'}^{(\bar l_{+b})}(x)-(\chi_{2,b'}^{[l]}(x))^\top(g^{[l]})^{-1}
\begin{pmatrix}
    g_{0,\bar l_{+b}}\\g_{1,\bar l_{+b}}\\\vdots\\ g_{l-1,\bar l_{+b}}
  \end{pmatrix}\Big),
  \end{aligned}
\end{align}
with the  corresponding linear forms given by
  \begin{align}
  \begin{aligned}
    Q^{(l)}_{+a}&:=\sum_{a'=1}^{p_1}A^{(l)}_{+a,a'}w_{1,a'},&
      \bar  Q^{(l)}_{-a}&:=\sum_{b'=1}^{p_2}\bar A^{(l)}_{-a,b'}w_{2,b'},& a&=1,\dots, p_1,\\
          Q^{(l)}_{-b}&:=\sum_{a'=1}^{p_1}A^{(l)}_{-b,a'}w_{1,a'},&            \bar
          Q^{(l)}_{+b}&:=\sum_{b'=1}^{p_2}\bar A^{(l)}_{+b,b'}w_{2,b'},&b&=1,\dots,p_2.
          \end{aligned}
  \end{align}
\end{definition}

Then, we can show that
\begin{theorem}\label{th:cd}
Whenever $l\geq \max(|\vec n_1|,|\vec n_2|)$ the following Christoffel--Darboux type formulae
\begin{align*}
 (x-y)K_{a',b'}^{[l]}(x,y)&=
\sum_{b =1}^{p_2}\bar A_{+b,b'}^{(l)}(x) A_{-b,a'}^{(l-1)}(y)
-\sum_{a=1}^{p_1}\bar A^{(l-1)}_{-a,b'}(x)A^{(l)}_{+a,a'}(y),\\ (x-y)K^{[l]}(x,y)&=
\sum_{b =1}^{p_2}\bar Q_{+b}^{(l)}(x) Q_{-b}^{(l-1)}(y)
-\sum_{a=1}^{p_1}\bar Q^{(l-1)}_{-a}(x)Q^{(l)}_{+a}(y),
\end{align*}
hold.
  \end{theorem}
\begin{proof}
From Lemma  \ref{lemma1} we deduce
\begin{multline*}
(\chi^{[l]}_{2,b'}(x))^\top \big((g^{[l]})^{-1}
\Upsilon_1^{[l]}-(\Upsilon_2^{[l]})^\top(g^{[l]})^{-1}\big)\chi^{[l]}_{1,a'}(y)=(\chi^{[l]}_{2,b'}(x))^\top
(g^{[l]})^{-1}
\big(g^{[l,\geq l]}(\Upsilon_{2}^{[l,\geq l]})^\top-\Upsilon_{1}^{[l,\geq l]}g^{[\geq
l,l]}\big)(g^{[l]})^{-1}\chi^{[l]}_{1,a'}(y),
\end{multline*}
so that, recalling Theorem \ref{th:abc}, we get
\begin{align}
   \begin{aligned}(y-x)K_{b',a'}^{[l]}(x,y)=&(\chi^{[l]}_{2,b'}(x))^\top (g^{[l]})^{-1}
\big(g^{[l,\geq l]}(\Upsilon_{2}^{[l,\geq l]})^\top-\Upsilon_{1}^{[l,\geq l]}g^{[\geq
l,l]}\big)(g^{[l]})^{-1}\chi^{[l]}_{1,a'}(y)\\&+
  (\chi^{[l]}_{2,b'}(x))^\top (g^{[l]})^{-1}\Upsilon_{1}^{[l,\geq l]}\chi^{[\geq l]}_{1,a'}(y)-(\Upsilon_{2}^{[l,\geq
  l]}\chi^{[\geq l]}_{2,b'}(x))^\top(g^{[l]})^{-1}  (\chi^{[l]}_{1,a'}(y)),
  \end{aligned}\label{cd1}
\end{align}
or
\begin{align*}
   \begin{aligned}(x-y)K_{b',a'}^{(l-1)}(x,y)=&\Big((\chi^{[\geq l]}_{2,b'}(x))^\top- (\chi^{[l]}_{2,b'}(x))^\top
   (g^{[l]})^{-1}  g^{[l,\geq l]}\Big) (\Upsilon_{2}^{[l,\geq l]})^\top (g^{[l]})^{-1}
   \chi^{[l]}_{1,a'}(y)\\&
   -(\chi^{[l]}_{2,b'}(x))^\top (g^{[l]})^{-1} \Upsilon_{1}^{[l,\geq l]}\Big(\chi^{[\geq l]}_{1,a'}(y)-g^{[\geq
   l,l]}(g^{[l]})^{-1}\chi^{[l]}_{1,a'}(y)\Big).
     \end{aligned}
\end{align*}

Finally, from Lemma \ref{lemma3} we conclude
\begin{align*}
\Upsilon_{1}^{[l,\geq l]}\chi_{1,a'}^{[\geq l]}(y)&=
\sum_{a =1}^{p_1}e_{(l-1)_{-a}}\chi_{1,a'}^{(l_{+a})}(y),\\
  \Upsilon_{1}^{[l,\geq l]}g^{[\geq l,l]}&=\sum_{a=1}^{p_1}e_{(l-1)_{-a}}\begin{pmatrix}
    g_{l_{+a},0}&g_{l_{+a},1}&\cdots& g_{l_{+a},l-1}
  \end{pmatrix},\\
(\chi_{2,b'}^{[\geq l]}(x))^\top  (\Upsilon_{2}^{[l,\geq l]})^\top&=
\sum_{b =1}^{p_2}\chi_{2,b'}^{(\bar l_{+b})}(x)e_{\overline {(l-1)}_{-b}}^\top,\\
 g^{[l,\geq l]}( \Upsilon_{2}^{[l,\geq l]})^\top&=\sum_{b =1}^{p_2}\begin{pmatrix}
    g_{0,\bar l_{+b}}\\g_{1,\bar l_{+b}}\\\vdots\\ g_{l-1,\bar l_{+b}}
  \end{pmatrix}e_{\overline {(l-1)}_{-b}}^\top,
\end{align*}
and consequently
\begin{align}
 \begin{aligned}(x-y)K_{b',a'}^{[l]}(x,y)=&
\sum_{b =1}^{p_2}\Big(\chi_{2,b'}^{(\bar l_{+b})}(x)-(\chi_{2,b'}^{[l]}(x))^\top(g^{[l]})^{-1}
\begin{pmatrix}
    g_{0,\bar l_{+b}}\\g_{1,\bar l_{+b}}\\\vdots\\ g_{l-1,\bar l_{+b}}
  \end{pmatrix}
\Big)e_{\overline {(l-1)}_{-b}}^\top(g^{[l]})^{-1}\chi_{1,a'}^{[l]}(y)\\&
 -\sum_{a=1}^{p_1}(\chi_{2,b'}^{[l]}(x))^\top(g^{[l]})^{-1}e_{(l-1)_{-a}}\big(\chi_{1,a'}^{(l_{+a})}(y)-
       \begin{pmatrix}
    g_{l_{+a},0}&g_{l_{+a},1}&\cdots& g_{l_{+a},l-1}
  \end{pmatrix}(g^{[l]})^{-1}\chi_{1,a'}^{[l]}(y)\big).
    \end{aligned}\label{cd4}
  \end{align}   Recalling  Definition \ref{associated linear forms} we get the announced result.

\end{proof}

The associated linear forms are identified with linear forms of multiple orthogonal polynomials as follows
\begin{pro}\label{associated forms as mop}
  We have the formulae
  \begin{align}
  \begin{aligned}
    Q_{+a}^{(l)}&=Q^{(\rII,a)}_{[\vec\nu_1(l-1)+\vec e_{1,a};\vec \nu_2(l-1)]},&
      Q_{-b}^{(l)}&=Q^{(\rI,b)}_{[\vec\nu_1(l);\vec \nu_2(l)-\vec e_{2,b}]},&
      \bar  Q_{+b}^{(l)}&=\bar Q^{(\rII,b)}_{[\vec\nu_2(l-1)+\vec e_{2,b};\vec \nu_1(l-1)]},&
     \bar  Q_{-a}^{(l)}&=Q^{(\rI,a)}_{[\vec\nu_2(l);\vec \nu_1(l)-\vec e_{1,a}]}.
     \end{aligned}
  \end{align}
\end{pro}
\begin{proof}
See Appendix \ref{II}.
\end{proof}

Proposition \ref{associated forms as mop} allows us to give the following form of the Christoffel--Darboux formula
stated in Theorem \ref{th:cd}
\begin{pro}\label{pro:cd}
 For  $l\geq \max(|\vec n_1|,|\vec n_2|)$ the following
  \begin{align}
  \begin{aligned}
 (x-y)K^{[l]}(x,y)=&
\sum_{b =1}^{p_2}\bar Q^{(\rII,b)}_{[\vec\nu_2(l-1)+\vec e_{2,b};\vec \nu_1(l-1)]}(x)
Q^{(\rI,b)}_{[\vec\nu_1(l-1);\vec \nu_2(l-1)-\vec e_{2,b}]}(y)\\
&-\sum_{a=1}^{p_1}\bar Q^{(\rI,a)}_{[\vec\nu_2(l-1);\vec \nu_1(l-1)-\vec
e_{1,a}]}(x)Q^{(\rII,a)}_{[\vec\nu_1(l-1)+\vec e_{1,a};\vec \nu_2(l-1)]}(y).
\end{aligned}
    \label{cd3}
    \\
     \begin{aligned}
 (x-y)K_{b',a'}^{[l]}(x,y)=&
\sum_{b =1}^{p_2}\bar A^{(\rII,b)}_{[\vec\nu_2(l-1)+\vec e_{2,b};\vec \nu_1(l-1)],b'}(x)
A^{(\rI,b)}_{[\vec\nu_1(l-1);\vec \nu_2(l-1)-\vec e_{2,b}],a'}(y)\\
&-\sum_{a=1}^{p_1}\bar A^{(\rI,a)}_{[\vec\nu_2(l-1);\vec \nu_1(l-1)-\vec
e_{1,a}],b'}(x)A^{(\rII,a)}_{[\vec\nu_1(l-1)+\vec e_{1,a};\vec \nu_2(l-1)],a'}(y).
\end{aligned}
\end{align}
holds.
\end{pro}
 Relation \eqref{cd3}  is precisely the Christoffel--Darboux formula derived in \cite{daems-kuijlaars}, the difference
 here is that derivation is based on the
 Gauss--Borel factorization problem for the moment matrix; i.e. only on algebraic arguments, and not in the Riemann--Hilbert problem found in
 \cite{daems-kuijlaars}, and hence the conditions on the weights are not so restrictive. However, the reader should
 notice that the Christoffel--Darboux kernel does not depend on the ladder determined by the composition vectors $\vec
 n_1,\vec n_2$, but only on the degree vectors $\vec\nu_1(l-1)$ and $\vec \nu_2(l-1)$. This was noticed in
 \cite{daems-kuijlaars0} for type I multiple orthogonality.
\begin{pro}\label{det-associated}
The associated polynomials introduced in Definition \ref{associated
linear forms} have the following  determinantal expressions
 \begin{align}\label{detmops+}
A_{+a,a'}^{(l)}&=\frac{1}{\det g^{[l]}}\det
  \left(\begin{BMAT}{cccc|c}{cccc|c}
    g_{0,0}&g_{0,1}&\cdots&g_{0,l-1}&\chi_{1,a'}^{(0)}\\
     g_{1,0}&g_{1,1}&\cdots&g_{1,l-1}&\chi_{1,a'}^{(1)}\\
     \vdots &\vdots&            &\vdots&\vdots\\
        g_{l-1,0}&g_{l-1,1}&\cdots&g_{l-1,l-1}&\chi_{1,a'}^{(l-1)}\\
          g_{l_{+a},0}&g_{l_{+a},1}&\cdots&g_{l_{+a},l-1}&\chi_{1,a'}^{(l_{+a})}
\end{BMAT}\right),
\end{align}
\begin{align}
 \label{detdualmops-}
  \bar A^{(l)}_{-a,b'}(x)&=\frac{(-1)^{l+ l_{-a}}}{\det g^{[l+1]}}\det \left(\begin{BMAT}{ccc|c}{ccc;ccc|c}
    g_{0,0}&\cdots&g_{0,l-1}&g_{0,l}\\
    \vdots&&\vdots&\vdots\\
     g_{l_{-a}-1,0}&\cdots&g_{l_{-a}-1,l-1}&g_{l_{-a}-1,l}\\
       g_{l_{-a}+1,0}&\cdots&g_{l_{-a}+1,l-1}&g_{l_{-a}+1,l}\\
      \vdots&&\vdots&\vdots\\
      g_{l-1,0}&\cdots&g_{l-1,l-1}&g_{l-1,l}\\
  \chi^{(0)}_{2,b'}&\cdots&\chi^{(l-1)}_{2,b'}&\chi^{(l)}_{2,b'}
  \end{BMAT}\right),
  \end{align}
\begin{align}
   \label{detmops-}
  A_{-b,a'}^{(l)}&=\frac{(-1)^{l+\bar l_{-b}}}{\det g^{[l+1]}}\det \left(\begin{BMAT}{ccc;ccc|c}{ccc|c}
    g_{0,0}&\cdots&g_{0,\bar l_{-b}-1}&g_{0,\bar l_{-b}+1}&\cdots&g_{0,l-1}&\chi^{(0)}_{1,a'}\\
    \vdots&&\vdots&\vdots&    &\vdots&\vdots\\
    g_{l-1,0}&\cdots&g_{l-1,\bar l_{-b}-1}&g_{l-1,\bar l_{-b}+1}&\cdots&g_{l-1,l-1}&\chi^{(l-1)}_{1,a'}\\
 g_{l,0}&\cdots&g_{l,\bar l_{-b}-1}&g_{l,\bar l_{-b}+1}&\cdots&g_{l,l-1}&\chi^{(l)}_{1,a'}
\end{BMAT}\right),
\end{align}
\begin{align}
 \label{detdualmops+}
  \bar A_{+b,b'}^{(l)}&=\frac{1}{\det g^{[l]}}\det
  \left(\begin{BMAT}{cccc|c}{cccc|c}
    g_{0,0}&g_{0,1}&\cdots&g_{0,l-1}& g_{0,\bar l_{+b}}\\
     g_{1,0}&g_{1,1}&\cdots&g_{1,l-1}& g_{1,\bar l_{+b}} \\
     \vdots &\vdots&        &\vdots&\vdots\\
        g_{l-1,0}&g_{l-1,1}&\cdots&g_{l-1,l-1}&g _{l-1,\bar l_{+b}}\\
        \chi_{2,b'}^{(0)}& \chi_{2,b'}^{(1)} & \cdots & \chi_{2,b'}^{(l-1)} & \chi_{2,b'}^{(\bar l_{+b})}
\end{BMAT}\right).
\end{align}
\end{pro}

\section{Connection with the  multi-component 2D Toda Lattice hierarchy}

In this section we introduce deformations of the Gauss--Borel factorization problem that give the
connection with the theory of the integrable
hierarchies of 2D Toda lattice type, in the multi-component flavor case. First, we introduce the continuous flows and
then the discrete ones. Let us stress that
both flows could be considered simultaneously but we consider them separately for the sake of simplicity and clearness in the exposition.

\subsection{Continuous deformations of the moment matrix}

\begin{definition}
The deformed moment matrix is given by
  \begin{align}\label{g.evol}
  g_{\vec n_1,\vec n_2}(t)&:=W_{0,\vec n_1}(t)g\bar W_{0,\vec n_2}(t)^{-1},
\end{align}
where we  use the following semi-infinite matrices
\begin{align*}
 W_{0,\vec n_1}(t)&:=\sum_{a=1}^{p_1}\exp\Big(\sum_{j=1}^\infty t_{j,a}\Lambda_{1,a}^{j}\Big)\in G_+,&
\bar W_{0,\vec n_2}(t)&:=\sum_{b=1}^{p_2}
\exp\Big(\sum_{j=1}^\infty \bar t_{j,b}(\Lambda_{2,b}^\top)^{j}\Big)\in G_-
\end{align*}
depending   on $t=(t_{j,a},\bar t_{j,b})_{j,a,b}$ with $t_{j,a},\bar t_{j,b}\in\R$, $j=1,2,\dots$, $a=1,\dots,p_1$ and $b=1,\dots,p_2$.
\end{definition}

As in the previous section and when the context is clear enough we will drop the subscripts associated with the compositions
$\vec n_1$ and $\vec n_2$. The reader should notice that the following semi-infinite matrices are well defined
\begin{align*}
 (W_{0,\vec n_1}(t)^{-1})^\top&=\sum_{a=1}^{p_1}\exp\Big(-\sum_{j=1}^\infty t_{j,a}(\Lambda_{1,a}^\top)^{j}\Big)\in G_-,&
(\bar W_{0,\vec n_2}(t)^{-1})^\top&=\sum_{b=1}^{p_2}
\exp\Big(-\sum_{j=1}^\infty \bar t_{j,b}\Lambda_{2,b}^{j}\Big)\in G_+.
\end{align*}

This deformation preserves the structure that characterizes $g$ as a moment matrix, in fact we have
\begin{theorem}
  The matrix $g(t)$ is a moment matrix with new ``deformed weights" given by
   \begin{align}\label{evolved}
 \begin{aligned}
    w_{1,a}(x,t)&=\Es_a(x,t)w_{1,a}(x),& \Es_a&:=  \exp\big(\sum_{j=1}^\infty t_{j,a}x^{j}\big),\\
    w_{2,b}(x,t)&=\bar\Es_b(x,t)^{-1}w_{2,b}(x), & \bar \Es_b&:=\exp\big(\sum_{j=1}^\infty \bar t_{j,b}x^{j}\big).
 \end{aligned}
\end{align}
\end{theorem}
\begin{proof}
Observe that
 \begin{align*}
   W_0(t)&=\sum_{j\geq 0}\sum_{a=1}^{p_1}\sigma^{(a)}_j(t) \Lambda_{1,a}^{j} ,&
    \bar  W_0(t)^{-1}=\sum_{j\geq 0}\sum_{b=1}^{p_2}(\Lambda_{2,b}^\top)^{j}\bar\sigma^{(b)}_j(t) ,
   \end{align*}
 where $\sigma^{(a)}_j $ is the $j$-th elementary Schur polynomial in the variables $t_{ja}$ and $\bar\sigma^{(j)}_b$ is also an
 elementary Schur polynomial but now in the variables $-\bar t_{j,b}$. To prove \eqref{evolved} we first discuss the action of
 $\Lambda_{1,a}$ and $\Lambda_{2,b}^{\top}$ on $g$ explicitly. Recalling \eqref{explicit g2} it is straightforward to
 see that
\begin{equation*}
    (\Lambda_{1,a}g\Lambda_{2,b}^{\top})_{i,j}=\int  x^{k_1(i)+1}w_{1,a_1(i)}(x)w_{2,a_2(j)}(x)x^{k_2(j)+1}
    \delta_{a_1(i),a} \delta_{a_2(j),b}
    \d \mu(x),
\end{equation*}
and consequently the following expression holds
\begin{equation*}
    (W_0 g \bar W_0^{-1})_{i,j}=\sum_{a=1}^{p_1} \sum_{b=1}^{p_2}\int  x^{k_1(i)}\Big(\sum_{l \geq 0} \sigma_l^{(a)}
    x^l\Big) w_{1,a_1(i)}(x)w_{2,a_2(j)}(x)\Big(\sum_{m \geq 0} \bar \sigma_m^{(b)} x^m\Big) x^{k_2(j)}
    \delta_{a_1(i),a} \delta_{a_2(j),b}
    \d \mu(x),
\end{equation*}
that leads directly to \eqref{evolved}.

\end{proof}
That the sign definition of the  weights is preserved under deformations is ensured by the fact that all times $t$ are  real. Let us comment that these deformations could be also be considered as evolutions, and from hereon we
indistinctly talk about deformation/evolution. If the initial measures have bounded support then there is no problem
with the exponential behavior at $\infty$ of the $\Es$ factors;  however, for unbounded situations a discussion is
needed for each case.

The Gauss--Borel  factorization problem
\begin{align}\label{facto-t}
  g_{\vec n_1,\vec n_2}(t)&=S(t)^{-1}\bar S(t),& S(t)&\in G_-, &\bar S (t)&\in G_+,
  \end{align}
with $S(t)$ lower triangular and $\bar S(t)$ upper triangular, will give the connection with integrable systems of Toda type.

Let us assume that the weights in $(\vec w_1,\vec w_2)$ are of the form \eqref{theform} and that conform an M-Nikishin, then  Theorem \ref{suffcondnik} indicates that for  small values of the times the new weights are also in the M-Nikishin class, ensuring that $(\vec w_1(t),\vec w_2(t),\mu)$ is a perfect system and therefore the Gauss--Borel factorization makes sense.



\subsection{Lax equations and the integrable hierarchy}
 Let us introduce the Lax machinery associated with the Gauss--Borel  factorization  that will lead to a multi-component 2D
 Toda lattice hierarchy as described in
 \cite{manas-martinez-alvarez}:
 \begin{definition}
 Associated with the deformed Gauss--Borel factorization we consider
   \begin{enumerate}
\item  Wave semi-infinite  matrices
\begin{align}
  W(t)&:=S(t)W_0(t),& \bar  W(t)&:=\bar S(t)\bar W_0(t).
\end{align}
\item Wave
\begin{align}\label{defbaker}
  \Psi_a(z,t)&:=W(t)\chi_{1,a}(z),& \bar\Psi_b(z,t)&:=\bar W(t)\chi_{2,b}^*(z),
\end{align}
and adjoint  wave semi-infinite vector  functions\footnote{In this point the reader should notice that there are
two differences between this definition of wave functions (also known as Baker--Akheizer functions) and the
one common in the literature, see for example \cite{adler-van moerbeke 1}. Our modifications are motivated  by two
facts, i) we prefer $\bar\Psi^*_b$ to
be a polynomial in $z$ and not in $z^{-1}$, up to  plane-wave factors, ii) we choose to have a direct
connection between wave functions and Cauchy
transforms of polynomials, with no $z^{-1}$ factors multiplying the Cauchy transforms when identified with wave
functions. If we denote by small $\psi$
the wave functions corresponding to the scheme of for example \cite{adler-van moerbeke 1} then we should have the
following correspondence
$\Psi_a^{(l)}(z)\leftrightarrow\psi^{(l)}_a(z)$,  $ z(\Psi_a^*)^{(l)}(z)\leftrightarrow(\psi^*_a)^{(l)}(z)$,
$z^{-1}\bar\Psi_b^{(l)}(z^{-1}
)\leftrightarrow \bar\psi^{(l)}_b(z)$ and $(\bar \Psi_b^*)^{(l)}(z^{-1})\leftrightarrow(\bar\psi^*_b)^{(l)}(z)$.}
\begin{align}\label{adjoint.baker}
  \Psi_a^*(z,t)&:= (W(t)^{-1})^\top\chi_{1,a}^*(z),&\bar \Psi_b^*(z,t)&:=(\bar W(t)^{-1})^\top\chi_{2,b}(z).
\end{align}

 \item Lax semi-infinite matrices
\begin{align}\label{deflax}
  L_a(t)&:=S(t)\Lambda_{1,a}S(t)^{-1},&
   \bar L_b(t)&:=\bar S(t)\Lambda_{2,b}^\top\bar S(t)^{-1}.
\end{align}
\item Zakharov--Shabat semi-infinite matrices
\begin{align}\label{defzs}
B_{j,a}&:=(L^j_a)_+,& \bar B_{j,b}&:=(\bar L^j_{b})_-,
  \end{align}
  where the subindex $+$ indicates the projection in the upper triangular matrices while the subindex
   $-$ the projection in the strictly lower triangular matrices.
     \end{enumerate}
 \end{definition}

Observe that
\begin{align}
  L_a\Psi_{a'}&=\delta_{a,a'}z\Psi_{a'},& \bar L_b^\top\bar\Psi_{b'}^*&=\delta_{b,b'}z^{-1}\bar\Psi_{b'}^*.
\end{align}
We  also mention that the matrices $S$ and $\bar S$ correspond to the Sato operators (also known as gauge operators)
of the integrable hierarchy we are deling with. Some times \cite{bergvelt} the operators $L_a$ are referred as
resolvents and the Lax name is reserved only for a convenient linear combination of the resolvents.

The reader should notice that as $S(t)\in G_-$ and $W_0(t)\in G_+$ the product $W(t)=S(t)W_0(t)$ is well defined as its coefficients are finite sums instead of series,
for $(\bar W(t)^{-1})^\top=(\bar S(t)^{-1})^{\top}(\bar W_0(t))^{-1})^{\top}$ we can apply the previous argument and therefore the product is well defined. However,
$(W(t)^{-1})^{\top}=(S(t)^{-1})^\top (W_0(t)^{-1})^\top$ is a product of elements which involves series instead of finite sum and its existence is not in principle ensured. The situation is reproduced with $\bar W(t)=\bar S(t)\bar W_0(t)$, and the existence of the product is not guaranteed. However, we notice that the simultaneous consideration of the factorization problems \eqref{facto-t} and \eqref{facto} leads $S(t)^{-1}\bar S (t)=W_0(t) S^{-1} \bar S \bar W_0(t)^{-1}$ that shows two products involving series, namely $W_0(t)S^{-1}$ and $\bar S\bar W_0(t)^{-1}$, but they are well defined if we assume the existence of both $LU$ factorizations. From hereon we give for granted the existence of $\bar W$ and $W^{-1}$, and as we will see they indeed involve series, which in the convergent situation lead to Cauchy transforms.

\begin{pro}\label{conti}
For the wave functions we have
    \begin{align}
    \Psi_a^{(k)}(z,t)&=A^{(k)}_a(z,t)\Es_a(z,t),&
    ( \bar \Psi_b^*)^{(k)}(z,t)&=\bar A_b^{(k)}(z,t)\bar\Es_b(z,t)^{-1},
\end{align}
where $A^{(k)}_a(x,t),\bar A_b^{(k)}(x,t)$ are the multiple orthogonal polynomials and  dual polynomials (in the $x$
variable)  corresponding to
\eqref{evolved}.
The evolved  linear forms,  associated with weights \eqref{evolved}, are
\begin{align}
  Q^{(k)}(x,t)&:=\sum_{a=1}^{p_1}A_a^{(k)}(x,t)w_{1,a}(x,t)=\sum_{a=1}^{p_1}\Psi_a^{(k)}(x,t)w_{1,a}(x),\\    \bar
  Q^{(k)}(x,t)&:=\sum_{b=1}^{p_2}(\bar
  A_b^*)^{(k)}(x,t)w_{2,b}(x,t)=\sum_{b=1}^{p_2}(\bar\Psi_b^*)^{(k)}(x,t)w_{2,b}(x),
\end{align}
which are bi-orthogonal polynomials of mixed type for each $t$
\begin{align}
  \int  Q^{(l)}(t,x)\bar Q^{(k)}(t,x)\d \mu(x)&=\delta_{l,k},& l,k\geq 0,
\end{align} and
   \begin{align}
   \bar  \Psi_b^{(k)}(z,t)&=\int \frac{Q^{(k)}(x,t)}{z-x}w_{2,b}(x)\d \mu(x),&
    ( \Psi_a^*)^{(k)}(z,t)&=\int \frac{\bar Q^{(k)}(x,t)}{z-x}w_{1,a}(x)\d \mu(x).
   \end{align}
\end{pro}
\begin{proof}
 From the definitions \eqref{defbaker} and \eqref{adjoint.baker}, and the factorization problem $Wg=\bar W$ we
 conclude
\begin{align}\label{baker}
 \bar\Psi_b&=\bar W\chi_{2,b}^*=S(W_0g)\chi_{2,b}^*,&
  \Psi_a^*&=(W^{-1})^\top\chi_{1,a}^*=(\bar S^{-1})^\top (g\bar W_0^{-1})^\top \chi_{1,a}^*.
\end{align}
We get,  in terms of the linear forms, the following identities
   \begin{align*}
   \bar  \Psi_b^{(k)}(z,t)&=\int \frac{Q^{(k)}(x,t)}{z-x}w_{2,b}(x)\d \mu(x),&
    ( \Psi_a^*)^{(k)}(z,t)&=\int \frac{\bar Q^{(k)}(x,t)}{z-x}w_{1,a}(x)\d \mu(x),
   \end{align*}
where the Cauchy transforms  are understood as  before.\footnote{The reader should notice that there is a difference
in this semi-infinite context, appropriate for the construction of multiple orthogonal polynomials, and the
bi-infinite case which is the one considered in \cite{ueno-takasaki}. In the present context we do not have
expressions, as we do have in the bi-infinite
situation, of the form
   \begin{align*}
   \bar  \Psi_b^{(k)}(z,t)&=(P_0+P_1 z^{-1}+\cdots)\exp\big(\sum_{j> 0}\bar t_{j,b}z^{j}\big),\\
    ( \Psi_a^*)^{(k)}(z,t)&=(Q_0+Q_1 z^{-1}+\cdots)\exp\big(-\sum_{j> 0}t_{j,a}z^j\big).
   \end{align*}
   The reason for this issue is rooted into non-invertibility of $\Lambda_a$. Indeed, for the semi-infinite case, we
   have
\begin{align*}
(  \Lambda_a^\top)^j\chi_a^*&=\big[z^j\chi_a^*\big]_-&&\Longrightarrow
& \exp\big(\sum_{j=1}^\infty c_j(\Lambda_a^{\top})^j\big)\chi_a^*=[\exp\big(\sum_{j=1}^\infty c_jz^j\big)\chi_a^*]_-
\end{align*}
where the subindex $-$ stands for the negative powers in $z$ in the Laurent expansion; while in the bi-infinite case
we drop the $-$ subindex in the previous
formulae.}
\end{proof}
We must stress in this point that these functions are not the evolved second kind functions  of the linear forms
   \begin{align}\label{evolved.ct}
   \bar  C_b^{(k)}(z,t)&:=\int \frac{Q^{(k)}(x,t)}{z-x}w_{2,b}(x,t)\d \mu(x),&
    (C_a)^{(k)}(z,t)&:=\int \frac{\bar Q^{(k)}(x,t)}{z-x}w_{1,a}(x,t)\d \mu(x).
   \end{align}
\begin{theorem}\label{proposition: the hierarchy}
For $j,j'=1,2,\dots$, $a,a'=1,\dots,p_1$ and $b,b'=1,\dots,p_2$ the following differential relations hold
  \begin{enumerate}
    \item Auxiliary linear systems for the wave matrices
    \begin{align}\label{auxlinsys}
      \frac{\partial W}{\partial t_{j,a}}&=B_{j,a} W, &\frac{\partial W}{\partial \bar t_{j,b}}&=\bar B_{j,b} W,&
      \frac{\partial \bar W}{\partial t_{j,a}}&=B_{j,a} \bar W, &\frac{\partial \bar W}{\partial \bar
      t_{j,b}}&=\bar B_{j,b} \bar W.
    \end{align}
    \item Linear systems for the wave and adjoint wave semi-infinite matrices
       \begin{align}\label{linbaker}
      \frac{\partial \Psi_{a'}}{\partial t_{j,a}}&=B_{j,a}\Psi_{a'}, &\frac{\partial \Psi_{a'}}{\partial \bar
      t_{j,b}}&=\bar B_{j,b} \Psi_{a'},&
      \frac{\partial \bar \Psi_{b'}}{\partial t_{j,a}}&=B_{j,a} \bar \Psi_{b'}, &\frac{\partial \bar
      \Psi_{b'}}{\partial \bar t_{j,b}}&=\bar B_{j,b} \bar \Psi_{b'},\\
      \frac{\partial \Psi^*_{a'}}{\partial t_{j,a}}&=-B_{j,a}^\top\Psi_{a'}^*, &\frac{\partial
      \Psi_{a'}^*}{\partial \bar t_{j,b}}&
      =-\bar B_{j,b}^\top \Psi_{a'}^*,&
      \frac{\partial \bar \Psi_{b'}^*}{\partial t_{j,a}}&=-B_{j,a}^\top \bar \Psi_{b'}^*, &\frac{\partial \bar
      \Psi_{b'}^*}{\partial \bar t_{j,b}}&=-\bar B_{j,b}^\top
      \bar \Psi_{b'}^*.
    \end{align}
    \item Linear systems for multiple orthogonal polynomials and their duals
\begin{align}\label{flow.mop}
\begin{aligned}
  \frac{\partial \A_{a'}}{\partial t_{j,a}}&=(B_{j,a}-\delta_{a,a'}x^j)\A_{a'},&  \frac{\partial \A_{a'}}{\partial
  \bar t_{j,b}}&=(\bar B_{j,b})\A_{a'},&
 \frac{\partial \bar \A_{b'}}{\partial t_{j,a}}&=-B_{j,a}^\top\bar\A_{b'},&  \frac{\partial \bar\A_{b'}}{\partial \bar
 t_{j,b}}&=(-\bar
 B_{j,b}^\top+\delta_{b,b'}x^j)\bar\A_{b'}.
\end{aligned}
\end{align}
    \item Lax equations
      \begin{align}\label{laxeq}
        \frac{\partial L_{a'}}{\partial t_{j,a}}&=[B_{j,a},L_{a'}], &\frac{\partial L_{a'}}{\partial \bar
        t_{j,b}}&=[\bar B_{j,b},L_{a'}], &
      \frac{\partial \bar L_{b'}}{\partial t_{j,a}}&=[B_{j,a}, \bar L_{b'}], &\frac{\partial \bar L_{b'}}{\partial
      \bar t_{j,b}}&=[\bar B_{j,b}, \bar L_{b'}].
    \end{align}
    \item Zakharov--Shabat equations
    \begin{align}\label{zseq}
      \frac{\partial B_{j,a}}{\partial t_{j',a'}}-
          \frac{\partial B_{j',a'}}{\partial t_{j,a}} +[B_{j,a},B_{j',a'}]&=0,\\
           \frac{\partial \bar B_{j,b}}{\partial \bar t_{j',b'}}-
          \frac{\partial \bar B_{j',b'}}{\partial \bar t_{j,b}} +[\bar B_{j,b},\bar B_{j',b'}]&=0,\\
           \frac{\partial B_{j,a}}{\partial \bar t_{j',b'}}-
          \frac{\partial \bar B_{j',b'}}{\partial t_{j,a}} +[B_{j,a},\bar B_{j',b'}]&=0.
    \end{align}
  \end{enumerate}
\end{theorem}
\begin{proof}
To prove \eqref{auxlinsys} we proceed as follows. In the first place we compute
 \begin{align*}
    \frac{\partial W_0}{\partial t_{j,a}}&=\Lambda_{1,a}^{j}W_0,  &\frac{\partial \bar W_0}{\partial \bar
    t_{j,b}}&=(\Lambda_{2,b}^{\top})^{j}\bar W_0,
 \end{align*}
 and in the second place we observe that
 \begin{align}
 \frac{\partial W}{\partial t_{j,a}} &=\Big(\frac{\partial S}{\partial t_{j,a}}S^{-1}+L_a^{j}\Big)W,&
  \frac{\partial \bar W}{\partial t_{j,a}} &=\Big(\frac{\partial \bar S}{\partial t_{j,a}}S^{-1}\Big)\bar W, \label{desfac1}\\
  \frac{\partial  W}{\partial \bar t_{j,b}}&=\Big(\frac{\partial  S}{\partial \bar t_{j,b}}  S^{-1}\Big) W,&
  \frac{\partial \bar W}{\partial \bar t_{j,b}}&=\Big(\frac{\partial \bar S}{\partial \bar t_{j,b}} \bar S^{-1}+
 \bar L_b^{j}\Big) \bar W.\label{desfac2}
 \end{align}
Now, using the factorization problem  we get
 \begin{align*}
\frac{\partial S}{\partial t_{j,a}}S^{-1}+L_a^{j}
 =\frac{\partial \bar S}{\partial t_{j,a}} \bar S^{-1},
\\
\frac{\partial \bar S}{\partial \bar t_{j,b}} \bar S^{-1}+
 \bar L_b^{j}
 =\frac{\partial S}{\partial \bar t_{j,b}} S^{-1},
 \end{align*}
which, taking the $+$ part (upper triangular) and the $-$ part (strictly lower triangular) imply
\begin{align}\label{desfac3}
\frac{\partial S}{\partial t_{j,a}}S^{-1}&=-(L_a^{j})_-, &\frac{\partial \bar S}{\partial t_{j,a}} \bar
S^{-1}&=(L_a^{j})_+, & \frac{\partial \bar S}{\partial \bar
t_{j,b}} \bar S^{-1}&=-(\bar L_b^{j})_+,
& \frac{\partial S}{\partial \bar t_{j,b}} S^{-1}&=(\bar L_b^{j})_-,
\end{align}
so using \eqref{desfac3} into \eqref{desfac1} and \eqref{desfac2} with the definitions \eqref{defzs} we obtain
\eqref{auxlinsys}.
The linear system \eqref{linbaker} is obtained by inserting  \eqref{defbaker} into \eqref{auxlinsys}.

To obtain the Lax equations \eqref{laxeq} 
we take derivatives of \eqref{deflax}
\begin{align*}
\frac{\partial L_{a'}}{\partial t_{j,a}}&=\Big[\frac{\partial S}{\partial t_{j,a}}S^{-1},L_{a'}\Big]=[B_{j,a},L_{a'}],&
\frac{\partial \bar L_{b'}}{\partial t_{j,a}}&=\Big[\frac{\partial \bar S}{\partial t_{j,a}}\bar S^{-1},\bar L_{b'}\Big]=[B_{j,a},\bar L_{b'}], \\
\frac{\partial L_{a'}}{\partial \bar t_{j,b}}&=\Big[\frac{\partial S}{\partial \bar t_{j,b}}S^{-1},L_{a'}\Big]=[\bar B_{j,b},L_{a'}],&
\frac{\partial \bar L_{b'}}{\partial \bar t_{j,b}}&=\Big[\frac{\partial \bar S}{\partial \bar t_{j,b}}\bar S^{-1},\bar L_{b'}\Big]=[\bar B_{j,b},\bar
L_{b'}].
\end{align*}
Finally, \eqref{zseq} are obtained as compatibility conditions for \eqref{auxlinsys}.
\end{proof}
  All these equations provide us with different descriptions of  a multi-component integrable hierarchy of the 2D Toda
  lattice hierarchy type that rules the
  flows of the multiple orthogonal polynomials with respect to  deformed weights. This integrable hierarchy is
  the Toda type extension of the
  multi-component KP hierarchy considered in \cite{bergvelt}.

\subsection{Darboux--Miwa discrete  flows}
We complete the previously considered  continuous flows with discrete flows, which we introduce through an iterated
application of  Darboux
transformations \cite{adler-van moerbeke 2}.

\begin{definition}
Given  sequences of complex numbers
 \begin{align}
 \begin{aligned}
 \lambda_a&:=\{\lambda_a(n)\}_{n\in\Z}\subset\C, & a&=1,\dots,p_1,&
 \bar\lambda_b&:= \{\bar \lambda_b(n)\}_{n\in\Z}\subset\C, & b&=1,\dots,p_2,
 \end{aligned}
 \end{align}
(where $\bar\lambda$ is not intended to denote the complex conjugate of $\lambda$) and two  vectors,
$(s_1,\dots,s_{p_1})\in\Z^{p_1}$ and $(\bar s_1,\dots,\bar s_{p_2})\in\Z^{p_2}$, we construct the following
semi-infinite
 matrices
 \begin{align}
 \begin{aligned}
   D_0&:=\sum_{a=1}^{p_1}D_{0,a}, & D_{0,a}&:=\begin{cases}
     \prod_{n=1}^{s_a}(\Lambda_{1,a}-\lambda_a(n)\Pi_{1,a}), &s_a>0,\\
     \Pi_{1,a},& s_a=0,\\
       \prod_{n=1}^{|s_a|}(\Lambda_{1,a}-\lambda_a(-n)\Pi_{1,a})^{-1},& s_a<0,
   \end{cases}\\
   \bar   D_0^{-1}&:=   \sum_{b=1}^{p_2} \big(\bar D_{0}^{-1}\big)_b & \big(\bar D_{0}^{-1}\big)_{b}&
   :=\begin{cases}
     \prod_{n=1}^{\bar s_b}(\Lambda_{2,b}^\top-\bar \lambda_b(n)\Pi_{2,b}), &\bar s_b>0,\\
     \Pi_{2,b},&\bar s_b=0,\\
       \prod_{n=1}^{|\bar s_b|}(\Lambda_{2,b}^\top-\bar \lambda_b(-n)\Pi_{2,b})^{-1},&\bar s_b<0,
   \end{cases}
\end{aligned}
 \end{align}
  where $s:=\{s_a,\bar s_b\}_{\substack{a=1,\dots,p_1\\b=1,\dots,p_2}}$ denotes the set of discrete times,
  in terms of which we define the  deformed moment matrix
 \begin{align}
   g(s)=D_0(s)g\bar D_0(s)^{-1}.
 \end{align}
\end{definition}

\begin{pro}
 The moment matrix $g(s)$ has the same form as the moment matrix $g$ but with new weights
\begin{align}\label{d.evolved}\begin{aligned}
  w_{1,a}(s,x)&=\Ds_{a}(x,s_a)w_{1,a}(x),& \Ds_a&:=\begin{cases}
     \prod_{n=1}^{s_a}(x-\lambda_a(n)), &s_a>0,\\
     1,& s_a=0,\\
       \prod_{n=1}^{|s_a|}(x-\lambda_a(-n))^{-1},& s_a<0,
   \end{cases}\\
   w_{2,b}(s,x)&=\bar \Ds_b(x,\bar s_b)^{-1}w_{2,b}(x),& \bar \Ds_b^{-1}&:=\begin{cases}
     \prod_{n=1}^{\bar s_b}(x-\bar \lambda_b(n)), &\bar s_b>0,\\
 1,&\bar s_b=0,\\
       \prod_{n=1}^{|\bar s_b|}(x-\bar \lambda_b(-n))^{-1},&\bar s_b<0,
   \end{cases}\end{aligned}
\end{align}
\end{pro}

Thus, the proposed discrete evolution introduces new zeroes and poles in the weights at the points  defined by
sequences of
$\lambda$'s. For example,  in the $a$-th direction, the $s_a$ flow in the positive direction, $s_a\to s_a+1$, introduces
a new zero at the point
$\lambda_a(s_a+1)$, while  if we move in the negative direction, $s_a\to s_a-1$, it introduces a simple pole at
$\lambda_a(s_a-1)$. Let us stress that for the time being we have not ensured the reality and positiveness/negativeness  of the evolved
weights, this will be considered later on.

\subsubsection{Miwa transformations}
Here we show that the discrete flows just introduced can be reproduced with the aid of Miwa shifts in the continuous
variables.
\begin{definition}
We consider two types of Miwa transformations:
\begin{enumerate}
\item We introduce the following time
shifts
\begin{align}
  t \rightarrow t\mp[z^{-1}]_a:=\Big\{t_{j,a'}  \mp\delta_{a',a}\frac{1}{jz^j},\;\bar t_{j,b'}\Big\}_{\substack{j=1,2,\dots,\\a'=1,\dots,p_1,\\b'=1,\dots,p_2}},
\end{align}
\item Dual   time shifts are
\begin{align}
 t\rightarrow t\pm\overline{[z^{-1}]}_b:=\Big\{t_{j,a'} ,\; \bar t_{j,b'}
 \pm\delta_{b',b}\frac{1}{jz^j}\Big\}_{\substack{j=1,2,\dots,\\a'=1,\dots,p_1,\\b'=1,\dots,p_2}}.
\end{align}
\end{enumerate}
 \end{definition}
 \begin{pro}
   The Miwa transformations produce the following effect on the weights
   \begin{align}
     \label{vertex weight}
  w_{1,a'}(x,t\mp [z^{-1}]_a,s)&=  \Big(1-\frac{x}{z}\Big)^{\pm\delta_{a,a'}}w_{1,a'}(x,t,s),&
    w_{2,b'}(x,t\mp [z^{-1}]_a,s)&=  w_{2,b'}(x,t,s),\\
    \label{vertex dualweight}
     w_{1,a'}(x,t\pm\overline{[z^{-1}]}_b,s)&=w_{1,a'}(x,t,s),&
  w_{2,b'}(x,t\pm\overline{[z^{-1}]}_b,s)&= \Big(1-\frac{x}{z}\Big)^{\pm\delta_{b,b'}}w_{2,b'}(x,t,s).
   \end{align}
 \end{pro}
 \begin{proof}
When we consider what happens to the evolutionary factors under these shifts we find
\begin{align}
  \exp\Big(\sum_{j}t_{j,a'}x^j\Big)\rightarrow
  \exp\Big(\sum_{j}\big(t_{j,a'}\mp\delta_{a',a}\frac{x^j}{jz^j}\Big)\Big)=
  \Big(1-\frac{x}{z}\Big)^{\mp\delta_{a',a}}\exp\Big(\sum_{j}t_{j,a'}x^j\Big),
\end{align}
and therefore the weights transform according to
\begin{align}\label{vertex weight evol}
  w_{1,a'}(x,t,s)\rightarrow   \Big(1-\frac{x}{z}\Big)^{\pm\delta_{a,a'}}w_{1,a'}(x,t,s),
\end{align}
which is like the Darboux transformations considered previously.
For the dual Miwa shifts  we consider what happens to the evolutionary factors under these shifts
\begin{align}
  \exp\Big(-\sum_{j,b'}\bar t_{j,b'}x^j\Big)\rightarrow   \exp\Big(-\sum_{j,b'}\Big(\bar
  t_{j,b'}\pm\delta_{b',b}\frac{x^j}{jz^j}\Big)\Big)=
  \Big(1-\frac{x}{z}\Big)^{\pm 1}\exp\Big(-\sum_{j,b'}\bar t_{j,b'}x^j\Big),
\end{align}
and the transformation for the weights is

\begin{align}\label{vertex dual weight evol}
  w_{2,b'}(x,t,s)\rightarrow   \Big(1-\frac{x}{z}\Big)^{\pm\delta_{b,b'}}w_{2,a'}(x,t,s).
\end{align}
 \end{proof}
 Thus, a comparison of \eqref{d.evolved}, \eqref{vertex weight} and \eqref{vertex dualweight} leads to
  \begin{pro}
Miwa transformations and discrete flows can be identified as follows
\begin{align}
\label{miwa-darboux}\begin{aligned}
 c_a w_{1,a}(x,t,s_a)&=
\begin{cases}
w_{1,a}(x,t-\sum_{n=1}^{s_a}\big[\lambda_a(n)^{-1}\big]_a), & s_a>0,\\
w_{1,a}(x,t),& s_a=0,\\
w_{1,a}(x,t+\sum_{n=1}^{|s_a|}\big[\lambda_a(-n)^{-1}\big]_a), & s_a<0,
\end{cases}&
c_a:=\begin{cases}
  \prod_{n=1}^{s_a}(-\lambda_a(n))^{-1},&s_a>0\\
  1,&s_a=0\\
   \prod_{n=1}^{|s_a|}(-\lambda_a(-n)),&s_a<0,
\end{cases}\\
 \bar c_b w_{2,b}(x,t,\bar s_b)&=
\begin{cases}
w_{2,b}(x,t+\sum_{n=1}^{\bar s_b}\overline{\big[\bar\lambda_b(n)^{-1}\big]}_b,x), & \bar s_b>0,\\
w_{2,b}(x,t),& \bar s_b=0,\\
w_{2,b}(x,t-\sum_{n=1}^{|\bar s_b|}\overline{\big[\bar \lambda_b(-n)^{-1}\big]}_b), & \bar s_b<0,
\end{cases}&
\bar c_b:=\begin{cases}
  \prod_{n=1}^{\bar s_b}(-\bar \lambda_b(n))^{-1},&\bar s_b>0\\
  1, & \bar s_b=0\\
   \prod_{n=1}^{|\bar s_b|}(-\bar\lambda_b(-n)),&\bar s_b<0.
\end{cases}
\end{aligned}
\end{align}
  \end{pro}
  As a conclusion, the discrete flows and Miwa shifts in the continuous flows are the very same thing, and therefore we could  work with continuous flows and Miwa transformations or with continuous/discrete flows. This discussion justifies the
  Miwa part in the name we gave to these discrete flows.

\subsubsection{Bounded from below measures}
Of course, in order to preserve the link with multiple orthogonal polynomials, these discrete flows must  preserve the
reality, regularity and sign constance  of the
weights, which generically is not the case.  When the support of the weights is bounded from below, i.e. there are
finite real numbers  $K_a$ and $K_b$,
such that $\text{supp} (w_{1,a}\d \mu)\subset [K_a,\infty)$ and $\text{supp} (w_{2,b}\d\mu)\subset [\bar K_b,\infty)$,
a possible solution is to place all the new
zeroes and poles in the real line but outside the corresponding support, $\lambda_a(n)<\text{inf}(\text{supp}
(w_{1,a}\d\mu))$ and
$\bar\lambda_b(n)<\text{inf}(\text{supp} (w_{2,b}\d\mu))$. A different approach, which will be considered in Appendix \ref{III}, is
to arrange the zeroes in complex conjugate
pairs.

To analyze the consequence of the discrete flows  on the integrable hierarchy we introduce two sets of shifts
operators:
\begin{definition}
  \begin{enumerate}
    \item Let us consider the sets of shift operators $\{T_a\}_{a=1}^{p_1}$ and $\{\bar T_b\}_{b=1}^{p_2}$, where
        $T_a$ stands for the shift $s_a\mapsto s_a+1$ and $\bar T_b$ stands for $s_b\mapsto \bar s_b+1$. The rest of the variables $\{s_{a'},\bar s_{b'}\}$ will
        remain constant.
\item  We introduce
\begin{align}
\begin{aligned}
q_a&:=  \I-\Pi_{1,a}(\I+\lambda_a(s_a+1))+\Lambda_{1,a},\\
\bar q_b&:=   \I  -\Pi_{2,b}(\I+\bar \lambda_b(\bar s_b+1))+\Lambda_{2,b}^\top.
\end{aligned}
\end{align}

\item We also define the operators
\begin{align}\label{delta}
  \delta_a&:=Sq_aS^{-1}=\I-C_{a}(\I+\lambda_a(s_a+1))+L_a,&
    \bar\delta_b&:=\bar S \bar q_b\bar S^{-1}=\I-\bar C_{b}(\I+\bar\lambda_b(\bar s_b+1))+\bar L_b,\\
    C_a&:=S\Pi_{1,a}S^{-1},& \bar C_b&:=\bar S\Pi_{2,b}\bar S^{-1}.\notag
\end{align}
Here the matrices $\delta_a$ and $\bar\delta_b$  are called lattice resolvents.

\item  Finally the semi-infinite wave matrices
\begin{align}\label{discrete-wave}
  W&:=SD_0,& \bar W&:=\bar S\bar D_0.
\end{align}
  \end{enumerate}
\end{definition}
Observe that
\begin{align}
\begin{aligned}
(T_a D_0)D_0^{-1}&=q_a,&\bar D_0^{-1} (T_a\bar D_0)&=\I,\\
  \bar D_0(\bar T_b\bar D_0^{-1})&=\bar q_b,& (\bar T_b  D_0) D_0^{-1}&=\I.&
  \end{aligned}
\end{align}
When we   assume that the semi-infinite matrices $\delta_a$ and $\bar\delta_b$ are $LU$ factorizable as in
\eqref{facto}, i.e. all their principal minors do not
vanish, we can write
\begin{align}\label{fac.delta}
  \delta_a&=\delta_{a,-}^{-1}\delta_{a,+},& \bar \delta_b&=\bar \delta_{b,-}^{-1}\bar\delta_{b,+},
\end{align}
where $\delta_{a,-}$ and $\bar\delta_{b,-}$ are  lower matrices as is  $S$  in \eqref{facto}, and $\delta_{a,+}$ and
$\bar\delta_{b,+}$ are upper matrices as
$\bar S$ in \eqref{facto}. We now show  that when the deformed moment matrix $g(s)$ is factorizable, and therefore the
multiple orthogonality  makes sense, the following  holds
\begin{pro}
  If the deformed moment matrix $g(s)$ is factorizable for all values of $s$ then  so is $\delta_a$ and $\bar
  \delta_b$ with
  \begin{align}
  \begin{aligned}
  \delta_{a,+}&= (T_a\bar S)\bar S^{-1}, &   \delta_{a,-}&=(T_a S)S^{-1},\\
   \bar\delta_{b,+}&=(\bar T_b\bar
   S)\bar S^{-1}, &  \bar\delta_{b,+}&=(\bar T_b S)S^{-1}.
   \end{aligned}
  \end{align}
\end{pro}
 \begin{proof}
   When we apply the discrete shifts to the Gauss--Borel factorization problem $g(s)=S^{-1}(s)\bar S(s)$ we get
\begin{align*}
  T_a(S^{-1})T_a(\bar S)&=T_a g(s)=(T_aD_0)D_0^{-1} g(s)=q_ag(s) &&\Rightarrow &((T_a S)S^{-1})^{-1}(T_a\bar S)\bar
  S^{-1}&=\delta_a ,\\
   \bar T_b(S^{-1})\bar T_b(\bar S)&=\bar T_b g(s)=g(s)\bar D_0(T_b\bar D_0^{-1})=g(s)\bar q_b &&\Rightarrow & ((\bar
   T_b S)S^{-1})^{-1}((\bar T_b\bar
   S)\bar S^{-1}&=\bar \delta_b,
\end{align*}
and the desired result follows.
 \end{proof}
Therefore, we can consider the following
\begin{definition}
The semi-infinite matrices $\omega_a$ and $\bar\omega_b$ are given by
       \begin{align}\label{omega}
         \omega_a&:=\delta_{a,-} \delta_a=\delta_{a,+},&
\bar \omega_b&:=\bar \delta_{b,-}=\bar \delta_{b,+} \bar \delta_b^{-1},
       \end{align}
     \end{definition}
and show that
\begin{pro}
 \begin{enumerate}
   \item    The following auxiliary linear systems
           \begin{align}\label{d.wave}
           \begin{aligned}
               T_aW &=\omega_a W,&T_a\bar W &=\omega_a \bar  W,\\
             \bar  T_bW &=\bar \omega_b W,&\bar T_b\bar W &=\bar \omega_b \bar  W,
                        \end{aligned}
                  \end{align}
are satisfied.
\item The Lax matrices fulfill the following relations
           \begin{align}\label{d.lax}
           \begin{aligned}
                T_aL_{a'}&=\omega_aL_{a'}\omega_a^{-1},  &  T_a\bar L_{b}&=\omega_a\bar L_{b}\omega_a^{-1},\\
             \bar T_bL_{a}&=\bar\omega_bL_{a}\bar\omega_b^{-1},&  \bar T_b\bar L_{b'}&=\bar\omega_b\bar
             L_{b'}\bar\omega_b^{-1}.        \end{aligned}
           \end{align}
\item The following discrete Zakharov--Shabat compatibility conditions hold
    \begin{align}\label{d.ZS2}
            ( T_a\omega_{a'})\omega_{a}&=( T_{a'}\omega_{a})\omega_{a'},&
            ( T_a\bar \omega_{b})\omega_{a}&=( \bar T_{b}\omega_{a})\bar \omega_{b},&
               ( \bar T_b\bar \omega_{b'})\bar\omega_{b}&=( \bar T_{b'}\bar\omega_{b})\bar \omega_{b'}.
                       \end{align}
                       \item When the discrete and continous flows are considered simultaneously, the
                           following equations
                             \begin{align}\label{d.ZS1}\begin{aligned}
            T_{a'}B_{j,a}&=(\partial_{j,a}\omega_{a'})\omega_{a'}^{-1}+\omega_{a'}B_{j,a}\omega_{a'}^{-1},&
          \bar
          T_{b}B_{j,a}=(\partial_{j,a}\bar\omega_{b})\bar\omega_{b}^{-1}+\bar\omega_{b}B_{j,a}\bar\omega_{b}^{-1},\\
           T_{a}\bar B_{j,b}&=(\bar\partial_{j,b}\omega_{a})\omega_{a}^{-1}+\omega_{a}\bar
           B_{j,b}\omega_{a}^{-1},&
          \bar  T_{b'}\bar B_{j,b}=(\bar\partial_{j,b}\bar\omega_{b'})\bar\omega_{b'}^{-1}+\bar\omega_{b'}\bar
          B_{j,b}\bar\omega_{b'}^{-1},
           \end{aligned}
                     \end{align}
                     are obtained.
 \end{enumerate}
\end{pro}
\begin{proof}
We compute
      \begin{align*}
          T_a W&=(T_a S)(T_a D_0)=(T_a S) S^{-1} S q_a S^{-1} S D_0=\delta_{a,-} \delta_a W=\delta_{a,+} W,\\
          T_a \bar W&=(T_a\bar S)\bar D_0=(T_a\bar S)\bar S^{-1} \bar S \bar D_0=\delta_{a,+}\bar W,\\
        \bar  T_b  W&=(  \bar  T_b S) D_0=( \bar  T_b S) S^{-1}  S  D_0=\bar \delta_{b,-}  W,\\
        \bar  T_b \bar W&=(\bar T_b \bar S)(T_b \bar D_0)=(T_b \bar S) \bar S^{-1} \bar S \bar q_b^{-1} \bar S^{-1}
        \bar S \bar D_0=\bar \delta_{b,+} \bar
        \delta_b^{-1} \bar W=\bar \delta_{b,-}\bar W,
                \end{align*}
from where we deduce  \eqref{d.wave}, which in turn imply \eqref{d.lax} and \eqref{d.ZS2}.

\end{proof}
The simultaneous consideration of continuous and discrete flows leads to the replacement  $W_0\to W_0D_0$ and $\bar
   W_0\to \bar W_0\bar D_0$, and the
corresponding modification of the weight's flows  is achieved by the multiplication of the continuous and discrete
evolutionary factors, in this context we
also have \eqref{d.ZS1}. These discrete flows could be understood as a sequence of Darboux transformations of $LU$ and $UL$ types in the
terminology of \cite{adler-van moerbeke 2}, which motivates the Darboux part in name we give to these discrete flows.
In fact,  we have that the lattice resolvents satisfy
           \begin{align*}
             \delta_a&=\delta_{a,-}^{-1}\delta_{a,+}&&\Rightarrow&
             T_a\delta_a&=\omega_a\delta_a\omega_a^{-1}=\delta_{a,+}\delta_{a,-}^{-1}
             \delta_{a,+}\delta_{a,+}^{-1}=\delta_{a,+}\delta_{a,-}^{-1},\\
              \bar\delta_b&=\bar\delta_{b,-}^{-1}\bar\delta_{b,+}&&\Rightarrow&
             \bar T_b\bar\delta_b&=\bar\omega_b\bar\delta_b\bar\omega_b^{-1}=\bar\delta_{b,-}\bar\delta_{b,-}^{-1}
            \bar \delta_{b,+}\bar \delta_{b,-}^{-1}=\bar\delta_{b,+}\bar\delta_{b,-}^{-1},
                   \end{align*}
which amounts to the typical permutation of the $LU$ factorization to the $UL$ factorization. When there is only one
component we have $\delta=L+\lambda$ and
$\bar \delta=\bar\lambda+\bar L$ and the shift corresponds to the classical $LU$ or $UL$ Darboux transformations.

       If  $A^{(k)}_a(x,s),\bar A_b^{(k)}(x,s)$ are the multiple orthogonal polynomials and  dual polynomials in the
       $x$ variable  corresponding to the  discrete evolution of the  weights \eqref{d.evolved} respectively we have
       the discrete version of Proposition \ref{conti}
       \begin{pro}
        The wave and adjoint wave functions \eqref{discrete-wave} are
 \begin{align}
    \Psi_a^{(k)}(z,s)&=A^{(k)}_a(z,s)\Ds_a(z,s_a)&
    ( \bar \Psi_b^*)^{(k)}(z,s)&=\bar A_b^{(k)}(z,s)\bar \Ds_b(z,\bar s_b)^{-1},
\end{align}
and the  the linear forms
\begin{align}\label{linear.form.baker.2}
  Q^{(k)}(x,s)&=\sum_{a=1}^{p_1}A_a^{(k)}(x,s)w_{1,a}(x,s), &   \bar Q^{(k)}(x,s)&=\sum_{b=1}^{p_2}(\bar
  A_b)^{(k)}(x,s)w_{2,b}(x,s),
\end{align}
associated with the weights  $w_{1,a}(x,s), w_{2,b}(x,s)$, can be expressed as
\begin{align}\label{linear.form.baker.1}
  Q^{(k)}(x,s)&:=\sum_{a=1}^{p_1}\Psi_a^{(k)}(x,s)w_{1,a}(x), &   \bar
  Q^{(k)}(x,s)&:=\sum_{b=1}^{p_2}(\bar\Psi_b^*)^{(k)}(x,s)w_{2,b}(x),
\end{align}
in terms of which we have the equations
   \begin{align}\label{cauchy.baker}
   \bar  \Psi_b^{(k)}(z,s)&=\int \frac{Q^{(k)}(x,s)}{z-x}w_{2,b}(x)\d \mu(x),&
    ( \Psi_a^*)^{(k)}(z,s)&=\int \frac{\bar Q^{(k)}(x,s)}{z-x}w_{1,a}(x)\d \mu(x).
   \end{align}
       \end{pro}
       Here the Cauchy transforms must be interpreted in exactly the same terms as in Proposition \ref{pro:cauchy  tr}.
    Observe that \eqref{cauchy.baker} do not correspond to the functions of the second kind
   \begin{align}
   \bar  C_b^{(k)}(z,s)&:=\int \frac{Q^{(k)}(x,s)}{z-x}w_{2,b}(x,s)\d \mu(x),&
    C_a^{(k)}(z,s)&:=\int \frac{\bar Q^{(k)}(x,s)}{z-x}w_{1,a}(x,s)\d \mu(x).
   \end{align}
           Notice also that from \eqref{delta}-\eqref{fac.delta}, relations that hold true for any $g$ and not only
           for the moment matrix, we get
           \begin{lemma}\label{lemon}
           We have that
            \begin{align}\label{form.omega}\begin{aligned}
                     \omega_a&=\omega_{a,0}\Lambda^{|\vec n_1|-n_{1,a}+1}+\omega_{a,1}\Lambda^{|\vec
                     n_1|-n_{1,a}}+\dots+\omega_{a,|\vec
                     n_1|-n_{1,a}+1},\\
              \bar \omega_b&=\bar\omega_{b,0}(\Lambda^\top)^{|\vec
              n_2|-n_{2,b}+1}+\bar\omega_{b,1}(\Lambda^\top)^{|\vec
              n_2|-n_{2,b}}+\dots+\bar\omega_{b,|\vec n_2|-n_{2,b}+1},\\
              \omega_a^\top&=\rho_{a,0}(\Lambda^\top)^{|\vec n_1|-n_{1,a}+1}+\rho_{a,1}(\Lambda^\top)^{|\vec
              n_1|-n_{1,a}}+\dots+\rho_{a,|\vec
             n_1|-n_{1,a}+1},\\
              \bar \omega_b^\top&=\bar\rho_{b,0}\Lambda^{|\vec n_2|-n_{2,b}+1}+\bar\rho_{b,1}\Lambda^{|\vec
              n_2|-n_{2,b}}+\dots+\rho_{b,|\vec
              n_2|-n_{2,b}+1},
           \end{aligned}
               \end{align}
           for some diagonal semi-infinite matrices \begin{align}
           \begin{aligned}
            \omega_{a,j}&=\diag(\omega_{a,j}(0),\omega_{a,j}(1),\dots),\\
           \bar\omega_{b,j}&=\diag(\bar\omega_{b,j}(0),\bar\omega_{b,j}(1),\dots),\\
           \rho_{a,j}&=\diag(\rho_{a,j}(0),\rho_{a,j}(1),\dots),\\
           \bar\rho_{b,j}&=\diag(\bar\rho_{b,j}(0),\bar\rho_{b,j}(1),\dots),
           \end{aligned}
                     \end{align}
           with
           \begin{align}
           \begin{aligned}
          \rho_{a,j}(k)&:=\omega_{a,j}(k-|\vec n_1|+n_{1,a}-1+j),\\
            \bar\rho_{b,j}(k)&:=\bar\omega_{b,j}(k+|\vec n_2|-n_{2,b}+1-j).
            \end{aligned}
           \end{align}
           \end{lemma}
       that with
       \begin{definition}
        We define
\begin{align}\label{gamma}\begin{aligned}
  \gamma_{a,a'}(s,x)&:=  (1-\delta_{a,a'}(1+\lambda_a(s_a+1)-x)),\\
\gamma_{b,b'}(s,x)&:=  (1-\delta_{b,b'}(1+\bar\lambda_b(\bar s_b+1)-x)),
\end{aligned}
\end{align}
       \end{definition}
leads to
\begin{pro}
The following equations
  \begin{align}\label{multi.variante}
\begin{aligned}
  (T_{a'} A_a^{(k)})\gamma_{a,a'}&=\omega_{a',0}(k)A_a^{(k+|\vec n_1|-n_{1,a'}+1)}+\dots+\omega_{a',|\vec
  n_1|-n_{1,a'}+1}(k) A_a^{(k)},\\
  \bar T_{b'}A_a^{(k)}&=\bar\omega_{b',0}(k)A_a^{(k-|\vec n_2|+n_{2,b'}-1)}+\dots+\bar\omega_{b',|\vec n_2|-n_{2,b'}+1}
  (k) A_a^{(k)}.
  \end{aligned}\\
  \label{dual.multi.variante}
\begin{aligned}
\rho_{a',0}(k) (T_{a'} \bar A_b^{(k-|\vec n_1|+n_{1,a'}-1)})+\dots+ \rho_{a',|\vec n_1|-n_{1,a'}+1}(k) (T_{a'} \bar
A_b^{(k)})&= \bar A_b^{(k)},\\
\Big( \bar\rho_{b',0} (k)(\bar T_{b'} \bar A_b^{(k+|\vec n_2|-n_{2,b'}+1)})+
 \dots+ \bar\rho_{b',|\vec n_2|-n_{2,b'}+1} (k) (\bar T_{b'} \bar A_b^{(k)})\Big)\bar\gamma_{b,b'}&= \bar A_b^{(k)},
\end{aligned}
\end{align}
are fulfilled.
\end{pro}
\begin{proof}
  For \eqref{multi.variante} recall the discrete auxiliary systems for $W$, while for \eqref{dual.multi.variante} just
  consider that \begin{align*}
  \omega_a^\top T_a((\bar W^{-1})^\top)&= (\bar W^{-1})^\top,&
  \bar  \omega_b^\top \bar T_b((\bar W^{-1})^\top)&= (\bar W^{-1})^\top.
\end{align*}
\end{proof}
 Notice that relations \eqref{multi.variante} and \eqref{dual.multi.variante} are among multiple orthogonal
 polynomials in the same ladder but with different weights, they link the polynomials for the weights $w_{1,a},
 w_{2,b}$ with those with $T_{a'}w_{1,a},T_{a'} w_{2,b}$ or  $\bar T_{b'}w_{1,a},\bar T_{b'} w_{2,b}$.

  \subsection{Symmetries, recursion relations and string equations }

We now return to the discussion of the symmetry of the moment matrix that we started in \S \ref{sec:recursion} but
with  evolved weights and the use of Lax matrices. The first observation is the following
\begin{pro}
The $j$-th power of the evolved Jacobi type matrix introduced in \S \ref{sec:recursion} is related with Lax matrices
through what we call a string equation:  \begin{align}\label{bigradedj}
J^{j}&=\sum_{a=1}^{p_1}L_a^j=\sum_{b=1}^{p_2}\bar L_b^j,& j=1,2,\dots,
  \end{align}
and the multiple orthogonal polynomials are  eigen-vectors:
     \begin{align}\label{bigradedjp}
J^{j}\A_{a'}&=x^j\A_{a'},&
(J^{j})^\top\bar\A_{b'}&=x^j\bar \A_{b'},
  \end{align}
  for $a'=1,\dots,p_1$ and $b'=1,\dots,p_2$.
\end{pro}
\begin{proof}
Using \eqref{sym} it can be proven by induction on $j$ that for any $j\geq 1$ the following equation holds
 \begin{align}\label{symj}
 \Lambda_{1,a}^j g  \Pi_{2,b}=
   \Pi_{1,a}g(\Lambda_{2,b}^\top)^j,
  \end{align}
so that
  \begin{align}\label{stringj}
    L_{a}^j\bar C_{b}= C_{a}\bar L_{b}^j.
  \end{align}
Summing over $a,b$ we deduce \eqref{bigradedj}. Moreover \eqref{bigradedjp} is obtained as follows
    \begin{align}
J^{j}\A_{a'}&=S\sum_{a=1}^{p_1}\Lambda_{1,a}^jS^{-1}S\chi_{1,a'}=x^j\A_{a'},\\
(J^j)^\top\bar\A_{a'}&=(\bar S^{-1})^\top\sum_{b=1}^{p_2}\Lambda_{2,a}^j\bar S^\top  (\bar
S^{-1})^\top\chi_{2,b'}=x^j\bar \A_{b'}.
  \end{align}
\end{proof}

We are ready to show that the symmetry  \eqref{sym} induces a  corresponding invariance on Lax matrices and multiple
orthogonal polynomials
\begin{pro}\label{proposition: symmetry}
  The following relations hold for $j = 1,2,\dots$
  \begin{align}
\big(\sum_{a=1}^{p_1}\frac{\partial}{\partial t_{j,a}}+\sum_{b=1}^{p_2}\frac{\partial}{\partial \bar t_{j,
b}}\big)L_{a'}&=0,&
\big(\sum_{a=1}^{p_1}\frac{\partial}{\partial t_{j,a}}+\sum_{b=1}^{p_2}\frac{\partial}{\partial \bar t_{j,
b}}\big)\bar
L_{b'}&=0,\\
\big(\sum_{a=1}^{p_1}\frac{\partial}{\partial t_{j,a}}+\sum_{b=1}^{p_2}\frac{\partial}{\partial \bar t_{j,
b}}\big)\A_{a'}&=0,&
\big(\sum_{a=1}^{p_1}\frac{\partial}{\partial t_{j,a}}+\sum_{b=1}^{p_2}\frac{\partial}{\partial \bar t_{j,
b}}\big)\bar\A_{b'}&=0.
\end{align}
\end{pro}
\begin{proof}
 See Appendix \ref{II}.
\end{proof}

\subsection{Bilinear equations and $\tau$-functions}
The proof of the bilinear identity needs three lemmas. For the first one, let  $W_{\vec n_1,\vec n_2},\bar W_{\vec
n_1,\vec n_2}$ be  the wave matrices associated with the moment matrix $g_{\vec n_1,\vec n_2}$; so that, $W_{\vec
n_1,\vec n_2}g_{\vec n_1,\vec n_2}=\bar W_{\vec n_1,\vec n_2}$. Then, we have
\begin{lemma}
  \label{lemma: matrix bilinear}
  The wave matrices associated with different compositions and times satisfy
  \begin{align}\label{bilinear-wave}
 W_{\vec n_1,\vec n_2}(t,s)\pi_{\vec n_1',\vec n_1}^\top W_{\vec n_1',\vec n_2'}(t',s')^{-1}=\bar W_{\vec n_1,\vec
 n_2}(t,s)\pi_{\vec n_2',\vec n_2}^\top\bar W_{\vec n_1',\vec n_2'}(t',s')^{-1},
\end{align}
  \end{lemma}
  \begin{proof}
    We consider simultaneously the following equations
    \begin{align*}
      W_{\vec n_1,\vec n_2}(t,s)g&=\bar W_{\vec n_1,\vec n_2}(t,s),\\
      W_{\vec n_1',\vec n_2'}(t',s')\pi_{\vec n_1',\vec n_1} g\pi_{\vec n_2',\vec n_2}^\top&=\bar W_{\vec n_1',\vec
      n_2'}(t',s'),
    \end{align*}
   where $g=g_{\vec n_1,\vec n_2}$, and we get
    \begin{align*}
      W_{\vec n_1,\vec n_2}(t,s)^{-1}\bar W_{\vec n_1,\vec n_2}(t,s)= \pi_{\vec n_1',\vec n_1}^\top W_{\vec n_1',\vec
      n_2'}(t',s')^{-1} \bar W_{\vec n_1',\vec n_2'}(t',s')\pi_{\vec n_2',\vec n_2}=g,
    \end{align*}
    and the result becomes evident.
  \end{proof}
For the second one, let  $(\cdot)_{-1}$ denote the   coefficient in $z^{-1}$ in the Laurent expansi\'{o}n around
$z=\infty$ (place where the Cauchy transforms make sense).
\begin{lemma}\label{chi-id}
  For the vectors $\chi_a$ the following formulae hold
  \begin{align*}
      \big(\sum_{a=1}^{p}\chi_a(\chi_a^*)^\top\big)_{-1}=\big(\sum_{a=1}^{p}\chi^*_a\chi_a^\top\big)_{-1}=\I,
  \end{align*}
\end{lemma}
and therefore
\begin{lemma}
For any couple of semi-infinite matrices $U$ and $V$ we have
\begin{align}\label{residue1}
UV&=\Big(  \sum_{a=1}^{p_1}(U\chi_{1,a}) \big(V^\top\chi_{1,a}^*\big)^\top\Big)_{-1}\\\label{residue2}
&=\Big(  \sum_{b=1}^{p_2}(U\chi_{2,b}^*)\big(V^\top\chi_{2,b}\big)^\top\Big)_{-1},
\end{align}
\end{lemma}
\begin{proof}
  It follows easily from Lemma \ref{chi-id}:
  \begin{align*}
    \Big(  \sum_{a=1}^{p_1}(U\chi_{1,a}) \big(V^\top\chi_{1,a}^*\big)^\top\Big)_{-1}&=
      U    \Big(\sum_{a=1}^{p_1}\chi_{1,a}(\chi_{1,a}^*)^\top \Big)_{-1}V=UV,\\
      \Big(  \sum_{b=1}^{p_2}(U\chi_{2,b}^*)\big(V^\top\chi_{2,b}\big)^\top\Big)_{-1}&=
      U \Big(\sum_{b=1}^{p_2}\chi_{2,b}^*\chi_{2,b}^\top\Big)_{-1}V=UV.
  \end{align*}
\end{proof}
We have the following
\begin{theorem}\label{theorem: bilinear}
\begin{enumerate}
  \item The wave functions and their companions satisfy
\begin{align*}
\sum_{a=1}^{p_1}\oint_{\infty}\Psi^{(k)}_{\vec n_1,\vec n_2,a}(z,t,s)(\Psi^*_{\vec n_1',\vec
n_2',a})^{(l)}(z,t',s')\d z=
\sum_{b=1}^{p_2}\oint_{\infty}\bar\Psi^{(k)}_{\vec n_1,\vec n_2,b}(z,t,s)(\bar\Psi^*_{\vec n_1',\vec
n_2',b})^{(l)}(z,t',s')\d z.
\end{align*}
\item Multiple orthogonal polynomials, their duals and the corresponding second kind functions  are linked by
 \begin{align}\label{bilinear2}
\sum_{a=1}^{p_1}\oint_{\infty}
A^{(k)}_{\vec n_1,\vec n_2,a}(z,t,s)\bar C_{\vec n_1',\vec n_2',a}^{(l)}(z,t',s')E_a(z)\d z=
\sum_{b=1}^{p_2}\oint_{\infty}C^{(k)}_{\vec n_1,\vec n_2,b}(z,t,s)\bar A_{\vec n_1',\vec
n_2',b}^{(l)}(z,t',s')\bar E_b(z)\d z,
\end{align}
where
\begin{align*}
  E_a&:=(\Es_a\Ds_a)(z,t,s)((\Es_a\Ds_a)(z,t',s'))^{-1},&
   \bar E_b&:=(\bar \Es_b\bar\Ds_b)(z,t,s)((\bar\Es_b\bar\Ds_b)(z,t',s'))^{-1}.
\end{align*}
\end{enumerate}
\end{theorem}
\begin{proof}
  \begin{enumerate}
    \item   If we set in  \eqref{residue1} $U=     W_{\vec n_1,\vec n_2}(t,s)$ and $V=\pi_{\vec n_1',\vec
        n_1}^\top W_{\vec n_1',\vec n_2'}(t',s')^{-1}$
    and in  \eqref{residue2} we put $U=\bar W_{\vec n_1,\vec n_2}(t,s)$ and $V= \pi_{\vec n_2',\vec n_2}^\top\bar
    W_{\vec n_1',\vec n_2'}(t',s')^{-1}$
    attending to
  \eqref{bilinear-wave}, recalling that $\Psi_{\vec n_1,\vec n_2,a}=W_{\vec n_1,\vec n_2}\chi_{\vec n_1,a}$,
  $\bar \Psi_{\vec n_1,\vec n_2,b}=\bar W_{\vec n_1,\vec n_2}\chi^*_{\vec n_2,b}$ and observing that
  $\Psi_{\vec n_1',\vec n_2',a}^*=(W_{\vec n_1',\vec n_2'}^{-1})^\top\pi_{\vec n_1',\vec n_1}\chi^*_{\vec n_1,a}$
  and
  $\bar \Psi_{\vec n_1',\vec n_2',b}^*=(\bar W_{\vec n_1',\vec n_2'}^{-1})^\top\pi_{\vec n_2',\vec n_2}\chi_{\vec
  n_2,b}$  we get  the stated  bilinear equation for the wave functions.\footnote{The reader familiarized with
  Toda bilinear equations should
  notice  that in the right hand term we are working at $z=\infty$ instead of, as customary, at $z=0$; the reason
  is that for the definition of $\bar\A_b$ we
  have used $\chi_2$ instead of $\chi_2^*$, in order to get polynomials in $z$, while normally one gets
  polynomials in $z^{-1}$. See footnote 2. Moreover,
  due to the redefinition of the wave functions there is no $\frac{\d z}{2\pi\text{i}z}$ factor}
\item
 We can   write
 \begin{multline*}
 W_{\vec n_1,\vec n_2}(t,s)\pi_{\vec n_1',\vec n_1}^\top W_{\vec n_1',\vec n_2'}(t',s')^{-1}=(S_{\vec n_1,\vec
 n_2}(t,s)W_{0,\vec n_1}(t,s)\pi_{\vec n_1',\vec n_1}^\top (W_{0,\vec n_1'}(t',s'))^{-1}\pi_{\vec n_1',\vec
 n_1})\pi_{\vec n_1',\vec n_1}^\top S_{\vec n_1',\vec n_2'}(t',s')^{-1},
 \end{multline*}
  which strongly suggests to consider in  \eqref{residue1}
  \begin{align*}
  U&=S_{\vec n_1,\vec n_2}(t,s)W_{0,\vec n_1}(t,s)\pi_{\vec n_1',\vec n_1}^\top(W_{0,\vec
  n_1'}(t',s'))^{-1}\pi_{\vec n_1',\vec n_1},& V&=\pi_{\vec n_1',\vec n_1}^\top S_{\vec n_1',\vec
  n_2'}(t',s')^{-1}.
  \end{align*}
  Analogously
 \begin{multline*}
\bar W_{\vec n_1,\vec n_2}(t,s)\pi_{\vec n_1',\vec n_1}^\top \bar W_{\vec n_1',\vec n_2'}(t',s')^{-1}=(\bar
S_{\vec n_1,\vec n_2}(t,s)\bar W_{0,\vec n_1}(t,s)\pi_{\vec n_1',\vec n_1}^\top (\bar W_{0,\vec
n_1'}(t',s'))^{-1}\pi_{\vec n_1',\vec n_1})\pi_{\vec n_1',\vec n_1}^\top \bar S_{\vec n_1',\vec
n_2'}(t',s')^{-1},
 \end{multline*}
suggest to set in   \eqref{residue2}
  \begin{align*}
  U&=\bar S_{\vec n_1,\vec n_2}(t,s)\bar W_{0,\vec n_2}(t,s)\pi_{\vec n_2',\vec n_2}^\top(\bar W_{0,\vec
  n_2'}(t',s'))^{-1}\pi_{\vec n_2',\vec n_2},& V&=\pi_{\vec n_2',\vec n_2}^\top \bar S_{\vec n_1',\vec
  n_2'}(t',s')^{-1}.
  \end{align*}
   The application of \eqref{residue1},\eqref{residue2} and \eqref{bilinear-wave} gives the alternative bilinear
   relations \eqref{bilinear2}
where we have used  the evolved Cauchy transforms \eqref{evolved.ct} and  introduce the evolutionary factors
\begin{align*}
  E_a&:=(\Es_a\Ds_a)(z,t,s)((\Es_a\Ds_a)(z,t',s'))^{-1},\\
   \bar E_b&:=(\bar \Es_b\bar\Ds_b)(z,t,s)((\bar\Es_b\bar\Ds_b)(z,t',s'))^{-1}.
\end{align*}
  \end{enumerate}
\end{proof}
The factors involved in this definition were introduced in \eqref{evolved} and \eqref{d.evolved}, so that we assume
the discrete flows within the bounded from
below support scenario, while if we consider the two-step discrete flows the replacement of the $\Ds$-factors by the
$\Ds'$- factors \eqref{d.evolved2} is
required.

It can be shown that for certain weights, for which we can use the Fubini and Cauchy theorems, and when one only
considers a finite number of continuous flows that the r.h.s and l.hs. in this bilinear relations are proportional to
$\int_{\R} Q_{\vec n_1,\vec n_2}^{(k)}(x,t)\bar Q_{\vec n_1',\vec n_2'}^{(l)}(x,t')\d \mu(x)$. This is a direct
consequence of
\begin{pro}\label{proposition: bilinear fubini-cauchy}
We have the following identity
\begin{align}
\int_{\R} Q_{\vec n_1,\vec n_2}(x,t,s)\bar Q_{\vec n_1',\vec n_2'}^\top(x,t',s')\d \mu(x)&= W_{\vec n_1,\vec
n_2}(t,s)\pi_{\vec n_1',\vec n_1}^\top (W_{\vec n_1',\vec n_2'}(t',s'))^{-1}\\&=\bar W_{\vec n_1,\vec
n_2}(t,s)\pi_{\vec n_2',\vec n_2}^\top (\bar W_{\vec n_1',\vec n_2'}(t',s'))^{-1}.
\end{align}
\end{pro}
\begin{proof}
  See Appendix \ref{II}.
\end{proof}



Now,  we will perform a full characterization of the $\tau$-functions associated with the multiple orthogonal polynomials
defined in this paper.
\begin{definition}
Let us define the following matrices
\begin{align}\label{def_g+a}
  g^{[l+1]}_{+a}:&=
    \left(\begin{BMAT}{ccccc}{cccc:c}
   g_{0,0}&g_{0,1}&\cdots&g_{0,l-1}&g_{0,l}\\
    g_{1,0}&g_{1,1}&\cdots& g_{1,l-1}&g_{1,l}\\
    \vdots&\vdots&&\vdots&\vdots \\
         g_{l-1,0}&g_{l-1,1}&\cdots& g_{l-1,l-1}&g_{l-1,l}\\
         g_{l_{+a},0}&g_{l_{+a},1}&\cdots& g_{l_{+a},l-1} &g_{l_{+a},l}
\end{BMAT}\right)&
  \bar g^{[l+1]}_{+b}:&=
   \left(\begin{BMAT}{cccc:c}{ccccc}
  g_{0,0}&g_{0,1}&\cdots& g_{0,l-1} &g_{0,\bar l_{+b}}\\
    g_{1,0}&g_{1,1}&\cdots& g_{1,l-1} &g_{1,\bar l_{+b}}\\
    \vdots&\vdots&&\vdots&\vdots\\
    g_{l-1,0}&g_{l-1,1}&\cdots& g_{l-1,l-1}&g_{l-1,\bar l_{+b}}\\
         g_{l,0}&g_{l,1}&\cdots& g_{l,l-1}&g_{l,\bar l_{+b}}
         \end{BMAT}\right).
\end{align}
\end{definition}
The matrix $g^{[l+1]}_{+a}$ is obtained from  $g^{[l+1]}$ replacing the last row (operation denoted by a dashed line)
by
\begin{align*}
(g_{l_{+a},0},g_{l_{+a},1},\dots,g_{l_{+a},l-1},g_{l_{+a},l}),
\end{align*} and $\bar g^{[l+1]}_{+b}$is obtained from  $g^{[l+1]}$ replacing the last column  by $(  g_{0,\bar
l_{+b}},g_{1,\bar l_{+b}},\dots,g_{l-1,\bar l_{+b}},g_{l,\bar
l_{+b}})^\top$. It is clear that if $a_1(l)=a$ then $g^{[l+1]}_{+a}=g^{[l+1]}$ and if $a_2(l)=b$ then $\bar
g^{[l+1]}_{+b}=g^{[l+1]}$.

The minors of the these matrices \eqref{def_g+a}  will be denoted as $M^{[l+1]}_{i,j}=\bar M^{[l+1]}_{i,j}$ for
$g^{[l+1]}$, $M^{[l+1]}_{+a,i,j}$ for $g^{[l+1]}_{+a}$ and $\bar M^{[l+1]}_{+b,i,j}$ for $\bar g^{[l+1]}_{+b}$.
Now we introduce the following determinants that are cofactors of the previously defined matrices
\begin{definition} \label{def_tau1}
The $\tau$-functions  are defined as follows
\begin{align} \label{def_taus}
\begin{aligned}
\tau_{+a,-a'}^{(l)}&:=(-1)^{l+l_{-a'}}M_{+a,l_{-a'},l}^{[l+1]}, &
\tau_{-b,-a}^{(l)}&:=(-1)^{\bar l_{-b}+l_{-a}}M_{l_{-a},\bar l_{-b}}^{[l+1]},
\end{aligned}\\
 \label{def_dualtaus}
\begin{aligned}
\bar \tau_{+b,-b'}^{(l)}&:=(-1)^{l+\bar l_{-b'}}\bar M_{+b,l,\bar l_{-b'}}^{[l+1]},&
\bar \tau_{-a,-b}^{(l)}&:=(-1)^{l_{-a}+\bar l_{-b}}\bar M_{l_{-a},\bar l_{-b}}^{[l+1]}.
\end{aligned}
\end{align}
Moreover,
\begin{enumerate}
  \item If $a_1(l)=a$ then we denote $\tau_{-a'}^{(l)}:=\tau_{+a,-a'}^{(l)}$ and $\bar \tau_{-b}^{(l)}:=\bar
      \tau_{-a,-b}^{(l)}$.
  \item If $a_2(l)=b$ then we denote $\tau_{-a}^{(l)}:=\tau_{-b,-a}^{(l)}$ and $\bar \tau_{-b'}^{(l)}:=\bar
      \tau_{+b,-b'}^{(l)}$.
  \item We also introduce $\tau^{(l)}=\bar \tau^{(l)} := \det g^{[l]}$ and
\begin{align*}
  \tau^{(l+1)}_{+a}&:=\det   g^{[l+1]}_{+a},& \bar\tau^{(l+1)}_{+b}&:=\det  \bar  g^{[l+1]}_{+b}.
\end{align*}
\end{enumerate}
\end{definition}

If $a_1(l)=a$ then   $\tau^{(l+1)}_{+a}=\tau^{(l+1)}$, and if $a_2(l)=b$ then
$\bar\tau^{(l+1)}_{+b}=\tau^{(l+1)}$.\\

Given a perfect combination $(\mu,\vec w_1,\vec w_2)$ and the corresponding  set of multiple orthogonal polynomials
$\{A_{[\vec \nu_1;\vec \nu_2],a}\}_{a=1}^{p_1}$, with degree vectors such that $|\vec\nu_1|=|\vec \nu_2|+1$,
  there  exists a $(\vec n_1,\vec n_2)$ ladder  and an integer $l$ with $|\vec \nu_1|=l+1$  and $|\vec \nu_2|=l$  such
  that the polynomials  $\{A_a^{(l)}\}_{a=1}^{p_1}$ coincide with $\{A_{[\vec \nu_1;\vec \nu_2],a}\}_{a=1}^{p_1}$.
  The final result does not depend upon the particular $(\vec n_1,\vec n_2)$ ladder we choose to get up to the given
  degrees in the ladder; however, the $\tau$-functions do indeed depend on the ladder chosen through a global sign.
  A simple sign-fixing rule is to choose the ladder $ \vec n_1=\vec \nu_1$ and $\vec n_2=\vec \nu_2+\vec e_{p_2}$.
We define
\begin{align*}
  \tau_{[\vec\nu_1;\vec\nu_2]}&:=\tau^{(l)}_{\vec\nu_1,\vec\nu_2},& l&=|\vec\nu_1|-1=|\vec\nu_2|,
\end{align*}
and we deduce
\begin{pro}\label{tausss}
Given degree vectors $(\vec \nu_1,\vec \nu_2)$ such that $|\vec\nu_1|=|\vec \nu_2|+1$, a composition with  $ \vec
n_1=\vec \nu_1$ and $\vec n_2=\vec \nu_2+\vec e_{p_2}$ and $l=|\vec\nu_1|-1=|\vec\nu_2|$,
 we have the following identities
  \begin{gather*}
\begin{aligned}
\tau_{+a,-a'}^{(l)}&=\varepsilon_{1,1}( a,a')\tau_{[\vec\nu_1-\vec e_{1,a'}+\vec e_{1,a};\vec\nu_2]},&
\bar \tau_{+b,-b'}^{(l)}&=\varepsilon_{2,2}(b,b')\tau_{[\vec\nu_1;\vec\nu_2-\vec e_{2,b'}+\vec e_{2,b}]},
\end{aligned}\\
\begin{aligned}
&\tau_{-b,-a}^{(l)}=\bar \tau_{-a,-b}^{(l)}&=\varepsilon_{2,1}(b, a)\tau_{[\vec\nu_1+\vec e_{1,p_1}-\vec
e_{1,a};\vec\nu_2+\vec e_{2,p_2}-\vec e_{2,b}]},
\end{aligned}
\end{gather*}
where
\begin{align*}
  \varepsilon_{1,1}(a,a')&:=(-1)^{\sum_{j=1}^{a}\nu_{1,j}+\sum_{j=1}^{a'}\nu_{1,j}+\delta_{a,p_1}-1}, & a'& <a \\
  \varepsilon_{1,1}(a,a')&:=(-1)^{\sum_{j=1}^{a}\nu_{1,j}+\sum_{j=1}^{a'}\nu_{1,j}+\delta_{a',p_1}},  & a'& > a \\
  \varepsilon_{2,2}(b,b')&:=(-1)^{\sum_{j=1}^{b}\nu_{2,j}+\sum_{j=1}^{b'}\nu_{2,j}-1}, & b'& <b\\
  \varepsilon_{2,2}(b,b')&:=(-1)^{\sum_{j=1}^{b}\nu_{2,j}+\sum_{j=1}^{b'}\nu_{2,j}}, & b'& >b\\
  \varepsilon_{2,1}(b,a)&:=(-1)^{\sum_{j=1}^{b}\nu_{2,j}+\sum_{j=1}^{a}\nu_{1,j}+\delta_{b,p_2}}, \\
  \varepsilon_{1,1}(a,a)&:=1=\varepsilon_{2,2}(b,b).
\end{align*}
In particular
\begin{align*}
  \tau^{(l)}_{-a}&= \varepsilon_{1,1}(p_1,a)\tau_{[\vec\nu_1+\vec e_{1,p_1}-\vec e_{1,a};\vec\nu_2]},&
 \bar   \tau^{(l)}_{-b}&=\varepsilon_{2,2}(p_2,b)\tau_{[\vec\nu_1;\vec\nu_2+\vec e_{2,p_2}-\vec e_{2,b}]},\\
  \tau^{(l+1)}_{+a}&=\varepsilon_{1,1}(a,p_1)\tau_{[\vec \nu_1+\vec e_{1,a};\vec\nu_2+\vec e_{2,p_2}]},
  & \bar\tau^{(l+1)}_{+b}&=\varepsilon_{2,2}(b,p_2)\tau_{[\vec \nu_1+\vec e_{1,p_1};\vec\nu_2+\vec e_{2,b}]}.
\end{align*}
\end{pro}
We now proceed to give the $\tau$-function representation  of multiple orthogonal polynomials, their duals, second
kind functions and bilinear equations.
The $\tau$-functions allow for compact expressions for the multiple orthogonal polynomials:
\begin{pro}\label{taumopis}
The mixed multiple orthogonal polynomials $A_a^{(l)}$, $A_{+a',a}^{(l)}$ and $A_{-b,a}^{(l)}$ have the following
$\tau$-function representation
\begin{align}\label{taumops}
A_a^{(l)}(z)&=A^{(\rII,a_1(l))}_{[\vec\nu_1(l);\vec
\nu_2(l-1)],a}=z^{\nu_{1,a}(l)-1}\dfrac{\tau^{(l)}_{-a}(t-[z^{-1}]_a)}{\tau^{(l)}(t)}, & l\geq 1, \\
{\label{taumops+}}
A_{+a',a}^{(l)}(z)&=A^{(\rII,a')}_{[\vec\nu_1(l-1)+\vec e_{1,a'};\vec
\nu_2(l-1)],a}=z^{\nu_{1,a}(l-1)+\delta_{a,a'}-1}\dfrac{\tau^{(l)}_{+a',-a}(t-[z^{-1}]_{a})}{\tau^{(l)}(t)}, & l\geq
1,\\
\label{taumops-}
A_{-b,a}^{(l)}(z)&=A^{(\rI,b)}_{[\vec\nu_1(l);\vec \nu_2(l)-\vec e_{2,b}],a}=z^{
\nu_{1,a}(l)-1}\dfrac{\tau^{(l)}_{-b,-a}(t-[z^{-1}]_{a})}{\tau^{(l+1)}(t)}, & l\geq 1.
\end{align}
The dual  polynomials $\bar A_b^{(l)}$, $\bar A_{+b',b}^{(l)}$, and $\bar A_{-a,b}^{(l)}$ have the following
$\tau$-function representation
\begin{align}\label{taudualmops}
\bar A_b^{(l)}(z)&=\bar A^{(\rI,a_2(b))}_{[\vec\nu_2(l);\vec \nu_1(l-1)],b}=z^{\nu_{2,b}(l)-1}\dfrac{\bar
\tau_{-b}^{(l)}(t+\overline{[z^{-1}]}_b)}{ \tau^{(l+1)}(t)}, & l\geq 1,\\
\label{taudualmops+}
\bar A_{+b',b}^{(l)}(z)&=\bar A^{(\rII,b')}_{[\vec\nu_2(l-1)+\vec e_{2,b'};\vec \nu_1(l-1)],b}=z^{ \nu_{2,b} (l-1)+
\delta_{b,b'}-1} \dfrac{\bar \tau_{+b',-b}^{(l)}(t+\overline{[z^{-1}]}_{b})}{ \tau^{(l)}(t)}, & l\geq 1,\\
\label{taudualmops-}
\bar A_{-a,b}^{(l)}(z)&=\bar A^{(\rI,a)}_{[\vec\nu_2(l);\vec \nu_1(l)-\vec e_{1,a}],b}=z^{ \nu_{2,b} (l)-1}\dfrac{\bar
\tau^{(l)}_{-a,-b}(t+\overline{[z^{-1}]}_{b})}{ \tau^{(l+1)}(t)}, & l\geq 1.
\end{align}
\end{pro}
\begin{proof}
See Appendix \ref{II}
\end{proof}

Observe that in the simple ladder defined above $(\vec\nu_1,\vec\nu_2+\vec e_{2,p_2})$ with
$l=|\vec\nu_2|=|\vec\nu_1|-1$ we have
\begin{align*}
  \vec\nu_1(l)&=\vec \nu_1, &\vec\nu_2(l-1)&=\vec\nu_2,\\
   \vec\nu_1(l-1)&=\vec \nu_1-\vec e_{1,p_1}, &\vec\nu_2(l)&=\vec\nu_2+\vec e_{2,p_2}.
\end{align*}
From  Proposition \ref{taumopis} we get
 \begin{align}
\begin{aligned}
A^{(\rII,p_1)}_{[\vec\nu_1;\vec \nu_2],a}(z)&=\varepsilon_{1,1}(p_1,a)z^{\nu_{1,a}-1}\dfrac{\tau_{[\vec\nu_1+\vec
e_{1,p_1}-\vec e_{1,a}a;\vec \nu_2]}(t-[z^{-1}]_a)}{\tau_{[\vec\nu_1;\vec \nu_2]}(t)}, \\
A^{(\rII,a')}_{[\vec\nu_1-\vec e_{1,p_1}+\vec e_{1,a'};\vec \nu_2],a}(z)&=\varepsilon_{1,1}(
a',a)z^{\nu_{1,a}-\delta_{a,p_1}+\delta_{a,a'}-1}
\dfrac{\tau_{[\vec\nu_1-\vec e_{1,a}+\vec e_{1,a'};\vec\nu_2]}(t-[z^{-1}]_{a})}{\tau_{[\vec\nu_1;\vec \nu_2]}(t)},\\
A^{(\rI,b)}_{[\vec\nu_1;\vec \nu_2+\vec e_{2,p_2}-\vec e_{2,b}],a}(z)&=\varepsilon_{2,1}(b, a)z^{
\nu_{1,a}-1}\dfrac{\tau_{[\vec\nu_1+\vec e_{1,p_1}-\vec e_{1,a};\vec\nu_2+\vec e_{2,p_2}-\vec
e_{2,b}]}(t-[z^{-1}]_{a})}{\tau_{[\vec\nu_1+\vec e_{1,p_1};\vec \nu_2+\vec e_{2,p_2}]}(t)},
\end{aligned}
\end{align}
 \begin{align}
 \begin{aligned}
\bar A^{(\rI,p_2)}_{[\vec\nu_2+\vec e_{2,p_2};\vec \nu_1-\vec
e_{p_1}],b}(z)&=\varepsilon_{2,2}(p_2,b)z^{\nu_{2,b}+\delta_{b,p_2}-1}\dfrac{\tau_{[\vec\nu_1;\vec\nu_2+\vec
e_{2,p_2}-\vec e_{2,b}]}(t+\overline{[z^{-1}]}_b)}{\tau_{[\vec\nu_1+\vec e_{1,p_1};\vec \nu_2+\vec e_{2,p_2}]}(t)},\\
\bar A^{(\rII,b')}_{[\vec\nu_2+\vec e_{2,b'};\vec \nu_1-\vec e_{1,p_1}],b}&=\varepsilon_{2,2}(b,b')z^{ \nu_{2,b}+
\delta_{b',b}-1} \dfrac{\tau_{[\vec\nu_1;\vec\nu_2-\vec e_{2,b}+\vec e_{2,b'}]}(t+\overline{[z^{-1}]}_{b})}{
\tau_{[\vec\nu_1;\vec \nu_2]}(t)},\\
\bar A^{(\rI,a)}_{[\vec\nu_2+\vec e_{2,p_2};\vec \nu_1-\vec e_{1,a}],b}&=\varepsilon_{2,1}(b, a)z^{ \nu_{2,b}
+\delta_{b,p_2}-1}\dfrac{\tau_{[\vec\nu_1+\vec e_{1,p_1}-\vec e_{1,a};\vec\nu_2+\vec e_{2,p_2}-\vec
e_{2,b}]}(t+\overline{[z^{-1}]}_{b})}{\tau_{[\vec\nu_1+\vec e_{1,p_1};\vec \nu_2+\vec e_{2,p_2}]}(t)}.
\end{aligned}
\end{align}

We now present the $\tau$-representation of the Cauchy transforms of the linear forms.
\begin{pro}\label{proposition: tau cauchy}
The Cauchy transforms have the following $\tau$-function representation
 \begin{align}\label{tau3}
  \bar C^{(l)}_a=z^{-\nu_{1,a}(l-1)-1}\frac{\tau_{+a}^{(l+1)}(t+[z^{-1}]_a)}{\tau^{(l+1)}(t)},\\
  \label{tau4}
   C^{(l)}_b=z^{-\nu_{2,b}(l-1)-1}\frac{\bar \tau_{+b}^{(l+1)}(t-\overline{[z^{-1}]}_b)}{\tau^{(l)}(t)}.
\end{align}
\end{pro}
\begin{proof}
  See Appendix \ref{II}
\end{proof}

We have the representation
 \begin{align*}
  C^{(\rI,p_2)}_{[\vec\nu_2+\vec e_{2,p_2};\vec \nu_1-\vec
  e_{1,p_1}],a}(z)&=\varepsilon_{1,1}(a,p_1)z^{-\nu_{1,a}-1+\delta_{a,p_1}}\frac{\tau_{[\vec \nu_1+\vec
  e_{1,a};\vec\nu_2+\vec e_{2,p_2}]}(t+[z^{-1}]_a)}{\tau_{[\vec \nu_1+\vec e_{1,p_1};\vec\nu_2+\vec
  e_{2,p_2}]}(t)},\\
   C^{(\rII,p_1)}_{[\vec\nu_1;\vec \nu_2],b}(z)&=\varepsilon_{2,2}(b,p_2)z^{-\nu_{2,b}-1}\frac{\tau_{[\vec \nu_1+\vec
   e_{1,p_1};\vec\nu_2+\vec e_{2,b}]}(t-\overline{[z^{-1}]}_b)}{\tau_{[\vec \nu_1;\vec\nu_2]}(t)}.
\end{align*}


Finally, we consider the $\tau$-function representation of the bilinear equation
\begin{pro}
  The $\tau$ functions fulfill the following bilinear relation
  \begin{multline}\label{bilinear2?}
\sum_{a=1}^{p_1}\oint_{z=\infty}
z^{\nu_{1,a}(k)-\nu_{1,a}'(l-1)-2}\tau^{(k)}_{\vec n_1,\vec n_2,-a}(t-[z^{-1}]_a)
\tau_{\vec n_1',\vec n_2',+a}^{(l+1)}(t'+[z^{-1}]_a)E_a(z)\d z\\=
\sum_{b=1}^{p_2}\oint_{z=\infty}z^{\nu_{2,b}'(l)-\nu_{2,b}(k-1)-2}\bar \tau_{\vec n_1,\vec
n_2,+b}^{(k+1)}(t-\overline{[z^{-1}]}_b)\bar
\tau_{\vec n_1',\vec n_2',-b}^{(l)}(t'+\overline{[z^{-1}]}_b)\bar E_b(z)\d z.
\end{multline}
\end{pro}
\begin{proof}
Just consider \eqref{bilinear2} together with \eqref{taumops}, \eqref{taudualmops}, \eqref{tau3} and \eqref{tau4}.
\end{proof}
This bilinear relation can also be written as follows
  \begin{multline}\label{bilinear2bis}
\sum_{a=1}^{p_1}\varepsilon_{11}(p_1,a)\varepsilon_{11}'(p_1,a)\oint_{z=\infty}
z^{\nu_{1,a}-\nu_{1,a}'-\delta_{a,p_1}-2}
\tau_{[\vec \nu_1+\vec e_{1,p_1}-\vec e_{1,a};\vec \nu_2]}(t-[z^{-1}]_a)
\tau_{\vec \nu_1'+\vec e_{1,a};\vec \nu_2'+\vec e_{2,p_2}}^{(l+1)}(t'+[z^{-1}]_a)E_a(z)\d z\\=
\sum_{b=1}^{p_2}\varepsilon_{22}(p_2,b)\varepsilon_{22}'(p_2,b)\oint_{z=\infty}z^{\nu_{2,b}'+\delta_{b,p_2}-\nu_{2,b}-2}
\tau_{[\vec \nu_1+\vec e_{1,p_1};\vec \nu_2+\vec e_{2,b}]}(t-\overline{[z^{-1}]}_b)
\tau_{\vec \nu_1';\vec \nu_2'+\vec e_{2,p_2}-\vec e_{2,b}]}(t'+\overline{[z^{-1}]}_b)\bar E_b(z)\d z.
\end{multline}
That with the identification $m^*=\vec\nu_1+\vec e_{1,p_1}$, $n^*=\vec\nu_2$, $m=\vec\nu_1'$ and $n=\vec\nu_2'+\vec
e_{2,p_2}$,  up to signs, is the bilinear relation (41) in \cite{adler-vanmoerbeke 5}.

\section*{Acknowledgements}
MM thanks economical support from the Spanish Ministerio de Ciencia e Innovaci\'{o}n, research
project FIS2008-00200 and UF thanks economical support from the SFRH / BPD / 62947 / 2009, Funda\c{c}\~{a}o para a Ci\^{e}ncia e a Tecnologia of Portugal,  PT2009-0031, Acciones Integradas Portugal, Ministerio de Ciencia e Innovaci\'{o}n de Espa\~{n}a, and
MTM2009-12740-C03-01. Ministerio de Ciencia e Innovaci\'{o}n de Espa\~{n}a. MM reckons different and clarifying discussions with Dr. Mattia Caffasso, Prof. Pierre van Moerbeke,  Prof. Luis Mart\'{\i}nez Alonso and Prof. David G\'{o}mez-Ullate. The authors of this paper are in debt with Prof. Guillermo L\'{o}pez Lagomasino who carefully read the manuscript and whose suggestions improved the paper indeed.

\begin{appendices}

\section{Proofs}\label{II}

\paragraph{Proof Proposition \ref{proposition: expressions for linear forms}}
The orthogonality relations can be recast into two alternative forms
\begin{align}\label{S.1}
 \begin{pmatrix}
   S_{l,0}&S_{l,1}&\cdots&S_{l,l-1}
 \end{pmatrix} \begin{pmatrix}
    g_{0,0}&g_{0,1}&\cdots&g_{0,l-1}\\
    g_{1,0}&g_{1,1}&\cdots&g_{1,l-1}\\
    \vdots&\vdots&&\vdots\\
     g_{l-1,0}&g_{l-1,1}&\cdots&g_{l-1,l-1}
  \end{pmatrix}&=- \begin{pmatrix}
   g_{l,0}&g_{l,1}&\cdots&g_{l,l-1}
 \end{pmatrix},& l&\geq 1,\\
 \label{S.2}
 \begin{pmatrix}
 S_{l,0}&S_{l,1}&\cdots&S_{l,l-1}&S_{l,l}
 \end{pmatrix} \begin{pmatrix}
    g_{0,0}&g_{0,1}&\cdots&g_{0,l}\\
    g_{1,0}&g_{1,1}&\cdots&g_{1,l}\\
    \vdots&\vdots&&\vdots\\
     g_{l,0}&g_{l,1}&\cdots&g_{l,l}
  \end{pmatrix}&=\underbrace{\begin{pmatrix}
   0&0&\cdots&0& \bar S_{l,l}
    \end{pmatrix}}_{\text{$l+1$ components}},& l&\geq 0.
\end{align}
From \eqref{linear forms S} we get
\begin{align}
  Q^{(l)}&=\sum_{k=0}^l S_{l,k}\xi_{1}^{(k)}\notag
  \\&=\xi_{1}^{(l)}-\begin{pmatrix}
    g_{l,0}&g_{l,1}&\cdots&g_{l,l-1}
  \end{pmatrix}(g^{[l]})^{-1}\xi_1^{[l]}& \text{use \eqref{S.1}}\label{S.1.1}\\
  &=\bar S_{l,l}\begin{pmatrix}
    0 &0 &\cdots &0& 1
  \end{pmatrix}
  (g^{[l+1]})^{-1}\xi^{[l+1]}_1. &\text{use \eqref{S.2}}\label{S.1.2}
\end{align}
 Cramer's method solves \eqref{S.1} as follows
\begin{align}\label{cofactor}
  S_{l,i}&=\frac{1}{\det
  g^{[l]}}\sum_{j=0}^{l-1}g_{l,j}(-1)^{i+j+1}M^{(l)}_{i,j}=\frac{(-1)^{i+l}M^{(l+1)}_{i,l}}{\det g^{[l]}},
 \end{align}
where $M_{i,j}^{(l)}$ is the $(i,j)$-minor of the truncated moment matrix $g^{[l]}$ defined in \eqref{gl}.
Therefore,
\begin{align*}
  Q^{(l)}&=
  \frac{1}{\det g^{[l]}}\sum_{i=0}^l (-1)^{i+l} M_{i,l}^{(l+1)}  \xi_1^{(i)}\\&=
\frac{1}{\det g^{[l]}}\det
  \left(\begin{BMAT}{cccc|c}{cccc|c}
    g_{0,0}&g_{0,1}&\cdots&g_{0,l-1}&\xi_1^{(0)}\\
     g_{1,0}&g_{1,1}&\cdots&g_{1,l-1}&\xi_1^{(1)}\\
     \vdots &\vdots&            &\vdots&\vdots\\
        g_{l-1,0}&g_{l-1,1}&\cdots&g_{l-1,l-1}&\xi_1^{(l-1)}\\
          g_{l,0}&g_{l,1}&\cdots&g_{l,l-1}&\xi_1^{(l)}
\end{BMAT}\right),& l&\geq 1.
\end{align*}

The orthogonality relations for the dual system can be written also in two alternative forms
\begin{align}\label{dS.1}
    \begin{pmatrix}
    g_{0,0}&g_{0,1}&\cdots&g_{0,l-1}\\
    g_{1,0}&g_{1,1}&\cdots&g_{1,l-1}\\
    \vdots&\vdots&&\vdots\\
     g_{l-1,0}&g_{l-1,1}&\cdots&g_{l-1,l-1}
  \end{pmatrix}\begin{pmatrix}
    \bar S'_{0,l}\\\bar S'_{1,l}\\\vdots\\ \bar S_{l-1,l}
 \end{pmatrix}&=-(\bar S_{l,l})^{-1}\begin{pmatrix}
   g_{0,l}\\g_{1,l}\\\vdots\\g_{l-1,l}
 \end{pmatrix},& l&\geq 1,\\ \begin{pmatrix}
    g_{0,0}&g_{0,1}&\dots&g_{0,l}\\
    g_{1,0}&g_{1,1}&\dots&g_{1,l}\\
    \vdots&\vdots&&\vdots\\
     g_{l,0}&g_{l,1}&\dots&g_{l,l}
  \end{pmatrix}
  \begin{pmatrix}
  \bar S'_{0, l}\\
  \bar S'_{1, l} \\
  \vdots \\
  \bar S'_{l, l}
  \end{pmatrix}
  &=
  \begin{pmatrix}
  0 \\
  0 \\
  \vdots \\0\\
  1
  \end{pmatrix},& l&\geq 0.\label{dS.2}
\end{align}
As before,  \eqref{linear forms S} leads to the following expressions for the dual linear forms
\begin{align}
\bar  Q^{(l)}&=\sum_{k=0}^l \bar S'_{k,l}\xi_{2}^{(k)}
  \notag\\&=(\bar S_{l,l})^{-1}\Big(\xi_{2}^{(l)}-(\xi_2^{[l]})^\top(g^{[l]})^{-1}\begin{pmatrix}
    g_{0,l}\\g_{1,l}\\\vdots\\g_{l-1,l}
  \end{pmatrix}\Big)& \text{use \eqref{dS.1}}\label{dS1.1}\\
  &=(\xi^{[l+1]}_2)^\top
  (g^{[l+1]})^{-1}\begin{pmatrix}
    0 \\0\\\vdots \\0\\ 1
  \end{pmatrix}. &\text{use \eqref{dS.2}}\label{dS1.2}
\end{align}

 From \eqref{dS.2} we obtain
\begin{align}\label{s'}
    \bar S'_{j,l}=\big({g^{[l+1]}}^{-1}\big)_{j,l}=\frac{(-1)^{l+j}M^{(l+1)}_{l,j}}{\det(g^{[l+1]})}, \quad
    j=0,\dots,l,
\end{align}
and consequently
\begin{align*}
    \bar Q ^{(l)}&=\sum_{j=0}^{l} \bar S'_{j,l} \xi_2^{(j)}=\frac{1}{\det
    g^{[l+1]}}\sum_{j=0}^{l}(-1)^{l+j}M^{(l+1)}_{l,j}\xi_2^{(j)}
\\&=\frac{1}{\det g^{[l+1]}}\det
  \left(\begin{BMAT}{cccc|c}{cccc|c}
    g_{0,0}&g_{0,1}&\cdots&g_{0,l-1}& g_{0,l}\\
     g_{1,0}&g_{1,1}&\cdots&g_{1,l-1}& g_{1,l} \\
     \vdots &\vdots&        &\vdots&\vdots\\
        g_{l-1,0}&g_{l-1,1}&\cdots&g_{l-1,l-1}&g _{l-1,l}\\
        \xi_2^{(0)}& \xi_2^{(1)} & \cdots & \xi_2^{(l-1)} & \xi_2^{(l)}
\end{BMAT}\right),& l\geq 0.
\end{align*}

\paragraph{Proof Proposition \ref{proposition: expressions for ct}}
 We have
 \begin{align*}
 C_b^{(l)}=\frac{1}{\det g^{[l]}}
\sum_{k=0}^l(-1)^{k+l}M^{(l+1)}_{k,l}\sum_{k_2=\nu_{2,b}(l-1)}^{\infty}z^{-k_2-1}\int
x^{k_1(k)}w_{1,a_1(k)}(x)w_{2,b}(x)x^{k_2}\d \mu(x),
 \end{align*}
 which according to  \eqref{Gamma}  recasts into
\begin{align}\notag
 C_b^{(l)}&=\frac{1}{\det g^{[l]}}
\sum_{k=0}^l(-1)^{k+l}M^{(l+1)}_{k,l}\bar\Gamma_{k,b}^{(l)},\\
&=\frac{1}{\det g^{[l]}}\det
  \left(\begin{BMAT}{cccc|c}{cccc|c}
    g_{0,0}&g_{0,1}&\cdots&g_{0,l-1}&\bar\Gamma^{(l)}_{0,b}\\
     g_{1,0}&g_{1,1}&\cdots&g_{1,l-1}&\bar\Gamma^{(l)}_{0,b}\\
     \vdots &\vdots&            &\vdots&\vdots\\
        g_{l-1,0}&g_{l-1,1}&\cdots&g_{l-1,l-1}&\bar\Gamma^{(l)}_{l-1,b}\\
          g_{l,0}&g_{l,1}&\cdots&g_{l,l-1}&\bar\Gamma^{(l)}_{l,b},
\end{BMAT}\right),& l&\geq 1.&&\label{ctdet}
 \end{align}
We also obtain
\begin{align*}
\bar C^{(l)}_a=\frac{1}{\det g^{[l+1]}}\sum_{k=0}^l
(-1)^{l+k}M^{(l+1)}_{l,k}\sum_{k_1=\nu_{1,a}(l-1)}^{\infty}z^{-k_1-1}\int  x^{k_1}w_{1,a}(x)w_{2,a_2(k)}x^{k_2(k)}\d
\mu(x),
\end{align*}
which can be written as \eqref{Gamma}
\begin{align}\notag
 \bar C_a^{(l)}&=\frac{1}{\det g^{[l+1]}}
\sum_{k=0}^l(-1)^{k+l}M^{(l+1)}_{l,k}\Gamma_k^{(l)},\\
&=\frac{1}{\det g^{[l+1]}}\det
  \left(\begin{BMAT}{cccc|c}{cccc|c}
    g_{0,0}&g_{0,1}&\cdots&g_{0,l-1}&g_{0,l}\\
     g_{1,0}&g_{1,1}&\cdots&g_{1,l-1}&g_{1,l}\\
     \vdots &\vdots&            &\vdots&\vdots\\
        g_{l-1,0}&g_{l-1,1}&\cdots&g_{l-1,l-1}&g_{l-1,l}\\
          \Gamma^{(l)}_{0,a}& \Gamma^{(l)}_{1,a}&\cdots& \Gamma^{(l)}_{l-1,a}& \Gamma^{(l)}_{l,a}
\end{BMAT}\right),& l&\geq 1.&&\label{ctdet-dual}
 \end{align}

 \paragraph{Proof Proposition \ref{proposition: CD type relations}}
 From \eqref{mop-s} and \eqref{cauchy-S} we deduce that
  \begin{align*}
    (\bar\Cs_a(z))^\top\A_{a'}(z')&=(\chi_{1,a}^*(z))^\top\chi_{1,a'}(z')=\frac{\delta_{a,a'}}{z-z'},& |z'|<|z|,\\
     (\Cs_b(z))^\top\bar\A_{a'}(z')&=(\chi_{2,b}^*(z))^\top\chi_{2,b'}(z')=\frac{\delta_{b,b'}}{z-z'},& |z'|<|z|,\\
      (\bar\Cs_a(z))^\top\Cs_{b}(z')&=(\chi_{1,a}^*(z))^\top g\chi^*_{2,b}(z').
  \end{align*}
  The two first relations imply the corresponding equations in the Proposition. For the third we observe that from
  \eqref{compact.g} we get
  \begin{align*}
    (\chi_{1,a}^*(z))^\top g\chi^*_{2,b'}(z')&=\int     (\chi_{1,a}^*(z))^\top
    \xi_1(x)(\xi_2(x))^\top\chi^*_{2,b'}(z')\d\mu(x)\\
    &=\int     (\chi_{1,a}^*(z))^\top
    \chi_{1,a}(x)(\chi_{2,b}(x))^\top\chi^*_{2,b'}(z')w_{1,a}(x)w_{2,b}(x)\d\mu(x)\\
       &=\int     \frac{1}{(z-x)(z'-x)}w_{1,a}(x)w_{2,b}(x)\d\mu(x)\\
       &=-\frac{1}{z-z'}\int     \Big(\frac{1}{z-x}-\frac{1}{z'-x}\Big)w_{1,a}(x)w_{2,b}(x)\d\mu(x).
  \end{align*}

\paragraph{Proof Proposition \ref{associated forms as mop}}
Using Definition \ref{associated linear forms} for the linear forms $Q^{(l)}_{+a}$ and multiplying by
$(\xi_2^{[l]}(x))^{\top}$ we have
 \begin{align*}
 Q_{+a}^{(l)}(x) (\xi_2^{[l]}(x))^{\top}=\xi_1^{(l_{+a})}(x)(\xi^{[l]}_2(x))^{\top}-
       \begin{pmatrix}
    g_{l_{+a},0}&g_{l_{+a},1}&\cdots& g_{l_{+a},l-1}
  \end{pmatrix}(g^{[l]})^{-1}\xi_1^{[l]}(x)(\xi^{[l]}_2(x))^{\top},
 \end{align*}
 integrating both sides we get
 \begin{align*}
 \int  Q_{+a}^{(l)}(x) (\xi^{[l]}_2(x))^{\top} \d \mu(x)
  & =\int  \xi_1^{(l_{+a})}(x)(\xi^{[l]}_2(x))^{\top} \d \mu(x)\\&\quad\quad\quad\quad-
       \begin{pmatrix}
    g_{l_{+a},0}&\cdots& g_{l_{+a},l-1}
  \end{pmatrix}(g^{[l]})^{-1}\int  \xi_1^{[l]}(x)(\xi^{[l]}_2(x))^{\top} \d \mu(x)  \\
  & =\int  \xi_1^{(l_{+a})}(x)(\xi^{[l]}_2(x))^{\top} \d \mu(x)-
       \begin{pmatrix}
    g_{l_{+a},0}& g_{l_{+a},1} &\cdots& g_{l_{+a},l-1}
  \end{pmatrix}(g^{[l]})^{-1}g^{[l]} \\
  & = \begin{pmatrix}
    g_{l_{+a},0}&g_{l_{+a},1}&\cdots& g_{l_{+a},l-1}
  \end{pmatrix}-
  \begin{pmatrix}
    g_{l_{+a},0}&g_{l_{+a},1}&\cdots& g_{l_{+a},l-1}
  \end{pmatrix}\\
  & =0,
  \end{align*}
  that written componentwise gives the following orthogonality relations
  \begin{align*}
\int_{\R} Q^{(l)}_{+a}(x) w_{2,a_2(k)}(x) x^{k_2(k)} \d \mu(x),  \quad k=0,\dots,l-1,
\end{align*}
or equivalently
\begin{align*}
\int_{\R} Q^{(l)}_{+a}(x) w_{2,b}(x) x^k \d \mu(x)&=0,  &0&\leq k\leq \nu_{2,b}(l-1)-1,&b&=1,\dots,p_2.
\end{align*}
Notice that, $A^{(l)}_{+a,a}$ is monic and $\deg A^{(l)}_{+a,a}(x)=k_1(l_{+a})$ but $A^{(l)}_{+a,a'}$ with $a \neq a'$
satisfy $\deg A^{(l)}_{+a,a'} \leq k_1((l-1)_{-a'})$.
This means that the set of polynomials $A^{(l)}_{+a,a'}(x)$ have degrees determined by $\vec \nu_1(l-1)+\vec e_{1,a}$
and a normalization with respect to the $a$-th component of type  II; i.e,
$Q_{+a}^{(l)}=Q^{(\rII,a)}_{[\vec\nu_1(l-1)+\vec e_{1,a};\vec \nu_2(l-1)]}$.

In a similar way, the associated linear forms $Q^{(l)}_{-b}(x)$ solve a mixed multiple orthogonal problem that can be
obtained as follows. From Definition \ref{associated linear forms} and multiplying by $(\xi_2^{[l]}(x))^{\top}$ we
get
\begin{align*}
Q_{-b}^{(l)}(x)(\xi_2^{[l+1]}(x))^{\top}=e_{\bar
l_{-b}}^\top(g^{[l+1]})^{-1}\xi_1^{[l+1]}(x)(\xi_2^{[l+1]}(x))^{\top},
\end{align*}
integrating both sides
\begin{align*}
\int_{\R} Q_{-b}^{(l)}(x)(\xi_2^{[l+1]}(x))^{\top} \d \mu(x)=e_{\bar
l_{-b}}^\top(g^{[l+1]})^{-1}\int_{\R}\xi_1^{[l+1]}(x)(\xi_2^{[l+1]}(x))^{\top} \d \mu(x) = e_{\bar l_{-b}}^{\top},
\end{align*}
and written componentwise
\begin{align*}
\int_{\R} Q_{-b}^{(l)}(x) x^{k_2(k)}w_{2,a_2(k)}(x) \d \mu(x)= \delta_{k,\bar l_{-b}}, \quad k=0, \cdots, l,
\end{align*}
that is equivalent to
\begin{align*}
\int_{\R} Q^{(l)}_{-b}(x) w_{2,b}(x) x^k \d \mu(x)&=\delta_{k,\bar l_{-b}},  &0&\leq k\leq
\nu_{2,b}(l)-1,&b&=1,\dots,p_2.
\end{align*}
Hence, the set $A^{(l)}_{-b,a'}$ is a type I normalized to the $b$-th component solution for a mixed multiple
orthogonality problem; the degrees satisfy $\deg A^{(l)}_{-b,a'} \leq l_{-a'}$. Moreover, the fact that the last
orthogonality condition in the $b$-th component is missing gives the identification
$Q_{-b}^{(l)}=Q^{(\rI,b)}_{[\vec\nu_1(l);\vec \nu_2(l)-\vec e_{2,b}]}$.

Using Definition $\ref{associated linear forms}$ and multiplying by $\xi^{[l]}_1(x)$ we have
\begin{align*}
\bar Q_{+b}^{(l)}(x)\xi^{[l]}_1(x)&=\Big(\xi_1^{[l]}(x)\xi_2^{(\bar
l_{+b})}(x)-\xi^{[l]}_1(x)(\xi^{[l]}_{2}(x))^\top(g^{[l]})^{-1}
\begin{pmatrix}
    g_{0,\bar l_{+b}}\\g_{1,\bar l_{+b}}\\\vdots\\ g_{l-1,\bar l_{+b}}
  \end{pmatrix}\Big),
\end{align*}
and integrating both sides
\begin{align*}
\int _\R\bar Q_{+b}^{(l)}(x)\xi_1(x)^{[l]} \d \mu(x) &= \Big(\int  \xi^{[l]}_1(x)\xi_2^{(\bar l_{+b})}(x)\d
\mu(x)-\int  \xi^{[l]}_1(x)(\xi_{2}^{[l]}(x))^\top \d \mu(x)(g^{[l]})^{-1}
\begin{pmatrix}
    g_{0,\bar l_{+b}}\\g_{1,\bar l_{+b}}\\\vdots\\ g_{l-1,\bar l_{+b}}
  \end{pmatrix}\Big) \\
& =\Big(\int  \xi^{[l]}_1(x)\xi_2^{(\bar l_{+b})}(x)\d \mu(x)-(g^{[l]})(g^{[l]})^{-1}
\begin{pmatrix}
    g_{0,\bar l_{+b}}\\g_{1,\bar l_{+b}}\\\vdots\\ g_{l-1,\bar l_{+b}}
  \end{pmatrix}\Big)\\
& =  \begin{pmatrix}
    g_{0,\bar l_{+b}}\\g_{1,\bar l_{+b}}\\\vdots\\ g_{l-1,\bar l_{+b}}
  \end{pmatrix}
  -
  \begin{pmatrix}
    g_{0,\bar l_{+b}}\\g_{1,\bar l_{+b}}\\\vdots\\ g_{l-1,\bar l_{+b}}
  \end{pmatrix}\\
  &=0,
\end{align*}
that componentwise leads to the following orthogonality relations
\begin{align*}
\int_{\R} \bar Q^{(l)}_{+b}(x) w_{1,a_1(k)}(x) x^{k_1(k)} \d \mu(x)=0,  \quad k=0,\dots,l-1,
\end{align*}
or alternatively
\begin{align*}
 \int  \bar Q_{+b}^{(l)}(x)w_{1,a}(x)x^k\d \mu(x)&=0,&0&\leq k\leq \nu_{1,a}(l-1)-1,&a&=1,\dots,p_1.
\end{align*}
Notice that, $\bar A^{(l)}_{+b,b}$ is monic and $\deg \bar A^{(l)}_{+b,b}=k_2(\bar l_{+b})$ but $\bar A^{(l)}_{+b,b'}$
with $b \neq b'$ satisfy $\deg A^{(l)}_{+b,b'} \leq k_1(\overline{(l-1)}_{-b'})$.
This means that the polynomials $\bar A^{(l)}_{+b,b'}$ have degrees determined by $\vec \nu_2(l-1)+\vec e_{2,b}$ and a
normalization with respect to the $b$-th component of type  II; i.e,   $\bar Q_{+b}^{(l)}=\bar
Q^{(\rII,b)}_{[\vec\nu_2(l-1)+\vec e_{2,b};\vec \nu_1(l-1)]}$.

Finally, we obtain the orthogonality relations for the linear forms $\bar Q^{(l)}_{-a}(x)$. From the definition we
get
\begin{align*}
\xi_{1}^{[l+1]}(x)\bar Q^{(l)}_{-a}(x)=\xi_{1}^{[l+1]}(x)(\xi_{2}^{[l+1]}(x))^\top(g^{[l+1]})^{-1}e_{l_{-a}},
\end{align*}
and integrating both sides
\begin{align*}
\int_{\R} \xi_1^{[l+1]}(x) \bar Q_{-a}^{(l)}(x) \d \mu(x) =\Big(\int_{\R}\xi_{1}^{[l+1]}(x)(\xi_{2}^{[l+1]}(x))^\top
\d \mu(x) \Big) (g^{[l+1]})^{-1}e_{l_{-a}}=e_{l_{-a}},
\end{align*}
and componentwise that means
\begin{align*}
\int_{\R} \bar Q_{-a}^{(l)}(x) x^{k_1(k)}w_{1,a_1(k)} (x) \d \mu(x)= \delta_{k,l_{-a}}, \quad k=0, \cdots, l,
\end{align*}
or equivalently
\begin{align*}
\int_{\R} \bar Q_{-a}^{(l)}(x) x^{k}w_{1,a} (x) \d \mu(x)&=\delta_{k,l_{-a}}, &0&\leq k\leq
\nu_{1,a}(l)-1,&a&=1,\dots,p_1,
\end{align*}
so the set $\bar A^{(l)}_{-a,b'}$ is a type I normalized to the $a$-th component solution for a mixed multiple
orthogonality problem. The degrees satisfy $\deg A^{(l)}_{-a,b'} \leq l_{-b'}$; and therefore we conclude that $\bar
Q_{-b}^{(l)}=Q^{(\rI,b)}_{[\vec\nu_2(l);\vec \nu_1(l)-\vec e_{1,a}]}$.

\paragraph{Proof of Proposition \ref{proposition: symmetry}}
Taking $+$ and $-$ parts in \eqref{bigradedj} we obtain
\begin{align*}
\sum_{a=1}^{p_1}L_a^j=\sum_{a=1}^{p_1}(L_a^j)_++\sum_{a=1}^{p_1}(L_a^j)_-=\sum_{a=1}^{p_1}(L_a^j)_++\sum_{b=1}^{p_2}(\bar
L_b^j)_-=\sum_{a=1}^{p_1}B_{j,a}+\sum_{b=1}^{p_2} \bar B_{j,b}=\sum_{b=1}^{p_2}\bar L_b^j.
\end{align*}

Using Lax equations and observing that $L_aL_{a'}=L_{a'}L_a$ and $\bar L_b\bar L_{b'}=\bar L_{b'}\bar L_{b}$ we have
the following symmetries for the Lax operators
\begin{align*}
\big(\sum_{a=1}^{p_1}\frac{\partial}{\partial t_{j,a}}+\sum_{b=1}^{p_2}\frac{\partial}{\partial \bar
t_{j,b}}\big)L_{a'}=\Big[\sum_{a=1}^{p_1}B_{j,a}+\sum_{b=1}^{p_2} \bar
B_{j,b},L_{a'}\Big]=\sum_{a=1}^{p_1}[L_a^j,L_{a'}]=0, \\
\big(\sum_{a=1}^{p_1}\frac{\partial}{\partial t_{j,a}}+\sum_{b=1}^{p_2}\frac{\partial}{\partial \bar
t_{j,b}}\big)\bar
L_{b'}=\Big[\sum_{a=1}^{p_1}B_{j,a}+\sum_{b=1}^{p_2} \bar B_{j,b},\bar L_{b'}\Big]=\sum_{b=1}^{p_2}[\bar L_b^j,\bar L_{b'}]=0.
\end{align*}
From \eqref{flow.mop} we conclude that the multiple orthogonal polynomials and their duals are also invariant
\begin{align*}
  \big(\sum_{a=1}^{p_1}\frac{\partial}{\partial t_{j,a}}+\sum_{b=1}^{p_2}\frac{\partial}{\partial \bar
  t_{j,b}}\big)\A_{a'}&=\Big(\sum_{a=1}^{p_1}B_{j,a}+\sum_{b=1}^{p_2} \bar
  B_{j,b}-x^j\Big)\A_{a'}=\big(J^j-x^j\big)\A_{a'}=0,\\
    \big(\sum_{a=1}^{p_1}\frac{\partial}{\partial t_{j,a}}+\sum_{b=1}^{p_2}\frac{\partial}{\partial \bar
    t_{j,b}}\big)\bar\A_{b'}&=
-\Big(\sum_{a=1}^{p_1}B_{j,a}+\sum_{b=1}^{p_2} \bar
B_{j,b}-x^j\Big)^\top\bar\A_{b'}=-\big(J^j-x^j\big)^\top\bar\A_{a'}=0.
\end{align*}

\paragraph{Proof of Proposition \ref{proposition: bilinear fubini-cauchy}}
We just follow the following chain of identities
  \begin{align*}
    W_{\vec n_1,\vec n_2}(t,s)\pi_{\vec n_1',\vec n_1}^\top  (W_{\vec n_1',\vec n_2'}(t',s'))^{-1}&=
    W_{\vec n_1,\vec n_2}(t,s)\pi_{\vec n_1',\vec n_1}^\top g_{\vec n_1',\vec n_2'}(\bar W_{\vec n_1',\vec
    n_2'}(t',s'))^{-1} \\
&    \hspace*{-4cm}=S_{\vec n_1,\vec n_2}(t,s)    W_{0,\vec n_1}(t,s)\pi_{\vec n_1',\vec n_1}^\top\Big(\int
\xi_{\vec n_1'}(x)\xi_{\vec n_2'}^\top(x)\d \mu(x)\Big)(W_{0,\vec n_2'}(t',s'))^{-1}
    (\bar S_{\vec n_1',\vec n_2'}(t',s') )^{-1} \\
    &    \hspace*{-4cm}=S_{\vec n_1,\vec n_2}(t,s)    W_{0,\vec n_1}(t,s)\Big(\int  \xi_{\vec n_1}(x)\xi_{\vec
    n_2'}^\top(x)\d \mu(x)\Big)(W_{0,\vec n_2'}(t',s'))^{-1}
    (\bar S_{\vec n_1',\vec n_2'}(t',s') )^{-1} \\
     &    \hspace*{-4cm}=S_{\vec n_1,\vec n_2}(t,s)\Big(\int  \xi_{\vec n_1}(x,t,s)\xi_{\vec n_2'}^\top(x,t',s')\d
     \mu(x)\Big)
    (\bar S_{\vec n_1',\vec n_2'}(t',s') )^{-1} \\
    &    \hspace*{-4cm}=\int  (S_{\vec n_1,\vec n_2}(t,s)(\xi_{\vec n_1}(x,t,s))(\bar S_{\vec n_1',\vec
    n_2'}^\top(t',s')) ^{-1} \xi_{\vec n_2'}(x,t',s') )^\top \d \mu(x)\\
    &\hspace*{-4cm}=\int_{\R} Q_{\vec n_1,\vec n_2}(x,t,s)\bar Q_{\vec n_1',\vec n_2'}^\top(x,t',s')\d \mu(x),
      \end{align*}
      where $\xi_{\vec n_1}(x,t,s)$ and $\xi_{\vec n_2'}(x,t',s')$ represent the vectors of weighted monomials but
      with evolved weights.

      \paragraph{Proof of Proposition \ref{taumopis}}

      To find the $\tau$-function of the multiple orthogonal representation we first need two lemmas
\begin{lemma}\label{lemarows}
Let $R^{(j)}$ be the $j$-th row of $\tau^{(l)}(t)$ and $R^{(j)}_z$ the $j$-th row of $\tau^{(l)}(t-[z^{-1}]_a)$, then
\begin{equation}\label{eqrows}
    R^{(j)}_z=R^{(j)}-\delta_{a_1(j),a}z^{-1}R^{(j')},
\end{equation}
where $j'=j+1$ if $r_1(j)<n_{1,a}-1$, but $j'=j+(|\vec n_1|-n_{1,a})+1$ if $r_1(j)=n_{1,a}-1$. This is also valid for
$\tau^{(l)}_{-a}$ ,  $\tau^{(l)}_{+a,-a'}$ and for $\tau^{(l)}_{-b,-a'}$. \\

Let now be $C^{(j)}$ the j-th column of $\bar \tau^{(l)}$ and $C^{(j)}_z$ the $j$-th column of $\bar
\tau^{(l)}(t+\overline{[z^{-1}]}_b)$, then
\begin{equation}\label{eqcolumns}
    C^{(j)}_z=C^{(j)}-\delta_{a_2(j),b}z^{-1}C^{(j')},
\end{equation}
where  $j'=j+1$ if $r_2(l)<n_{2,b}-1$ but $j'=j+(|\vec n_2|-n_{2,b})+1$ if $r_2(j)=n_{2,b}-1$. This is also valid for
$\bar \tau^{(l)}_{-b}$, $\bar \tau^{(l)}_{+b,-b'}$ and for $\bar \tau^{(l)}_{-a,-b'}.$
\end{lemma}
\begin{proof}
It follows directly from \eqref{vertex weight} and \eqref{vertex
dualweight}.
\end{proof}

Let us  recall the skew multi-linear character of  determinants and the consequent formulation in terms  of wedge
products of covectors. Observe that
\begin{lemma} \label{lemacovectors}
Given a set of covectors $\{r_1,\dots,r_n\}$ it can be  shown that
     \begin{align}\label{covectors}
       \bigwedge_{j=1}^n(zr_j-r_{j+1})=\sum_{j=1}^{n+1}(-1)^{n+1-j}z^{j-1} \;r_1\wedge r_2\wedge\dots\wedge \hat
       r_j\wedge\dots \wedge r_{n+1},
     \end{align}
     where the notation $ \hat r_j$ means that we have erased the covector $r_j$ in the wedge product
     $r_1\wedge\dots\wedge r_{n+1}$.
\end{lemma}
\begin{proof}
It can be done directly by induction.
\end{proof}

The proof of Proposition \ref{taumopis} relies on Lemma \ref{lemarows}, Lemma \ref{lemacovectors}, Corollary
\ref{cormops} and Proposition \ref{det-associated}. First let's focus on \eqref{taumops}; it is clear that
$z^{\nu_{1,a}(l)-1} \tau_{-a}^{(l)}(t-[z^{-1}]_a)$ expands in $z$ according to \eqref{eqrows} for $\tau ^{(l)}_{-a}$
and to \eqref{covectors}. Now $n=k_1(l_{-a})$ and the covectors  $r_j$ should be taken equal to those rows $R^{(j)}$
with $a_1(j)=a$. Observe that there are only $k_1( l_{-a})(=\nu_{1,a}(l)-1)$ rows which are non-trivially transformed.
In this form we get the identification of
     \eqref{detmops} with  \eqref{taumops}, where the terms corresponding to the wedge with one covector deleted
     corresponds to the minors
     $M_{j,l}^{[l+1]}$. Now, looking to \eqref{taumops+} and \eqref{taumops-} we expand again in $z$ and use the same
     technique based on \eqref{eqrows} for $\tau ^{(l)}_{+a,-a'}$ and $\tau ^{(l)}_{-b,-a}$ and \eqref{covectors}.
     These allow to link \eqref{detmops+} to \eqref{taumops+} and \eqref{detmops-} to \eqref{taumops-}.

To prove \eqref{taudualmops} we proceed similarly. Looking at \eqref{eqcolumns} for $\bar \tau^{(l)}_{-b}$ observe
that there are only $k_2(\bar l_{-b})(=\nu_{2,b}(l)-1)$ columns which are non-trivially transformed. Now, recalling
\eqref{dualmops.2} and using \eqref{covectors} but with $r_j$ being the columns $C^{(j)}$, such that $a_2(j)=b$, and
$n=k_2(\bar l_{-b})$, we get the desired result. Finally for $\eqref{taudualmops+}$ and $\eqref{taudualmops-}$ we
expand around $z$ to see the equivalence between \eqref{detdualmops+} and \eqref{taudualmops+} and the equivalence
between \eqref{detdualmops-} and \eqref{taudualmops-}.

\paragraph{Proof of Proposition \ref{proposition: tau cauchy}}
We need the following two lemmas:
\begin{lemma}
  Let $R^{(j)}$ be the $j$-th row of $g_{+a}^{[l+1]}$ and $R^{(j)}_z$ the $j$-th row of
  $g_{+a}^{[l+1]}(t+[z^{-1}]_a)$,   we get
\begin{equation}
    R^{(j)}_z=R^{(j)}+\delta_{a_1(j),a}\sum_{j'=1}^\infty z^{-k_1(j')}R^{(j+j')}\delta_{a_1(j+j'),a}.
\end{equation}
Let  $C^{(j)}$ be the $j$-th column of $\bar g_{+b}^{[l+1]}$ and
 $C^{(j)}_z$ the $j$-th column of $\bar g_{+b}^{[l+1]}(t-\overline{[z^{-1}]}_b)$, then
 \eqref{vertex dualweight} gives
\begin{equation}
    C^{(j)}_z=C^{(j)}+\delta_{a_2(j),b}\sum_{j'=1}^\infty z^{-k_2(j')}C^{(j+j')}\delta_{a_2(j+j'),b}.
\end{equation}
\end{lemma}
\begin{proof}
For the first equality insert the expansion
\begin{align*}
  \Big(1-\frac{x}{z}\Big)^{-1}=\sum_{k=0}^\infty\frac{x^k}{z^k}
\end{align*}
into \eqref{vertex dualweight}. The other equation is proven similarly.
\end{proof}

\begin{lemma}
The following  identity
\begin{align}\label{covectors.2}
  \bigwedge_{j=1}^n
  \Big(\sum_{i=0}^{\infty}r_{j+i}z^{-i}\Big)=r_1\wedge\dots\wedge r_{n-1}\wedge \Big(\sum_{i=0}^\infty
  r_{n+i}z^{-i}\Big)
\end{align}
holds.
\end{lemma}
\begin{proof}
Use induction in $n$.
\end{proof}
Finally Proposition \ref{proposition: tau cauchy} is proven using  \eqref{ctdet-dual} and  \eqref{ctdet}.

\section{Discrete flows associated with binary Darboux transformations}\label{III}

 When the supports of the measures are not bounded from below, (in which case  the new ``weights" \eqref{d.evolved} do
 not have in general a definite sign and therefore
should not be considered as such), there is an alternative form of constructing discrete flows which preserve the
positiveness/negativeness of the measures.  The
construction is based in the previous one, but now the shift is the composition of two consecutive shifts associated
with the pair $\lambda_a(n)$ and $\lambda_a(n+1)$, being complex numbers conjugate to each other; i.e., we consider
 \begin{definition}
We define a deformed moment matrix
     \begin{align}
   g(s)=D'_0(s)g(\bar D'_0(s))^{-1}.
 \end{align}
 with
 \begin{align}
 \begin{aligned}
   D'_0&:=\sum_{a=1}^{p_1}D'_{0,a}, \\ D'_{0,a}&:=\begin{cases}
     \prod_{n=1}^{s_a}\big(|\lambda_a(n)|^2\Pi_{1,a}-
   2\operatorname{Re}(\lambda_a(n))\Lambda_{1,a}+\Lambda^{2}_{1,a}\big), &s_a>0,\\
     \Pi_{1,a},& s_a=0,\\
       \prod_{n=1}^{|s_a|}\big(|\lambda_a(-n)|^2\Pi_{1,a}-
   2\operatorname{Re}(\lambda_a(-n))\Lambda_{1,a}+\Lambda^{2}_{1,a}\big)^{-1},& s_a<0,
   \end{cases}
   \end{aligned}
   \end{align}
   \begin{align}
   \begin{aligned}
   (\bar   D_0')^{-1}&:=   \sum_{b=1}^{p_2} \big((\bar   D_0(s)')^{-1}\big)_b, \\ \big((\bar   D_0(s)')^{-1}\big)_b
  & :=\begin{cases}
   \prod_{n=1}^{\bar s_b}(|\bar\lambda_b(n)|^2\Pi_{2,b}-2\operatorname{Re}(\bar\lambda_b(n))\Lambda_{2,b}^\top)
   +(\Lambda_{2,b}^\top)^{2} )\big) &\bar s_b>0,\\
     \Pi_{2,b},&\bar s_b=0,\\
  (\prod_{n=1}^{\bar s_b}\big(|\bar\lambda_b(-n)|^2\Pi_{2,b}-2\operatorname{Re}(\bar\lambda_b(-n))\Lambda_{2,b}^\top)
   +(\Lambda_{2,b}^\top)^{2} \big)^{-1},&\bar s_b<0.
   \end{cases}
\end{aligned}
 \end{align}
 \end{definition}
\begin{pro}
The previously defined deformed moment matrix corresponds to a moment matrix with the following  positive/negative evolved
weights
\begin{align}\label{d.evolved2}\begin{aligned}
  w_{1,a}(s,x)&=\Ds_{a}'(x,s_a)w_{1,a}(x),& \Ds'_a&:=\begin{cases}
     \prod_{n=1}^{s_a}|x-\lambda_a(n)|^2, &s_a>0,\\
     1,& s_a=0,\\
       \prod_{n=1}^{|s_a|}|x-\lambda_a(-n)|^{-2},& s_a<0,
   \end{cases}\\
   w_{2,b}(s,x)&=\bar \Ds'_b(x,\bar s_b)^{-1}w_{2,b}(x),& (\bar \Ds'_b)^{-1}&:=\begin{cases}
     \prod_{n=1}^{\bar s_b}|x-\bar \lambda_b(n)|^2, &\bar s_b>0,\\
 1,&\bar s_b=0,\\
       \prod_{n=1}^{|\bar s_b|}|x-\bar \lambda_b(-n)|^{-2},&\bar s_b<0.
   \end{cases}\end{aligned}
\end{align}
\end{pro}
Proceeding as in the previous case
\begin{definition}
We introduce
  \begin{align}
  \begin{aligned}
q'_a&:=
  \I-\Pi_{1,a}(\I-|\lambda_a(s_a+1)|^2)-2\operatorname{Re}(\lambda_a(s_a+1))
  \Lambda_{1,a}+\Lambda_{1,a}^{2},\\
\bar q'_b&:=
\I  -\Pi_{2,b}(\I-|\bar \lambda_b(\bar s_b+1)|^2)-2\operatorname{Re}(\bar\lambda_b(\bar
s_b+1))(\Lambda_{2,b}^\top)+(\Lambda_{2,b}^\top)^{2},
\end{aligned}
\end{align}
and
\begin{align}
\label{delta'}
\begin{aligned}
  \delta_a'&:=\I-C_{a}(\I-|\lambda_a(s_a+1)|^2)-2\operatorname{Re}(\lambda_a(s_a+1))
 L_{a}+L_a^2,\\
    \bar\delta_b'&:=\I-\bar C_{b}(\I-|\bar \lambda_b(\bar s_b+1)|^2)-2\operatorname{Re}(\bar\lambda_b(\bar s_b+1))\bar
    L_b+\bar L_b ^2.
    \end{aligned}
\end{align}
\end{definition}

 The wave and adjoint wave functions  now have the form
 \begin{align}
    \Psi_a^{(k)}(z,s)&=A^{(k)}_a(z,s)\Ds'_a(z,s_a),&
    ( \bar \Psi_b^*)^{(k)}(z,s)&=\bar A_b^{(k)}(z,s)\bar\Ds'_b(z,\bar s_b)^{-1},
\end{align}
and the expressions \eqref{linear.form.baker.1}-\eqref{cauchy.baker} still hold.

If we introduce $\omega_a'$ and $\bar\omega_b'$ as in \eqref{omega} but replacing $\delta$ by $\delta'$,  the
equations \eqref{d.wave}-\eqref{d.ZS1} hold
true by replacement of $\omega$ by $\omega'$. Now, the form $\omega'$ differs from \eqref{form.omega} as now we have
   \begin{align}\label{form.omega'}\begin{aligned}
                     \omega_a'&=\omega'_{a,0}\Lambda^{2(|\vec n_1|-n_{1,a}+1)}+\dots+\omega'_{a,2(|\vec
                     n_1|-n_{1,a}+1)},\\
             { \bar \omega_b}'&=\bar\omega'_{b,0}(\Lambda^\top)^{2(|\vec n_2|-n_{2,b}+1)}+\dots+\bar\omega'_{b,2(|\vec
             n_2|-n_{2,b}+1)}.
           \end{aligned}
               \end{align}

With the definition of
\begin{align}
\begin{aligned}
    \gamma_{a,a'}'(s,x)&:=  (1-\delta_{a,a'}(1-|x-\lambda_a(s_a+1)|^2),\\
\gamma_{b,b'}'(s,x)&:=  (1-\delta_{b,b'}(1-|x-\bar\lambda_b(\bar s_b+1)|^2),
\end{aligned}
\end{align}
we have that
\begin{pro}
 The present setting  \eqref{multi.variante} and \eqref{dual.multi.variante} are replaced by
\begin{align}
&\left\{\begin{aligned}
  (T_{a'} A_a^{(k)})\gamma'_{a,a'}&=\omega_{a',0}'A_a^{(k+2(|\vec n_1|-n_{1,a}+1))}+\dots+
  \omega_{a',2(|\vec n_1|-n_{1,a}+1)}(k) A_a^{(k)},\\
  \bar T_{b'}A_a^{(k)}&=\bar\omega_{b,0}'(k)A_a^{(k-2(|\vec n_2|-n_{2,b'}+1))}+\dots+\bar\omega_{b,2(|\vec
  n_2|-n_{2,b'}+1)}' A_a^{(k)},
\end{aligned}\right.\\
&\left\{\begin{aligned}
 \rho_{a',0}'(T_{a'} \bar A_b^{(k-2(|\vec n_1|-n_{1,a'}+1))})+\dots+ \rho_{a',2(|\vec n_1|-n_{1,a'}+1)}'(k) (T_{a'} \bar
 A_b^{(k)})&= \bar A_b^{(k)},\\
\Big( \bar\rho_{b',0}' (k)(\bar T_{b'} \bar A_b^{(k+2(|\vec n_2|-n_{2,b'}+1)})+
 \dots+\bar\rho'_{b',2(|\vec n_2|-n_{2,b'}+1)} (\bar T_{b'} \bar A_b^{(k)})\Big)\bar\gamma_{b,b'}'&= \bar A_b^{(k)}.
\end{aligned}\right.
\end{align}
\end{pro}
\end{appendices}


\begin{thebibliography}{99}

\bibitem{adler}M. Adler and P. van Moerbeke, \emph{Group factorization, moment matrices and Toda lattices}, Int. Math.
    Res. Notices \textbf{12} (1997) 556-572.

\bibitem{adler-van moerbeke} M. Adler and P. van Moerbeke, \emph{Generalized orthogonal polynomials, discrete KP and
    Riemann--Hilbert  problems}, Commun. Math. Phys. \textbf{207} (1999) 589-620.

\bibitem{adler-vanmoerbeke 0} M. Adler and P. van Moerbeke, \emph{Vertex operator solutions to the discrete KP
    hierarchy}, Commun. Math. Phys. \textbf{203} (1999) 185-210.

 \bibitem{adler-van moerbeke 1} M. Adler and P. van Moerbeke, \emph{The spectrum of coupled random matrices}, Ann.
     Math. \textbf{149} (1999) 921-976.

\bibitem{adler-van moerbeke 2} M. Adler and P. van Moerbeke, \emph{Darboux transforms on band matrices, weights and
associated polynomials}, Int. Math. Res. Not. \textbf{18} (2001) 935-984.

 \bibitem{adler-vanmoerbeke 5} M. Adler, P. van Moerbeke and P. Vanhaecke, \emph{Moment matrices and multi-component
     KP, with applications to random matrix theory}, Commun. Math. Phys. \textbf{286} (2009) 1-38.

\bibitem{cum} C. \'{A}lvarez-Fern\'{a}ndez, U. Fidalgo and M. Ma\~{n}as,
\emph{The multicomponent 2D Toda hierarchy: generalized matrix orthogonal polynomials, multiple
 orthogonal polynomials and Riemann--Hilbert problems}.
Inv. Prob. \textbf{26} (2010) 055009   (17 pp).

\bibitem{Ap} R. Apery. \emph{Irrationalite de $\zeta(2)$ et $\zeta(3)$}, Ast\`{e}risque \textbf{61} (1979) 11-13.

\bibitem{bergvelt} M. J. Bergvelt and A. P. E. ten Kroode, \emph{Partitions, vertex operators constructions and
    multi-component KP equations}, Pacific J. Math. \textbf{171} (1995) 23-88.

\bibitem{Be} F. Beukers, \emph{Pad\'{e} approximation in number theory,} Lec. Not. Math. \textbf{888}, Springer
    Verlag, Berlin, 1981, 90-99.

\bibitem{BleKui} P.M. Bleher and A.B.J. Kuijlaars, \emph{ Random matrices with external source and multiple orthogonal
    polynomials}, Int. Math. Res. Not. \textbf{2004} (2004), 109-129.

\bibitem{carlet} G. Carlet, \emph{The extended bigraded Toda hierarchy} J. Phys. A: Math.  Gen. \textbf{39} (2006)
    9411-9435.

\bibitem{Co} J. Coates, \emph{On the algebraic approximation of functions}, I, II, III. Indag. Math. \textbf{28}
    (1966) 421-461.

\bibitem{daems-kuijlaars} E. Daems and A. B. J. Kuijlaars,  \emph{Multiple orthogonal polynomials of mixed type and
    non-intersecting Brownian motions}, J. of Approx. Theory \textbf{146} (2007) 91-114.

\bibitem{daems-kuijlaars0} E. Daems and A. B. J. Kuijlaars,  \emph{A Christoffel--Darboux formula for multiple
    orthogonal polynomials}, J. of Approx. Theory \textbf{130} (2004) 188-200.

    \bibitem{date1}  E. Date, M. Jimbo, M. Kashiwara,  and T. Miwa, \emph{Operator approach to the
        Kadomtsev-Petviashvili equation. Transformation groups for soliton equations. III}  J. Phys. Soc. Jpn.
        \textbf{50},  (1981) 3806-3812.

\bibitem{date2}   E. Date, M. Jimbo, M. Kashiwara,  and T. Miwa, \emph{Transformation groups for soliton equations.
    Euclidean Lie algebras and reduction of the KP hierarchy},  Publ. Res. Inst. Math. Sci. \textbf{18} (1982)
    1077-1110.

\bibitem{date3}   E. Date, M. Jimbo, M. Kashiwara,  and T. Miwa, \emph{Transformation groups for soliton equations} in
    \emph{Nonlinear Integrable Systems-Classical Theory and Quantum Theory}  M. Jimbo and T. Miwa (eds.) World
    Scientific, Singapore, 1983, pp. 39-120.

    \bibitem{felipe} R. Felipe  F. Ongay,  \emph{Algebraic aspects of the discrete KP hierarchy}, Linear Algebra App.
        \textbf{338} (2001) 1-17.


\bibitem{LF2} U. Fidalgo Prieto and G. L\'{o}pez Lagomasino, \emph{Nikishin systems are perfect}, \texttt{arXiv:1001.0554}
    (2010).

\bibitem{fokas} A. S. Fokas, A. R. Its and A. V. Kitaev, \emph{The isomonodromy approach to matrix models in 2D quatum
    gravity}, Commun. Math. Phys. (1992) 395-430.

\bibitem{He} Ch. Hermite, \emph{Sur la fonction exponentielle}, C. R. Acad. Sci. Paris 77 (1873), 18-24, 74-79,
    226-233, 285-293; reprinted in his Oeuvres, Tome III, Gauthier-Villars, Paris, 1912, 150-181.

\bibitem{Ja} H. Jager, \emph{A simultaneous generalization of the Pad\'{e} table,  I-VI}, Indag. Math. \textbf{26} (1964),
    193-249.

\bibitem{kac}    V. G. Kac and J. W. van de Leur, \emph{The n-component KP hierarchy and representation theory}, J.
    Math. Phys. \textbf{44} (2003) 3245-3293.

\bibitem{arno} A. B. J. Kuijlaars,   \emph{ Multiple orthogonal polynomial ensembles}, \texttt{arXiv:0902.1058v1 [math.PR] } (2009).



\bibitem{macdonald}  I. G. Macdonald,  \emph{Symmetric Functions and Hall Polynomials}, Clarendon Press, Oxford, 1995.


\bibitem{Ma} K. Mahler, \emph{Perfect systems}, Compos. Math. \textbf{19} (1968), 95-166.

\bibitem{manas-martinez-alvarez} M. Ma\~{n}as, L. Mart\'{\i}nez Alonso and C. \'{A}lvarez-Fern\'{a}ndez, \emph{The multicomponent 2D
    Toda hierarchy: discrete flows and string equations}, Inv. Prob. \textbf{25} (2009) 065007 (31 pp).

\bibitem{manas-martinez} M. Ma\~{n}as and L. Mart\'{\i}nez Alonso, \emph{The multicomponent 2D Toda hierarchy: dispersionless
    limit}, Inv. Prob. \textbf{25}  (2009) 115020 (22 pp).


\bibitem{mulase} M. Mulase, \emph{Complete integrability of the Kadomtsev--Petviashvili equation}, Adv. Math.
    \textbf{54} (1984)
 57-66.

\bibitem{Nik}
 E. M. Nikishin,  \emph{On simultaneous Pad\'{e} approximants} Matem. Sb. {\bf 113} (1980), 499--519 (Russian); English
 translation in Math. USSR Sb.  {\bf 41} (1982), 409-425.

\bibitem{orlov} A. Yu. Orlov and  D. M. Scherbin, \emph{Fermionic representation for basic hypergeometric functions
    related to Schur polynomials}, \texttt{arXiv:nlin/0001001v4 [nlin.SI]}.

\bibitem{sato}  M. Sato, \emph{Soliton equations as dynamical systems on infinite dimensional Grassmann manifolds},
    Res. Inst. Math. Sci. Kokyuroku \textbf{439}, (1981) 30-46 .

\bibitem{ShTm}  J. A.  Shohat and J.D. Tamarkin, \emph{The problem of moments},   American Mathematical Society (1943).

\bibitem{simon} B. Simon, \emph{The Christoffel-Darboux Kernel},  Proceedings of Symposia in Pure Mathematics
    \textbf{79}:``Perspectives in Partial Differential Equations, Harmonic Analysis and Applications: A Volume in
    Honor of Vladimir G. Maz'ya's 70th Birthday'',  (2008) 295-336. \texttt{arXiv:0806.1528}

\bibitem{So}  V.N. Sorokin, \emph{On simultaneous approximation of several linear forms}, Vestnik Mosk Univ., Ser
    Matem \textbf{1} (1983) 44-47.

\bibitem{ueno-takasaki} K. Ueno and K. Takasaki, \emph{Toda lattice hierarchy}, in Group Representations and Systems
of Differential Equations, Adv. Stud. Pure Math. \textbf{4} (1984) 1-95.

\bibitem{stanley} R. P. Stanley, \emph{Enumerative combinatorics}, Cambridge University Press, Cambridge (1998).

\bibitem{assche} W. Van Assche, J. S. Geromino and A. B. J. Kuijlaars, \emph{Riemann--Hilbert problems for multiple
    orthogonal polynomials} in: Bustoz et al (eds.), Special Functions 2000: Current Perspectives and Future
    Directions, Kluwer Academic Publishers, Dordrecht, 2001, pp 23-59.





\end{thebibliography}
\end{document}